\documentclass[leqno,11pt]{amsart}
\usepackage[top=1.25in, bottom=1.25in, left=1in, right=1in]{geometry}
\usepackage{amssymb,amsmath,latexsym,amsfonts,amsbsy, amsthm,dsfont}
\usepackage{hyperref}
\usepackage{color}
\usepackage{mathrsfs}
\usepackage{graphicx}
\usepackage{comment}
\usepackage{enumerate}

\definecolor{pku}{RGB}{139,0,18}

\let\pa=\partial
\let\f=\frac

\let\pa=\partial
\let\na=\nabla

\let\al=\alpha
\let\b=\beta

\let\d=\delta
\let\D=\Delta
\let\g=\gamma

\let\lam=\lambda
\let\Lam=\Lambda
\let\om=\omega

\let\va=\varphi

\let\r=\rho

\let\th=\theta

\def\di{\mathrm{div}\,}

\def\osc{\mathrm{osc}\,}

\newcommand{\beq}{\begin{equation}}
\newcommand{\eeq}{\end{equation}}
\newcommand{\beqo}{\begin{equation*}}
\newcommand{\eeqo}{\end{equation*}}
\newcommand{\ben}{\begin{eqnarray}}
\newcommand{\een}{\end{eqnarray}}
\newcommand{\beno}{\begin{eqnarray*}}
\newcommand{\eeno}{\end{eqnarray*}}

\newtheorem{thm}{Theorem}[section]
\newtheorem{lem}{Lemma}[section]
\newtheorem{cor}{Corollary}[section]
\newtheorem{prop}{Proposition}[section]

\theoremstyle{definition}
\newtheorem{definition}{Definition}[section]

\theoremstyle{remark}
\newtheorem{step}{Step}

\newtheorem{rmk}{Remark}[section]

\newcommand{\pv}{\mathrm{p.v.}}

\newcommand{\CC}{\mathcal{C}}
\newcommand{\CE}{\mathcal{E}}

\newcommand{\CR}{\mathcal{R}}

\newcommand{\CF}{\mathcal{F}}
\newcommand{\CT}{\mathcal{T}}
\newcommand{\CP}{\mathcal{P}}

\newcommand{\CH}{\mathcal{H}}

\newcommand{\CN}{\mathcal{N}}

\newcommand{\BR}{\mathbb{R}}
\newcommand{\BZ}{\mathbb{Z}}
\newcommand{\BT}{\mathbb{T}}

\newcommand{\BN}{\mathbb{N}}

\numberwithin{equation}{section}

\begin{document}

\title[Tangential Peskin Problem]{Global Solutions to the Tangential Peskin Problem in 2-D}
\author{Jiajun Tong}
\address{Beijing International Center for Mathematical Research, Peking University, China}
\email{tongj@bicmr.pku.edu.cn}

\date{\today}
\maketitle

\begin{abstract}
We introduce and study the tangential Peskin problem in 2-D, which is a scalar drift-diffusion equation with a nonlocal drift.
It is derived with a new Eulerian perspective from a special setting of the 2-D Peskin problem where an infinitely long and straight 1-D elastic string deforms tangentially in the Stokes flow induced by itself in the plane.
For initial datum in the energy class satisfying natural weak assumptions, we prove existence of its global solutions.
This is considered as a super-critical problem in the existing analysis of the Peskin problem based on Lagrangian formulations.
Regularity and long-time behavior of the constructed solution is established.
Uniqueness of the solution is proved under additional assumptions.
%Such analysis provides useful guidance for studying the original 2-D Peskin problem.
\end{abstract}

\section{Introduction}
\label{sec: intro}
\subsection{Problem formulation and main results}
Consider the following scalar drift-diffusion equation on $\BT := \BR/(2\pi \BZ) = [-\pi,\pi)$,
\beq
\pa_t f
= \CH f \cdot \pa_x f - f  (-\D)^{\f12} f,\quad f(x,0) = f_0(x).
\label{eqn: the main equation}
\eeq
Here $\CH$ denotes the Hilbert transform on $\BT$,
\beq
\CH f(x) : = \f{1}{\pi}\mathrm{p.v.}\int_\BT \f{f(y)}{2\tan \big(\f{x-y}{2}\big)}\,dy,
\label{eqn: Hilbert transform}
\eeq
and
\beq
(-\D)^{\f12} f(x) : = \f{1}{\pi}\mathrm{p.v.}\int_\BT \f{f(x)-f(y)}{4\sin ^2\big(\f{x-y}{2}\big)}\,dy = \CH f'(x).
\label{eqn: fractional Laplacian}
\eeq
In the paper, we want to study positive global solutions $f= f(x,t)$ to \eqref{eqn: the main equation} under rather weak assumptions on the initial data.

In Section \ref{sec: derivation} of the paper, we will show that the drift-diffusion equation \eqref{eqn: the main equation} arises as a scalar model of the 2-D Peskin problem.
In general, the 2-D Peskin problem, also known as the 2-D Stokes immersed boundary problem in the literature,  describes coupled motion of a 2-D Stokes flow and 1-D closed elastic string immersed in it \cite{lin2019solvability,mori2019well,rodenberg20182d,tong2018stokes}.
Its mathematical formulation and related analytical works will be reviewed in Section \ref{sec: 2D Peskin}.
The equation \eqref{eqn: the main equation} stems from a special setting of it, where an infinitely long and straight elastic string deforms  and moves only \emph{tangentially} in a Stokes flow in $\BR^2$ induced by itself.
If we assume the string lies along the $x$-axis, then $f = f(x,t)$ in \eqref{eqn: the main equation} represents how much the string segment at the spatial point $(x,0)$ gets stretched in the horizontal direction (see \eqref{eqn: def of f}).
Thus, we shall call \eqref{eqn: the main equation} the \emph{tangential Peskin problem} in 2-D.
Analysis of such simplified models may improve our understanding of well-posedness and possible blow-up mechanism of the original 2-D Peskin problem. % with general initial string configurations.
See Section \ref{sec: derivation of the model} for more detailed discussions.

Suppose $f$ is a positive smooth solution to \eqref{eqn: the main equation}.
It is straightforward to find that $F:=\f1{f}$ verifies a conservation law
\beq
\pa_t F = \pa_x\big(\CH f \cdot F\big),\quad F(x,0) = F_0(x):=\big(f_0(x)\big)^{-1}.
\label{eqn: equation for big F}
\eeq
This leads to the following definition of (positive) weak solutions to \eqref{eqn: the main equation}.

\begin{definition}[Weak solution]
\label{def: weak solution}
Let $f_0\in L^1(\BT)$ such that $f_0>0$ almost everywhere, and $F_0:=\f{1}{f_0}\in L^1(\BT)$.
For $T\in [0,+\infty]$, we say that $f =f(x,t)$ is a (positive) weak solution to \eqref{eqn: the main equation} on $\BT\times [0,T)$, if the followings hold.
\begin{enumerate}[(a)]
\item $f>0$ a.e.\;in $\BT\times [0,T)$, and $F\in L^\infty([0,T); L^1(\BT))$, where $F =\f{1}{f}$ a.e.\;in $\BT\times [0,T)$;
%\textcolor{red}{Any continuity at the initial time to make sense of the initial data?}
\item $\CH f\cdot F\in L^1_{loc}(\BT\times [0,T))$;
\item For any $\va \in C_0^\infty(\BT\times [0,T))$,
\beq
\int_{\BT} \va(x,0) F_0(x)\,dx + \int_{\BT\times [0,T)} \pa_t \va \cdot  F\,dx\,dt
= \int_{\BT\times [0,T)} \pa_x \va \cdot \CH f \cdot F\,dx\,dt.
\label{eqn: weak formulation}
\eeq
\end{enumerate}
In this case, we also say $F$ is a (positive) weak solution to \eqref{eqn: equation for big F}.

If $T=+\infty$, the weak solution is said to be global.
\end{definition}

\begin{rmk}
Given \eqref{eqn: the main equation}, it is tempting to define the weak solutions by requiring
\[
\begin{split}
&\;\int_{\BT} \va(x,0) f_0(x)\,dx + \int_{\BT\times [0,T)} \pa_t \va \cdot  f\,dx\,dt\\
= &\; \int_{\BT\times [0,T)} \pa_x \va \cdot \CH f \cdot f \,dx\,dt
+
2 \int_{\BT\times [0,T)} (-\D)^{\f14}(\va f)(-\D)^{\f14} f \,dx\,dt
\end{split}
\]
to hold for any $\va \in C_0^\infty(\BT\times [0,T))$.
Although this formulation imposes no regularity assumptions on $1/f$, it does not effectively relocate the derivative to the test function, so we choose not to work with it.
\end{rmk}

We introduce the $H^s$-semi-norms for $s>0$,
\[
\|f\|_{\dot{H}^s(\BT)}^2: = \sum_{k\in \BZ} |k|^{2s} |\hat{f}_k|^2,
\]
where the Fourier transform of a function $f$ in $L^1(\BT)$ is defined by
\beq
\CF(f)_k = \hat{f}_k := \f{1}{2\pi}\int_\BT f(x)e^{-ikx}\,dx.
\label{eqn: Fourier transform}
\eeq

Our first result is on the existence of a global weak solution with nice properties.
\begin{thm}
\label{thm: main thm}
Suppose $f_0\in L^1(\BT)$ satisfies $f_0>0$ almost everywhere, with $F_0:=\f{1}{f_0}\in L^1(\BT)$.
Then \eqref{eqn: the main equation} admits a global weak solution $f = f(x,t)$ in the sense of Definition \ref{def: weak solution}.
It has the following properties.
\begin{enumerate}

\item \label{property: energy inequality in the main thm}
    (Energy relation) $f\in L^\infty([0,+\infty); L^1 (\BT))\cap L^2([0,+\infty); \dot{H}^{1/2}(\BT))$.
    For any $t\geq 0$,
    \beqo
    \|f(\cdot,t)\|_{L^1} + 2\int_0^t \|f(\cdot,\tau)\|_{\dot{H}^{1/2} }^2\,d\tau \leq \|f_0\|_{L^1},
    %\label{eqn: energy inequality}
    \eeqo
    while for any $0<t_1\leq t_2$,
    \beqo
    \|f(\cdot,t_2)\|_{L^1} + 2\int_{t_1}^{t_2} \|f(\cdot,\tau)\|_{\dot{H}^{1/2} }^2\,d\tau = \|f(\cdot,t_1)\|_{L^1}.
    \eeqo

\item \label{property: positivity and smoothness}
    (Instant positivity and smoothness) For any $t_0>0$, the solution $f$ is $C^\infty$ and positive in $\BT\times [t_0,+\infty)$, so $f$ is a positive strong solution to \eqref{eqn: the main equation} for all positive times (cf.\;Definition \ref{def: strong solution}).
    To be more precise, define $F = \f1f$ pointwise on $\BT\times (0,+\infty)$.
    Then $\|f(\cdot,t)\|_{L^\infty}$ and $\|F(\cdot,t)\|_{L^\infty}$ are non-increasing in $t\in (0,+\infty)$.
    They satisfy that
    \beq
\|f(\cdot,t)\|_{L^\infty(\BT)}\leq Ct^{-\f12} \|f_0\|_{L^1}^{\f12}\quad \forall\, t\in (0,\|f_0\|_{L^1}^{-1}],
\label{eqn: upper bound for f}
\eeq
where $C$ is a universal constant, and that
\beq
\|F(\cdot,t)\|_{L^\infty}\leq \f18\|F_0\|_{L^1}
\exp\left[\coth\left(\f{4}{\pi}\|F_0\|_{L^1}^{-1} t\right)\right]\quad \forall\, t>0.
\label{eqn: upper bound for F}
\eeq
Moreover, for any $\alpha\in (0,\f15)$, $\|f(\cdot,t)\|_{\dot{C}^{\al}}\in L^1([0,+\infty))$, satisfying that, for $t\in (0, \|f_0\|_{L^1}^{-1}]$,
\[
\int_0^t \|f(\cdot,\tau)\|_{\dot{C}^{\al}}\,d\tau
\leq C_\al t^{\f{1-5\al}{2}} \big(\|f_0\|_{L^1}-2\pi f_\infty\big)^{\f{1-2\al}{2}}
\|f_0\|_{L^1}^{\f{\al}{2}} \|F_0\|_{L^1}^{2\al},
\]
where $C_\al>0$ is a universal constant depending on $\al$, and where $f_\infty$ is defined in \eqref{eqn: equilibrium} below.

\item (Conservation law)
    $\|F(\cdot,t)\|_{L^1} = \|F_0\|_{L^1}$ for all $t\geq 0$, and the map $t\mapsto F(\cdot,t)/\|F_0\|_{L^1}$ is continuous from $[0,+\infty)$ to $\CP_1(\BT)$.
    Here $\CP_1(\BT)$ denotes the space of probability densities on $\BT$ equipped with $1$-Wasserstein distance $W_1$ %.
    %The $W_1$-distance is defined under the natural metric on $\BT$
    (see e.g.\;\cite[Chapter 6]{villani2009optimal}).

\item \label{property: long-time convergence}
    (Long-time exponential convergence)
    Define
\beq
f_\infty : = \left(\f{1}{2\pi}\int_\BT \f{1}{f_0(x)}\,dx\right)^{-1}.
\label{eqn: equilibrium}
\eeq
Then for arbitrary $j,k\in \BN$, $\|\pa_t^j \pa_x^k (f(\cdot,t)-f_\infty)\|_{L^2(\BT)}$ is finite for any $t>0$, and it decays to $0$ exponentially as $t\to +\infty$.
It admits an upper bound that only depends on $t$, $j$, $k$, $\|f_0\|_{L^1}$, and $\|F_0\|_{L^1}$.
See Proposition \ref{prop: boundedness and decay of H^k norms of h}.

\item \label{property: large-time analyticity}
    (Large-time analyticity and sharp decay) There exists $T_*>0$ depending only on $\|f_0\|_{L^1}$ and $\|F_0\|_{L^1}$, such that $f(x,t)$ is analytic in the space-time domain $\BT\times (T_*,+\infty)$.
    In addition, for any $t>0$,
    \beqo
    \|f(\cdot,t+T_*)\|_{\dot{\CF}_{\nu(t)}^{0,1}} \leq Cf_\infty,
    \eeqo
    with
    \[
    \nu(t)\geq \f12 \ln \left[1 + \exp(2f_\infty t)\right] - \f12 \ln 2,
    \]
    and $C>0$ being a universal constant.
    Here the $\dot{\CF}^{0,1}_{\nu(t)}$ semi-norms for Wiener-algebra-type spaces will be defined in Section \ref{sec: analyticity and sharp decay estimate}.
\end{enumerate}
\end{thm}
\begin{rmk}\label{rmk: energy equality}
We remark that $L^1(\BT)$ is considered to be the energy class for the initial data $f_0$ (see Remark \ref{rmk: energy class}), while the condition $F_0\in L^1(\BT)$ is natural given \eqref{eqn: equation for big F} (also see Remark \ref{rmk: interpretation of f and F}).
If in addition $F_0\in L^p(\BT)$ for some $p>1$, the energy equality holds up to $t = 0$ for the constructed solution, i.e., for any $t\geq 0$,
\beqo
\|f(\cdot,t)\|_{L^1} + 2\int_0^t \|f(\cdot,\tau)\|_{\dot{H}^{1/2} }^2\,d\tau = \|f_0\|_{L^1}.
\eeqo
\end{rmk}

Note that we do not claim the above properties for all weak solutions, but only for the constructed one.
Regularity of general weak solutions to \eqref{eqn: the main equation} is wildly open.

It will be clear in Section \ref{sec: derivation} that the equation \eqref{eqn: the main equation} and Theorem \ref{thm: main thm} are formulated with a Eulerian perspective.
In Corollary \ref{cor: rephrasing in Lagrangian} in Section \ref{sec: derivation of the model}, we will recast these results in the corresponding Lagrangian coordinate, which is more commonly used in the analysis of the Peskin problem.
Roughly speaking, in the Lagrangian framework, for arbitrary initial data in the energy class satisfying some weak assumptions but no restriction on its size, we are able to well define a global solution to the 2-D tangential Peskin problem with nice properties.
Let us note that, in the existing studies on the Peskin problem with the Lagrangian perspectives, this is considered as a super-critical problem. %, but here we solve it by taking the novel Eulerian point of view.
See more discussions in Section \ref{sec: derivation}.

We also would like to highlight the following very interesting property of \eqref{eqn: the main equation}, which is helpful
when proving Theorem \ref{thm: main thm}.
It states that band-limited initial datum give rise to band-limited solutions, with the band width being uniformly bounded in time.
This allows us to easily construct global solutions from such initial datum.
This property is found by taking the Fourier transform of \eqref{eqn: the main equation}.
We will justify it in Section \ref{sec: band-limited initial data}.

\begin{prop}\label{prop: band-limited initial data intro}
Suppose $f_0>0$ is band-limited, namely, there exists some $K\in\BN$, such that $\CF(f_0)_k = 0$ for all $|k|>K$.
Then \eqref{eqn: the main equation} has a global strong solution $f = f(x,t)$ that is also positive and band-limited, such that $\hat{f}_k(t) = 0$ for all $|k|>K$ and $t\geq 0$.
Such a solution is unique, and it can be determined by solving a finite family of ODEs of its Fourier coefficients.
In fact, we let $\hat{f}_k(t)\equiv 0$ for all $|k|>K$.
Denote $\bar{f}(t):= \CF(f(\cdot,t))_0 = \f{1}{2\pi}\int_\BT f(x,t)\,dx$.
Then for $k\in [0,K]$, $\hat{f}_k(t) = \CF(f(\cdot,t))_k$ solves
\[
\f{d}{dt} \hat{f}_k(t)
=
- k \bar{f}(t)\hat{f}_{k}(t)
- \sum_{j>0} 2(k+2j) \hat{f}_{k+j}(t)\overline{\hat{f}_j(t)},
\quad \hat{f}_k(0) = \CF(f_0)_k,
\]
where $\hat{f}_0(t)$ should be understood as $\bar{f}(t)$.
Finally, let $\hat{f}_k(t) = \overline{\hat{f}_{-k}(t)}$ for $k\in [-K,0)$.
\end{prop}

Uniqueness of the weak solution to \eqref{eqn: the main equation} is also an open question given the very weak assumptions on the initial data and the solution.
To achieve positive results on the uniqueness, we introduce a new notion of weak solutions.
We inherit the notations from Definition \ref{def: weak solution}.

\begin{definition}[Dissipative weak solution]
\label{def: dissipative weak solution}
Under the additional assumption $\ln f_0\in H^{\f12}(\BT)$,
a weak solution $f =f(x,t)$ to \eqref{eqn: the main equation} on $\BT\times [0,T)$ (in the sense of Definition \ref{def: weak solution}) is said to be a dissipative weak solution, if it additionally verifies
\begin{enumerate}[(a)]

\item $\ln f\in L^\infty([0,T); H^{1/2}(\BT))$ and $\pa_x\sqrt{f}\in L^2(\BT\times [0,T))$, satisfying that (cf.\;Lemma \ref{lem: H 1/2 estimate for h})
\beqo
\|\ln f(\cdot,t)\|_{\dot{H}^{\f12}}^2
+ 4\int_0^t \int_\BT  \big(\pa_x \sqrt{f} \big)^2\,dx\,dt
\leq
\|\ln f_0\|_{\dot{H}^{\f12}}^2
\eeqo
for any $t\in [0,T)$.

\item
    For any convex function $\Phi\in C^1_{loc}((0,+\infty))$ such that $\Phi(F_0(x))\in L^1(\BT)$, it holds
    \beq
    \int_\BT \Phi(F(x,t))\,dx + \int_0^t \int_\BT \big(\Phi(F)-F\Phi'(F)\big)(-\D)^{\f12}f\,dx\,dt
    \leq \int_\BT \Phi(F_0(x))\,dx
    \label{eqn: entropy type estimate}
    \eeq
    for all $t\in [0,T)$ (cf.\;Remark \ref{rmk: entropy estimate}).
\end{enumerate}
\end{definition}

Then we state uniqueness of the dissipative weak solutions under additional assumptions.

\begin{thm}\label{thm: uniqueness}
Suppose $\ln f_0\in H^{\f12}(\BT)$.
Further assume $f_0\in L^\infty(\BT)$ and $f_0\geq \lam$ almost everywhere with some $\lam>0$.
Then \eqref{eqn: the main equation} admits a unique dissipative weak solution.

In particular, the solution coincides with the one constructed in Theorem \ref{thm: main thm}, and thus it satisfies all the properties there.
\end{thm}

\subsection{Related studies}
\label{sec: related studies}
The equation \eqref{eqn: the main equation} is derived from a special setting of the 2-D Peskin problem, but we choose to review related works in that direction in Section \ref{sec: 2D Peskin}.
Beyond that, \eqref{eqn: the main equation} is reminiscent of many classic PDEs arising in fluid dynamics and other subjects.
We list a few, but we note that the list of equations and literature below is by no means exhaustive.

\begin{enumerate}[(I)]
  \item
In \cite{cordoba2005formation}, C\'{o}rdoba, C\'{o}rdoba, and Fontelos (CCF) studied
\[
\pa_t \th - \CH \th\cdot \pa_x \th = -\nu (-\D)^{\f{\alpha}{2}}\th
\]
on $\BR$.
They showed finite-time blow-up for the inviscid case $\nu = 0$, and yet well-posedness when $\nu>0$ and $\alpha\geq 1$, with smallness of initial data assumed for the critical case $\alpha = 1$.
%Those well-posedness results were improved in \cite{}.
See \cite{dong2008well,li2008blowup,silvestre2016transport,lazar2016nonlocal} for related results.

When $\nu>0$, the CCF model is known as a 1-D analogue of the dissipative quasi-geostrophic equation in $\BR^2$
\[
\pa_t \th + u\cdot \na \th =- (-\D)^{\f{\al}{2}}\th,\quad u = (-\CR_2\th,\CR_1\th).
\]
Here $\CR_1$ and $\CR_2$ denote the Riesz transforms on $\BR^2$.
The case $\alpha = 1$ is again critical, for which global well-posedness results are established in e.g.\;\cite{caffarelli2010drift,constantin2001critical,kiselev2007global}; also see \cite{dong2010dissipative}.

Compared with the CCF model with $\nu>0$ and $\al = 1$, our equation \eqref{eqn: the main equation} features a fractional diffusion with variable coefficient.
The diffusion is enhanced when $f$ is large, while it becomes degenerate when $f$ is almost zero.
The latter turns out to be one of the main difficulties in the analysis of \eqref{eqn: the main equation}.
%In spite of such degeneracy, we can show that, under the assumptions on the initial data in Theorem \ref{thm: main thm}, the solution enjoys several a priori smoothing estimates;
See Lemma \ref{lem: lower bound} and Section \ref{sec: estimates for h} for how we handle such possible degeneracy.

\item The equation
\beq
\pa_t \rho + \pa_x\big( \rho \CH \rho\big) = 0,
\label{eqn: classic equation with a different sign}
\eeq
or equivalently (at least for nice solutions $\rho = \r(x,t)$),
\beq
\pa_t \rho = -\pa_x \rho \cdot \CH \rho - \rho (-\D)^{\f12}\rho,
\label{eqn: classic equation with a different sign expanded}
\eeq
has been studied extensively in several different contexts, such as fluid dynamics \cite{baker1996analytic,chae2005finite,castro2008global}, and dislocation theory \cite{head1972dislocation1, head1972dislocation2, head1972dislocation3, deslippe2004dynamic, biler2010nonlinear}.
Interestingly, \eqref{eqn: classic equation with a different sign} has similarities with both \eqref{eqn: the main equation} and \eqref{eqn: equation for big F}.
Indeed, \eqref{eqn: classic equation with a different sign expanded} differs from \eqref{eqn: the main equation} only in the sign of the drift term.
On the other hand, both \eqref{eqn: equation for big F} and \eqref{eqn: classic equation with a different sign} are conservation laws;
the transporting velocity in \eqref{eqn: classic equation with a different sign} is $\CH \r$, while it is $-\CH (\f{1}{F})$ in \eqref{eqn: equation for big F}.
It is not immediately clear how such differences lead to distinct features of these equations as well as different difficulties in their analysis, but apparently, it is less straightforward to propose and study the weak formulation of \eqref{eqn: the main equation} than that of \eqref{eqn: classic equation with a different sign}.

A generalization of \eqref{eqn: classic equation with a different sign} reads
\[
\pa_t \rho + \th \pa_x\big( \rho \CH \rho\big)+(1-\th)\CH \r \cdot \pa_x \r = 0
\]
where $\th\in [0,1]$ is a constant \cite{morlet1998further}.
Formally, \eqref{eqn: the main equation} corresponds to the case $\th = -1$, up to a change of variable $\r = -f$, which is not covered by existing works to the best of our knowledge.

Another generalization of \eqref{eqn: classic equation with a different sign} %that shows up in many different problems
is %an equation of the form
\[
\pa_t \r - \di\big(\r \na (-\D)^{-s}\r\big) = 0.
\]
The case $s = 1$ arises in the hydrodynamics of vortices in the
Ginzburg-Landau theory \cite{weinan1994dynamics}; its well-posedness was proved in \cite{lin2000hydrodynamic}.
When $s\in (0,1)$, it was introduced as a nonlinear porous medium equation
with fractional potential pressure \cite{caffarelli2011nonlinear,caffarelli2011asymptotic,caffarelli2013regularity,caffarelli2016regularity}. Its well-posedness, long-time asymptotics, regularity of the solutions, as well as other properties has been investigated.
If we set the equation in 1-D and denote $\al = 2(1-s)$, then it can be written as
\[
\pa_t \rho + \big(\pa_x^{-1}(-\D)^{\f{\al}{2}}\rho\big) \pa_x \rho  = -\rho (-\D)^{\f{\al}{2}} \rho.
\]
This turns out to be a special case of the 1-D Euler alignment system %;
%see e.g.~%Section 1 and in particular Eqn.~(1.37) of
 for studying flocking dynamics \cite{do2018global, %}, and also \cite{
shvydkoy2017eulerian,shvydkoy2018eulerian,tan2019singularity}.
Clearly, \eqref{eqn: classic equation with a different sign} corresponds to the case $s = \f12$, or equivalently, $\alpha = 1$ in 1-D.

\item It was derived in \cite{steinerberger2019nonlocal} that the dynamics of real roots of a high-degree polynomial under continuous differentiation can be described by
\[
\pa_t u + \pa_x \left(\arctan \left(\f{\CH u}{u}\right)\right) = 0,
\]
or equivalently (at least for nice solutions $u = u(x,t)$),
\beq
\pa_t u = \f{\CH u\cdot \pa_x u - u(-\D)^{\f12} u}{u^2+(\CH u )^2}.
\label{eqn: eqn for dynamics of roots}
\eeq
Also see \cite{shlyakhtenko2020fractional,steinerberger2020free}.
Recently, its well-posedness was studied in \cite{alazard2022dynamics,granero2020nonlocal,kiselev2022global,kiselev2020flow}.
Although \eqref{eqn: eqn for dynamics of roots} looks similar to \eqref{eqn: the main equation} if one ignores the denominator of the right-hand side, its solutions can behave differently from those to \eqref{eqn: the main equation} because of very different form of nonlinearity and degeneracy.
Even when $u$ is sufficiently close to a constant $\bar{u}$, \cite{granero2020nonlocal} formally derives that $u-\bar{u}$ should solve
\[
\pa_t v  = -(-\D)^{\f12}v + \pa_x (v\CH v)
\]
up to suitable scaling factors and higher-order errors.
This equation becomes \eqref{eqn: classic equation with a different sign} formally after a change of variable $\rho = 1-v$.
%See e.g.\;\cite{chae2005finite} for the analysis of this equation.

\item Letting $\om := -f$, we may rewrite \eqref{eqn: the main equation} as
\beq
\pa_t \om +u \pa_x \om = \om \pa_x u,\quad
u = \CH \om.
\label{eqn: recast as De Gregorio}
\eeq
Then it takes a similar form as the De Gregorio model \cite{de1990one,de1996partial},
\beq
\pa_t \om +u \pa_x \om = \om \pa_x u,\quad
u_x = \CH \om.
\label{eqn: De Gregorio}
\eeq
The latter is known as a 1-D model for the vorticity equation of the 3-D Euler equation, playing an important role in understanding possible blow-ups in the 3-D Euler equation as well as many other PDEs in fluid dynamics.
Its analytic properties are quite different from \eqref{eqn: the main equation}; see e.g.\;\cite{castro2010infinite,chen2021finite,elgindi2020effects,jia2019gregorio,lei2020constantin,okamoto2008generalization,Huang2022OnSF} and the references therein.
Nevertheless, we refer the readers to Lemma \ref{lem: H 1 estimate for sqrt f} below, which exhibits an interesting connection of these two models due to their similar algebraic forms.
\end{enumerate}

\subsection{Organization of the paper}
In Section \ref{sec: derivation}, we review the Peskin problem in 2-D and derive \eqref{eqn: the main equation} formally in a special setting of it with a new Eulerian perspective.
We highlight Corollary \ref{cor: rephrasing in Lagrangian} in Section \ref{sec: derivation of the model}, which recasts the results in Theorem \ref{thm: main thm} in the Lagrangian formulation.
Sections \ref{sec: a priori estimates}-\ref{sec: smoothing and decay} establish a series of a priori estimates for nice solutions.
We start with some basic estimates for \eqref{eqn: the main equation} in Section \ref{sec: a priori estimates}, and then prove in Section \ref{sec: positivity and boundedness} that smooth solutions to \eqref{eqn: the main equation} admit finite and positive a priori upper and lower bounds at positive times. %, where the bounds only depend on $\|f_0\|_{L^1}$ and $\|F_0\|_{L^1}$.
In order to make sense of the weak formulation, we prove in Section \ref{sec: estimates for h} that certain H\"{o}lder norms of $f$ are time-integrable near $t=0$. %, we can show that certain H\"{o}lder norm of $f$ is time-integrable near $t=0$.
Section \ref{sec: global higher-order estimates} further shows that the solutions would become smooth at all positive times, and they converge to a constant equilibrium exponentially as $t\to +\infty$.
Section \ref{sec: continuity at initial time} addresses the issue of the time continuity of the solutions at $t = 0$.
In Section \ref{sec: global well-posedness}, we prove
existence of the desired global weak solutions, first for band-limited and positive initial data in Section \ref{sec: band-limited initial data}, and then for general initial data in Section \ref{sec: existence for general initial data}.
Corollary \ref{cor: rephrasing in Lagrangian} will be justified there as well.
Section \ref{sec: uniqueness} proves Theorem \ref{thm: uniqueness} on the uniqueness.
We show in Section \ref{sec: analyticity and sharp decay estimate} that the solution becomes analytic when $t$ is sufficiently large, and we also establishes sharp rate of exponential decay of the solution.
Finally, in Appendix \ref{sec: H-1/2 estimate for F}, we discuss an $H^{-\f12}$-estimate for $F$ that is of independent interest.

\subsection*{Acknowledgement}
The author is supported by the National Key R\&D Program of China under the grant 2021YFA1001500.
The author would like to thank Dongyi Wei, Zhenfu Wang, De Huang, and Fanghua Lin for helpful discussions.

\section{The Tangential Peskin Problem in 2-D}
\label{sec: derivation}

\subsection{The 2-D Peskin problem}
\label{sec: 2D Peskin}
The 2-D Peskin problem, also known as the Stokes immersed boundary problem in 2-D, describes a 1-D closed elastic string immersed and moving in a 2-D Stokes flow induced by itself.
Let the string be parameterized by $X = X(s,t)$, where $s\in \BT$ is the Lagrangian coordinate (as opposed to the arclength parameter).
Then the 2-D Peskin problem is given by
\begin{align}
&\;-\Delta u+ \nabla p = \int_{\mathbb{T}}F_X(s,t)\delta(x-X(s,t))\,ds,\label{eqn: Stokes equation}\\
&\;\mathrm{div}\, u = 0,\quad |u|,|p|\rightarrow 0\mbox{ as }|x|\rightarrow \infty,\label{eqn: incompressible condition}\\
%&\;f(x,t) = ,\label{eqn: elastic forcing}\\
&\;\frac{\partial X}{\partial t}(s,t) = u(X(s,t),t),\quad%\label{eqn: motion of the membrane}\\
X(s,0) = X_0(s).\label{eqn: motion of the membrane and initial configuration}
\end{align}
Here $u = u(x,t)$ and $p = p(x,t)$ denote the flow field and pressure in $\BR^2$, respectively.
The right-hand side of \eqref{eqn: Stokes equation}
%$f(x,t)$
is the singular elastic force applied to the fluid that is only supported along the string. %, defined by \eqref{eqn: elastic forcing}.
$F_X$ is the elastic force density in the Lagrangian coordinate, determined by the string configuration $X$.
In general, it is given by \cite{peskin2002immersed}
\begin{equation*}
F_X(s,t) = \partial_s\left(\mathcal{T}(|X'(s,t)|, s,t)\frac{X'(s,t)}{|X'(s,t)|}\right).
%\label{eqn: Lagrangian representation of the elastic force}
\end{equation*}
Here and in what follows, we write $\partial_s X(s,t)$ as $X'(s,t)$ for brevity.
$\mathcal{T}$ denotes the tension along the string.
In the simple case of Hookean elasticity, %$\CE(p) = \f{k_0}2 p^2$,
%Then
$\mathcal{T}(|X'(s,t)|,s,t) = k_0 |X'(s,t)|$, where $k_0>0$ is the Hooke's constant, and thus $F_X(s,t) = k_0X''(s,t)$.

The Peskin problem is closely related to the well-known numerical immersed boundary method \cite{peskin1972flow,peskin2002immersed}.
As a model problem of fluid-structure interaction, it has been extensively studied and applied in numerical analysis for decades.
Nonetheless, analytical studies of it started only recently.
In the independent works \cite{lin2019solvability} and \cite{mori2019well},
the authors first studied well-posedness of the problem with $F_{X} = X''(s,t)$.
Using the 2-D Stokeslet
\beq
G_{ij}(x) = \f{1}{4 \pi}\left(-\ln |x|\d_{ij} + \f{x_i x_j}{|x|^2}\right),\quad i,j = 1,2,
\label{eqn: Stokeslet in 2-D}
\eeq
the problem was reformulated into a contour dynamic equation in the Lagrangian coordinate
\beq
%\begin{split}
\pa_t X(s,t) = %&\;
\int_\BT G(X(s,t)-X(s',t)) X''(s',t)\,ds'.
\label{eqn: contour dynamic equation}
\eeq
%in terms of the string dynamics using  only (see \eqref{eqn: contour dynamic equation} below),
%We omitted the time dependence for conciseness.
Based on this, well-posedness results of the string evolution were proved in \cite{lin2019solvability} and \cite{mori2019well} in subcritical spaces $H^{\f52}(\BT)$ and $C^{1,\alpha}(\BT)$, respectively, under the so called \emph{well-stretched condition} on the initial configuration ---
a string configuration $Y(s)$ is said to satisfy the well-stretched condition  with constant $\lam>0$, if for any $s_1,s_2\in \BT$,
\beq
\big|Y(s_1)-Y(s_2)\big|\geq \lam |s_1-s_2|.
\label{eqn: well-stretched condition}
\eeq
It is worthwhile to mention that \cite{mori2019well} also established a blow-up criterion, stating that if a solution fails at a finite time, then either $\|X(\cdot,t)\|_{\dot{C}^{1,\alpha}}$ goes to infinity, or the best well-stretched constant shrinks to zero.

Since then, many efforts have been made to establish well-posedness under the scaling-critical regularity --- in this problem, it refers to $W^{1,\infty}(\BT)$ and other regularity classes with the same scaling.
Garc\'{i}a-Ju\'{a}rez, Mori, and Strain considered the case where the fluids in the interior and the exterior of the string can have different viscosities, and they showed global well-posedness for initial data of medium size in the critical Wiener space $\CF^{1,1}(\BT)$ \cite{garcia2020peskin}.
In \cite{gancedo2020global}, Gancedo, Granero-Belinch\'{o}n, and Scrobogna introduced a toy scalar model that captures the string motion in the normal direction only, for which they proved global well-posedness for small initial data in the critical Lipschitz class $W^{1,\infty}(\BT)$.
Recently, Chen and Nguyen established well-posedness of the original 2-D Peskin problem in the critical space $\dot{B}_{\infty,\infty}^{1}(\BT)$ \cite{chen2021peskin}.
However, in the Hookean case, the total energy of the system %\eqref{eqn: Stokes equation}-\eqref{eqn: motion of the membrane and initial configuration} (or \eqref{eqn: contour dynamic equation})
is $\f12 \|X(\cdot,t)\|_{\dot{H}^1}^2$.
Hence, the Cauchy problem with the initial data belonging to the energy class is considered to be super-critical.

In addition to these studies, the author introduced a regularized Peskin problem in \cite{tong2021regularized}, and proved its well-posedness and its convergence to the original singular problem as the regularization diminishes.
Aiming at handling general elasticity laws beyond the Hookean elasticity, \cite{rodenberg20182d} and the very recent work \cite{cameron2021critical} extended the analysis to a fully nonlinear Peskin problem.
Readers are also referred to \cite{li2021stability} for a study on the case where the elastic string has both stretching and bending energy.

\subsection{Derivation of the tangential Peskin problem in 2-D}
\label{sec: derivation of the model}
In this subsection, we will present a formal derivation of the tangential Peskin problem in 2-D, without paying too much attention to regularity and integrability issues.
It is not the main point of this paper to rigorously justify the derivation, although it clearly holds for nice solutions.

Consider an infinitely long and straight Hookean elastic string lying along the $x$-axis, which admits tangential deformation in the horizontal direction only.
By \eqref{eqn: contour dynamic equation}, this feature gets preserved for all time in the flow induced by itself.
%This gives an exact model in 1-D regarding tangential displacement only.
With abuse of notations, we still use $X(s,t)$ to denote the actual physical position (i.e., the $x$-coordinate) of the string point with the Lagrangian label $s\in \BR$; note that $X(\cdot,t)$ is now a scalar function.
Assume $X: s\mapsto X(s,t)$ to be suitably smooth and strictly increasing, so that it is a bijection from $\BR$ to $\BR$.

Under the assumption of no transversal deformation, \eqref{eqn: contour dynamic equation} reduces to (thanks to \eqref{eqn: Stokeslet in 2-D})
\beq
\pa_t X(s,t) =
\f{1}{4 \pi} \int_\BR \left(-\ln |X(s,t)-X(s',t)|  + 1\right) X''(s',t)\,ds'.
\label{eqn: contour dynamic equation 1D prelim}
\eeq
By integration by parts,
\beq
\pa_t X(s)
=  -\f{1}{4\pi} \mathrm{p.v.}\int_\BR \f{ X'(s')^2}{X(s)-X(s')} \,ds'
= -\f{1}{4\pi} \mathrm{p.v.}\int_\BR \f{X'(s')}{X(s)-X(s')} \,dX(s').
\label{eqn: contour dynamic equation 1D}
\eeq
Here we omitted the $t$-dependence for conciseness.

We define $f = f(x,t)$ such that
\beq
f(X(s,t),t) = X'(s,t). % whenever $x = X(s)$.
\label{eqn: def of f}
\eeq
Physically, $f(x,t)$ represents the extent of local stretching of the elastic string at the spatial position $x$.
Then \eqref{eqn: contour dynamic equation 1D} becomes
\beq
\begin{split}
\pa_t X(s,t)
= &\; -\f{1}{4\pi} \mathrm{p.v.}\int_\BR \f{f(X(s',t),t)}{X(s,t)-X(s',t)} \,dX(s',t)\\
= &\; -\f{1}{4\pi} \mathrm{p.v.}\int_\BR \f{f(y,t)}{X(s,t)-y} \,dy = -\f14 \CH f(\cdot,t)\big|_{x = X(s,t)}.
\end{split}
\label{eqn: transporting velocity for particles}
\eeq
Here $\CH$ denotes the Hilbert transform on $\BR$.
Differentiating \eqref{eqn: transporting velocity for particles} with respect to $s$ yields
\beqo
\pa_t \big(X'(s,t)\big)
= -\f14 (-\D)^{\f12} f(\cdot,t)\big|_{x = X(s,t)}\cdot X'(s,t).
\eeqo
Using \eqref{eqn: def of f}, this can be further recast as
\beqo
\pa_t f(X(s,t),t) + \pa_t X(s,t)\cdot \pa_x f(X(s,t),t)
= -f(X(s,t),t)\cdot \f14(-\D)^{\f12}f(\cdot,t)\big|_{x=X(s,t)}.
\eeqo
Using \eqref{eqn: transporting velocity for particles} again on the left-hand side and letting $x = X(s,t)$, we obtain a scalar equation on $\BR$ for $f = f(x,t)$ %that will be studied in the rest of the paper
\beq
\pa_t f
= \f{1}{4}\left( \CH f \cdot \pa_x f - f (-\D)^{\f12} f \right).
\label{eqn: equation for f}
\eeq
Now assuming $f(x,t)$ to be spatially $2\pi$-periodic for all time, and discarding the coefficient $\f14$ by a change of the time variable, we obtain \eqref{eqn: the main equation} on $\BT$.
Note that, in this periodic setting, the Hilbert transform and the fractional Laplacian in \eqref{eqn: the main equation} should be re-interpreted as those on $\BT$ (see \eqref{eqn: Hilbert transform} and \eqref{eqn: fractional Laplacian} for their definitions).

We introduce this special setting of the 2-D Peskin problem, mainly motivated by the interest of understanding global behavior of the 2-D Peskin problem with general initial data, which is a challenging problem.
Given the blow-up criterion in \cite{mori2019well} mentioned above, %states that solutions to the 2-D Peskin problem develop finite-time singularity only when $\|X(\cdot,t)\|_{\dot{C}^{1,\alpha}}$ goes to infinity, or when the best well-stretched constant shrinks to zero.
it is natural to ask whether finite-time singularity would occur in some typical scenarios.
For example, if a local string segment is smooth at the initial time, and if it is known to ``never gets close to other parts of the string", one may ask whether or not it would spontaneously develop finite-time singularity, such as loss of regularity and aggregation of Lagrangian points, etc.
The tangential Peskin problem proposed here can be the first step in this direction.

We state a few remarks on \eqref{eqn: equation for f} and the above derivation.

\begin{rmk}
\eqref{eqn: equation for f} provides the exact string dynamics in the special setting considered here, as no approximation is made in the derivation.
This is in contrast to the scalar model in \cite{gancedo2020global}, where the tangential deformation of the string should have made a difference to the evolution but gets ignored purposefully.

%To conclude this section, let us make an additional remark on this new tangential model \eqref{eqn: contour dynamic equation 1D} (or equivalently \eqref{eqn: equation for f}).

Although the tangential problem \eqref{eqn: equation for f} (or equivalently, \eqref{eqn: contour dynamic equation 1D}) is derived in a simplified geometric setting, different from that of the original 2-D Peskin problem \eqref{eqn: contour dynamic equation} where the string is a closed curve,
%
%the case of an infinitely long and straight elastic string and it
%may be treated as a toy model,
we will reveal a surprising connection between the two problems in a forthcoming paper \cite{Tong2023Geometric}.
We prove that, in the original 2-D Peskin problem, if the closed elastic string initially exhibits a perfectly circular shape, then it should remain to be a circle of the same radius for all time, while the center of the circle can move.
After normalizing the position of the circle center, we find that the tangential deformation along the circle is described \emph{exactly} by the tangential model.
This allows us to fully characterize the global dynamics of such circular configurations under rather weak assumptions (also see Corollary \ref{cor: rephrasing in Lagrangian} below).
Beyond that, analysis in this paper also provides useful guidance for studying the original 2-D Peskin problem there with more general initial data.
\end{rmk}

\begin{rmk}\label{rmk: interpretation of f and F}
Given the physical interpretation of $f$, $F(x,t): = 1/f(x,t)$ represents the spatial density of Lagrangian material points along the elastic string.
Thus, there should be no surprise that $F$ satisfies a conservation law \eqref{eqn: equation for big F}, and it is natural to assume $F$ to have $L^1$-regularity at least.
\end{rmk}

\begin{rmk}
\label{rmk: scaling}
If $f = f(x,t)$ is a sufficiently smooth solution to \eqref{eqn: equation for f} on $\BR$ with initial data $f(x,0) = f_0(x)$, then for any $\lambda >0$,
\[
\lam f(x,\lam t), \; f(\lam x,\lam t),\; \lam f\left(\f{ x}{\lam}, t\right)
\]
are all solutions, corresponding to initial datum $\lam f_0(x)$, $f_0(\lam x)$, and $\lam f_0(x/\lam)$, respectively.
\end{rmk}

\begin{rmk}
\label{rmk: remark on the periodicity}
We imposed periodicity in \eqref{eqn: equation for f} mainly for convenience of analysis.
Many results in this paper should still hold in the case of $\BR$, given suitable decay conditions at the spatial infinity.
In that case, one has to pay extra attention to integrability issues from time to time.
We thus choose to work on $\BT$ to avoid such technicality.

If the spatial periodicity is to be imposed, it is enough to assume the $2\pi$-periodicity, as one can always rescale in space and time.
Also note that the $2\pi$-periodicity is imposed in the Eulerian coordinate.
In the Lagrangian coordinate, $X'(\cdot,t)$ is periodic as well given the monotonicity and bijectivity assumptions of $X(\cdot,t)$, but the period is $\f{2\pi}{f_\infty}$, with $f_\infty$ defined in \eqref{eqn: equilibrium}.
Indeed,
\[
X^{-1}(x+2\pi,t)-X^{-1}(x,t)= \int_{X^{-1}(x,t)}^{X^{-1}(x+2\pi,t)} \f{1}{f(X(s,t),t)}\,dX(s,t) = \int_{x}^{x+2\pi} \f{1}{f(y,t)}\,dy.
\]
Observe that the last integral is a constant independent of $x$ and $t$ thanks to the periodicity of $f$ and the conservation law \eqref{eqn: equation for big F} of $\f1f$.
\end{rmk}

\begin{rmk}
One may consider general elasticity laws.
Assume the tension takes the simple form $\CT = \CT(|X'(s,t)|)$, with $\CT(\cdot)$ satisfying suitable assumptions.
Then we can follow the above argument to obtain (again with the coefficient $\f14$ discarded)
\beq
\pa_t f
= \CH \big(\CT(f)\big) \cdot \pa_x f - f (-\D)^{\f12} \CT(f).
\label{eqn: equation for f general tension}
\eeq
Here $f$ is still defined by \eqref{eqn: def of f}.
In this case, $F:= \f1{f}$ solves
\beq
\pa_t F = \pa_x\big(\CH [\CT(f)] \cdot F\big).
\label{eqn: equation for big F general tension}
\eeq

Suppose $\CT:[0,+\infty]\to [0,+\infty]$ is sufficiently smooth and strictly increasing.
Define $g(x,t):= \CT(f(x,t))$.
Then \eqref{eqn: equation for f general tension} can be written as
\beq
\pa_t g
= \CH g \cdot \pa_x g - \CN(g) (-\D)^{\f12} g.
\label{eqn: the tension equation}
\eeq
Here $\CN$ is a function that sends $\CT(f)$ to $\CT'(f)f$, i.e.,
\[
\CN(g) := \CT'\big(\CT^{-1}(g)\big)\cdot \CT^{-1}(g).
\]
If $\CT(f) = f^\g$ with $\g>0$, then $\CN(g) = \g g$.
% and thus  becomes
%\beqo
%\pa_t g
%= \CH g \cdot \pa_x g - \g g (-\D)^{\f12} g,
%\eeqo
Denoting $\om: = -\g g$, we may write \eqref{eqn: the tension equation} as
\[
\pa_t \om +\g^{-1} u \pa_x \om = \om \pa_x u,\quad
u = \CH \om.
\]
This provides a formal analogue of Okamoto-Sakajo-Wunsch generalization \cite{okamoto2008generalization} of the De Gregorio model (also see Section \ref{sec: related studies}).
\end{rmk}

Most of, if not all, the previous analytical studies of the 2-D Peskin problem are based on the Lagrangian formulations, such as \eqref{eqn: contour dynamic equation}, \eqref{eqn: contour dynamic equation 1D prelim}, and \eqref{eqn: contour dynamic equation 1D},
whereas \eqref{eqn: equation for f} (or equivalently \eqref{eqn: the main equation}) is formulated in the Eulerian coordinate.
We will thus recast our main result Theorem \ref{thm: main thm} in the Lagrangian coordinate.
Even though it is concerned with a special setting of the 2-D Peskin problem, compared with the existing results, it significantly weakens the assumption on the initial data for which a global solution can be well-defined.
%We also direct the readers to the remark at the end of this section.

Recall that $X = X(s,t)$ represents the evolving configuration of the infinitely long elastic string lying along the $x$-axis.
It is assumed to be periodic in space.
It is supposed to solve \eqref{eqn: contour dynamic equation 1D} in the Lagrangian coordinate with the initial condition $X(s,0) = X_0(s)$.
Regarding the existence of such an $X(s,t)$, we have the following result, which will be proved in Section \ref{sec: existence for general initial data}.

\begin{cor}
\label{cor: rephrasing in Lagrangian}
Suppose $X_0:s\mapsto X_0(s)\in \BR$ is a function defined on $\BR$, satisfying that
\begin{enumerate}[(i)]
  \item \label{assumption: periodicity}
  $X_0(s)$ is strictly increasing in $s$, and $X_0(s+2\pi) = X_0(s)+2\pi$ for any $s\in \BR$;
  \item \label{assumption: energy class}
  When restricted on $[-\pi,\pi]$, $X_0\in H^1([-\pi,\pi])$;
  \item \label{assumption: relaxation of well-stretched condition}
      The inverse function of $X_0$ on any compact interval of $\BR$ is absolutely continuous.
\end{enumerate}
Then there exists a function $X = X(s,t)$ defined on $\BR \times [0,+\infty)$, satisfying that
\begin{enumerate}[(1)]
\item $X(s,t)$ is a smooth strong solution of \eqref{eqn: contour dynamic equation 1D} in $\BR\times (0,+\infty)$, i.e.,
\[
\pa_t X(s,t)
=  -\f{1}{4\pi} \mathrm{p.v.}\int_\BR \f{ X'(s',t)^2}{X(s,t)-X(s',t)} \,ds'
\]
holds pointwise in $\BR\times (0,+\infty)$.

\item \label{property: uniform convergence}
For any $\al\in (0,\f12)$,
$X(s,t)-s \in C^\al(\BR\times [0,+\infty))$.
As a consequence, $X(\cdot,t)$ converges uniformly to $X_0(\cdot)$ as $t\to 0^+$.
%\textcolor{red}{To be checked.}

\item For any $t\geq 0$, $X(\cdot,t)$ verifies the above three assumptions on $X_0(\cdot)$, with $\|X(\cdot,t)\|_{\dot{H}^1([-\pi,\pi])}$ being non-increasing in $t\in [0,+\infty)$.
    It additionally satisfies the well-stretched condition when $t>0$, i.e., there exists $\lambda = \lambda(t)>0$, such that
\[
|X(s_1,t)-X(s_2,t)|\geq \lambda(t)|s_1-s_2|\quad \forall\, s_1,s_2 \in \BR.
\]
In fact, $\f{1}{\lambda(t)}$ satisfies an estimate similar to \eqref{eqn: upper bound for F}. %, or equivalently, $\lam(t)$ enjoys a lower bound similar to \eqref{eqn: lower bound for f}.

\item
There exists a constant $c_\infty$, such that as $t\to +\infty$, $X(s,t)$ converges uniformly to $X_\infty(s): = s+c_\infty$ with an exponential rate.
The exponential convergence also holds for higher-order norms.
\end{enumerate}

\end{cor}

For simplicity, we assumed in \eqref{assumption: periodicity} that the periodicity of $X_0$ is $2\pi$ in both the Lagrangian and the Eulerian coordinate (cf.\;Remark \ref{rmk: remark on the periodicity}).
The assumption \eqref{assumption: energy class} only requires $X_0$ to belong to the energy class (see the discussion in Section \ref{sec: 2D Peskin} and also Remark \ref{rmk: energy class} below), without imposing any condition on its size.
The assumption \eqref{assumption: relaxation of well-stretched       condition} is a relaxation of the well-stretched assumption that is commonly used in the literature.
The function $X(s,t)$ found in Corollary \ref{cor: rephrasing in Lagrangian} can be naturally defined as a global solution to \eqref{eqn: contour dynamic equation 1D} with the initial data $X_0$.
For brevity, we choose not to elaborate the notion of the solution.

\section{Preliminaries}
\label{sec: a priori estimates}

In Sections \ref{sec: a priori estimates}-\ref{sec: smoothing and decay}, we will establish a series of a priori estimates for nice solutions to \eqref{eqn: the main equation} (and also \eqref{eqn: equation for big F}), which prepares us for proving existence of weak solutions later.
Throughout these sections, we will always assume $f = f(x,t)$ is a positive strong solution to \eqref{eqn: the main equation}.
The definition of the strong solution is as follows.

\begin{definition}
\label{def: strong solution}
Given $T\in [0,+\infty]$, $f(x,t)$ is said to be a strong solution to \eqref{eqn: the main equation} on $\BT\times [0,T)$, if \
\begin{enumerate}[(a)]
\item $f\in C_{loc}^1(\BT\times [0,T))$;
\item $(-\D)^{\f12} f$ can be defined pointwise as a function in $C_{loc}(\BT\times [0,T))$;
\item \eqref{eqn: the main equation} holds pointwise in $\BT\times [0,T)$.
\end{enumerate}
Here the time derivative of $f$ at $t = 0$ is understood as the right derivative.
\end{definition}

When discussing strong solutions, we always implicitly assume that the initial data $f_0$ is correspondingly smooth.

We start with a few basic properties of positive strong solutions.

\begin{lem}
\label{lem: energy bound and decay of L^p norms}
Suppose $f$ is a positive strong solution to \eqref{eqn: the main equation} on $\BT\times [0,T)$.
Then
\begin{enumerate}
\item (Energy estimate) For any $t\in [0,T)$,
\beq
\f12 \|f(\cdot,t)\|_{L^1} + \int_0^t \|f(\cdot,\tau)\|_{\dot{H}^{1/2} }^2\,d\tau = \f12 \|f_0\|_{L^1}.
\label{eqn: energy estimate}
\eeq

\item \label{part: decay of L^p norms of f and F}
(Decay of $L^p$-norms of $f$ and $F$)
For any $p\in [1,+\infty)$, $\|f(\cdot,t)\|_{L^p }$ and $\|F(\cdot,t)\|_{L^p}$ are non-increasing in $t$.
In particular, $\|F(\cdot,t)\|_{L^1}$ is time-invariant, and $\|f(\cdot,t)\|_{L^1}\geq 2\pi f_\infty$, where $f_\infty$ is defined in \eqref{eqn: equilibrium}.
\end{enumerate}

\begin{proof}
\eqref{eqn: energy estimate} can be readily proved by integrating \eqref{eqn: the main equation} and integration by parts.

For arbitrary $\alpha\in \BR$, %we have
%\beq
%\pa_t f^\alpha =  \CH f \cdot \pa_x f^\alpha - \alpha f^\alpha (-\D)^{\f12} f,
%\label{eqn: equation for f alpha power}
%\eeq
%or equivalently,
\beq
\pa_t f^\alpha = \pa_x \big(\CH f \cdot f^\alpha \big) - (1+\alpha)f^\alpha (-\D)^{\f12} f.
\label{eqn: equation for f alpha power conversed form}
\eeq
Integrating \eqref{eqn: equation for f alpha power conversed form} yields
\[
\f{d}{dt}\int_\BT f^\alpha(x,t) \,dx = -\f{1+\alpha}{\pi} \int_{\BT} f^{\alpha}(x,t)\cdot \mathrm{p.v.}\int_\BT \f{f(x,t) - f(y,t)}{4\sin^2\big(\f{x-y}{2}\big)}\,dy\,dx
\]
Given the regularity assumption on $f(\cdot,t)$, we may exchange the $x$- and $y$-variables to obtain
%Hence,
\[
\begin{split}
\f{d}{dt}\int_\BT f^\alpha(x,t) \,dx = &\; -\f{1+\alpha}{2\pi} \lim_{\d \to 0^+} \int_{|x-y|\geq \d} \big(f^{\alpha}(x,t)-f^{\alpha}(y,t)\big)\cdot \f{f(x,t) - f(y,t)}{4\sin^2\big(\f{x-y}{2}\big)}\,dx\,dy.
\end{split}
\]
%We skip its justification.
It is clear that, for any $a,b \geq 0\in \BR$,
\[
\big(a^{\alpha}-b^{\alpha}\big)(a-b)
\begin{cases}
\geq 0, & \mbox{if } \alpha > 0, \\
=0,     & \mbox{if } \alpha = 0, \\
\leq 0, & \mbox{if } \alpha < 0.
\end{cases}
\]
Hence, for all $\alpha\in (-\infty,-1]\cup[0,+\infty)$,
\[
\f{d}{dt}\int_\BT f^\alpha(x,t) \,dx \leq 0.
\]
This complete the proof of the second claim.
%In particular, if $F\in L^p$ for any $p\geq 1$, then it is uniformly $L^p$.
The time-invariance of $\|F(\cdot,t)\|_{L^1}$ follows directly from \eqref{eqn: equation for big F}.
We can further prove the lower bound for $\|f(\cdot,t)\|_{L^1}$ by the Cauchy-Schwarz inequality.
\end{proof}
\begin{rmk}
\label{rmk: energy class}
Recall that \eqref{eqn: the main equation} is derived from the Hookean elasticity case with $k_0 = 1$, where the elastic energy density is $\CE(p) = \f12 p^2$ as $\CE'(p) = \CT(p)$.
We also find
\[
\f12\|f(\cdot,t)\|_{L^1(\BT)} = \f12 \int_{X^{-1}(0,t)}^{X^{-1}(2\pi,t)} f(X(s,t),t)\,dX(s,t) = \int_{X^{-1}(0,t)}^{X^{-1}(2\pi,t)}\f12 |X'(s,t)|^2\,ds,
\]
which is exactly the total elastic energy of the string in one period in the Lagrangian coordinate.
This is why we called \eqref{eqn: energy estimate} the energy estimate.
Hence, we shall also call $L^1(\BT)$ the energy class for \eqref{eqn: the main equation}.
%\textcolor{red}{
In general, for \eqref{eqn: equation for f general tension}, the energy estimate writes
\[
\f{d}{dt}\int_\BT \f{\CE(f)}{f}\,dx = -\|\CT(f)\|_{\dot{H}^{\f12}}^2, \mbox{ where }\CE(p) := \int_0^p \CT(q)\,dq.
\]
%}
\end{rmk}
\end{lem}

\begin{lem}
\label{lem: max principle}
Assume $f$ to be a positive strong solution to \eqref{eqn: the main equation} on $\BT\times [0,T)$.
Then $f$ satisfies maximum and minimum principle,
i.e., for all $0\leq t_1\leq t_2<T$,
\[
\inf_{x} f(x,t_1)\leq \inf_{x} f(x,t_2) \leq \sup_{x} f(x,t_2) \leq \sup_{x} f(x,t_1).
\]
If we additionally assume $\pa_x f \in C_{loc}^1(\BT\times [0,T))$ and that $(-\D)^{\f12}(\pa_x f)$ can be defined pointwise in $\BT\times [0,T)$ and is locally continuous,
then $\pa_x f$ also satisfies the maximum and minimum principle, i.e.,
\[
\inf_{x} \pa_x f(x,t_1)\leq \inf_{x} \pa_x f(x,t_2) \leq \sup_{x} \pa_x f(x,t_2) \leq \sup_{x} \pa_x f(x,t_1).
\]

\begin{proof}
It follows immediate from \eqref{eqn: the main equation} that $f$ satisfies the maximum/minimum principle, since $f\geq 0$.
Differentiating \eqref{eqn: the main equation} yields
\beqo
\pa_t\big(\pa_x f\big)
= \CH f \cdot \pa_x \big(\pa_x f\big) - f (-\D)^{\f12} \big(\pa_x f\big).
%\label{eqn: equation for derivative of f}
\eeqo
Hence, $\pa_x f$ also satisfies the maximum/minimum principle.
\end{proof}
\begin{rmk}
\label{rmk: meaning of max and min principle}
In the context of the Peskin problem, the well-stretched assumption \eqref{eqn: well-stretched condition} (with constant $\lambda>0$) corresponds to $f(\cdot,t)\geq \lam$ in \eqref{eqn: the main equation} and \eqref{eqn: equation for f}.
As a result, strong solutions to \eqref{eqn: contour dynamic equation 1D prelim} satisfy well-stretched condition for all time with the constant $\lam$ as long as it is initially this case.
Besides, since
\[
\pa_x f(X(s,t),t) = \f{X''(s,t)}{X'(s,t)} = \f12 \pa_{s} \ln (|X'(s)|^2),
\]
if $f(x,0)\geq \lam>0$ and $\|X_0''\|_{L^\infty}<+\infty$, we have $X''$ to be uniformly bounded in time.
\end{rmk}
\end{lem}

\section{Instant Positivity and Boundedness}
\label{sec: positivity and boundedness}
%Next, we establish estimates regarding upper and lower bounds for $f$.
In this section, we want to prove that for any $t>0$, the solution $f(\cdot,t)$ admits positive and finite lower and upper bounds, which only depend on $\|f_0\|_{L^1}$ and $\|F_0\|_{L^1}$ a priori.

\begin{lem}
\label{lem: lower bound}
Suppose  $f = f(x,t)$ is a positive strong solution on $\BT\times [0,+\infty)$.
Denote
$f_*(t) := \min_{x\in \BT} f(x,t)$.
%Denote $A = \sqrt{8}\|F_0\|_{L^1(\BT)}^{-1/2}>0$.
Then for all $t>0$,
\beq
%f_*(t) \geq %\min \left\{\exp\left(-\f{1}{A}\right),\,
%\exp\left[-\f1A\cdot \coth\left(\f{A}{2\pi}t\right)\right], %\right\},
f_*(t)\geq 8\|F_0\|_{L^1}^{-1}
\exp\left[- \coth\left(\f{4}{\pi}\|F_0\|_{L^1}^{-1} t\right)\right].
\label{eqn: lower bound for f}
\eeq
or equivalently,
%if $\|F(\cdot,t)\|_{L^\infty}\gg 1$, then
\[
\|F(\cdot,t)\|_{L^\infty}\leq \f18\|F_0\|_{L^1}
%\max\left\{\exp\left(\f{1}{A}\right),\,
\exp\left[\coth\left(\f{4}{\pi}\|F_0\|_{L^1}^{-1} t\right)\right].
%\right\}
\]
In particular, when $t\leq \|F_0\|_{L^1}$,
\beq
\|F(\cdot,t)\|_{L^\infty}\leq \f18\|F_0\|_{L^1}
%\max\left\{\exp\left(\f{1}{A}\right),\,
\exp\big(C\|F_0\|_{L^1} t^{-1}\big),
%\right\}
\label{eqn: simplified upper bound for F}
\eeq
where $C$ is a universal constant.
\begin{proof}
Denote $A := \sqrt{8}\|F_0\|_{L^1(\BT)}^{-1/2}>0$.
We first prove that
\beq
f_*(t) \geq %\min \left\{\exp\left(-\f{1}{A}\right),\,
\exp\left[-\f1A\cdot \coth\left(\f{A}{2\pi}t\right)\right].
\label{eqn: lower bound for f prelim}
\eeq

Suppose the minimum $f_*(t)$ is attained at $x_*(t)$.
Since $\coth(\f{A}{2\pi}t)>1$, we may assume $f_*(t)< \exp(-\f{1}{A})$ for all time of interest, as otherwise \eqref{eqn: lower bound for f prelim} holds automatically. % due to the minimum principle Lemma \ref{lem: max principle}.
We first derive a lower bound for $-(-\D)^{\f12} f(x_*(t),t)$.
With $\d\in (0,\pi)$ to be determined,
\[
\begin{split}
&\;-(-\D)^{\f12} f(x_*,t)\\
\geq  &\; \f1\pi \int_{\BT \backslash [-\d,\d]} \f{f(x_*+y)-f(x_*)}{4\sin^2\big(\f{y}{2}\big)}\,dy\\
= &\; \f1\pi \left( \int_{\BT \backslash [-\d,\d]} \f{f(x_*+y)}{4\sin^2\big(\f{y}{2}\big)}\,dy\right)
\left(\int_{\BT \backslash [-\d,\d]} \f{1}{f(x_*+y)}\,dy\right)
\left(\int_{\BT \backslash [-\d,\d]} F(x_*+y)\,dy\right)^{-1}\\
&\;- \f{f(x_*)}{\pi \tan\big(\f{\d}{2}\big)}\\
\geq &\; \f1\pi \left( \int_\d^\pi \f{1}{\sin\big(\f{y}{2}\big)}\,dy\right)^2
\|F\|_{L^1(\BT)}^{-1}
- \f{f(x_*)}{\pi \tan\big(\f{\d}{2}\big)}\\
\geq &\; \f{4}{\pi}\left|\ln \tan\f{ \d}{4}\right|^2 \|F_0\|_{L^1(\BT)}^{-1}
- \f{f(x_*)}{2\pi \tan\big(\f{\d}{4}\big)}.
\end{split}
\]
%where $C$ is a universal constant.
Here we used Cauchy-Schwarz inequality and the fact $\|F(\cdot,t)\|_{L^1(\BT)} \equiv \|F_0\|_{L^1(\BT)}$.
%\[
%\f{d}{dt}f_*(t) \geq \f14 f_*(t) (-\D)^{\f12} f(x_*(t),t).
%\]
Since $f_*(t)< \exp(-\f{1}{A})<1$, we may take $\d\in (0,\pi)$ such that
$\tan (\f{\d}{4}) =  f(x_*)$.
This gives
\[
-(-\D)^{\f12} f(x_*,t)
\geq \f{4}{\pi}\|F_0\|_{L^1(\BT)}^{-1} \big|\ln f(x_*,t)\big|^2 - \f{1}{2\pi},
\]

Since $f\in C_{loc}^1(\BT\times [0,+\infty))$, we may argue as in e.g.\;\cite[Theorem 3.1]{cordoba2009maximum} to find that
%Now we claim as in \textcolor{red}{(add ref here, I suppose its proof is similar to that in the Muskat problem. The argument may affect the definition of strong solutions.)} that
\[
\f{d}{dt}f_*(t) = - f(x_*,t)(-\D)^{\f12} f(x_*,t) \geq f_*(t) \left[\f{4}{\pi}\|F_0\|_{L^1(\BT)}^{-1} \big|\ln f_*(t)\big|^2 - \f{1}{2\pi}\right].
\]
Hence, with $A = \sqrt{8}\|F_0\|_{L^1(\BT)}^{-1/2}>0$,
\[
\f{d}{dt}\left(\f{1}{\ln f_*(t)}\right)
\leq
-\f1{2\pi} \left[A^2
%8\|F_0\|_{L^1(\BT)}^{-1}
- \big(\ln f_*(t)\big)^{-2}\right].
\]
%\textcolor{red}{We need to make sure the right-hand side has the correct sign.}
%Recall
Denote $g_*(t) = (\ln f_*(t))^{-1}$.
Since $f_*(t)<\exp(-\f{1}{A})$, we have $g_*(t)\in (-A,0]$ for all time of interest (including $t = 0$).
Then we write the above inequality as
%\[
%\left(\f{1}{A - g_*}+\f{1}{A + g_*} \right) \f{dg_*}{dt}  \leq -\f{A}{4\pi}.
%\]
\[
\f{d}{dt}\ln \left(\f{A + g_*(t)}{A - g_*(t)} \right) \leq -\f{A}{\pi}.
\]
Since $g_*(0)\in (-A,0]$,
\[
\ln \left(\f{A + g_*(t)}{A - g_*(t)} \right) \leq -\f{A}{\pi} t + \ln \left(\f{A + g_*(0)}{A - g_*(0)} \right)\leq -\f{A}{\pi}t.
\]
Therefore,
%\[
%\f{2A}{A - g_*(t)} \leq 1+\exp\left(-\f{A}{4\pi}t\right).
%\]
\[
\f{1}{\ln f_*(t)} \leq A\cdot \f{\exp\big(-\f{A}{\pi}t\big)-1}{\exp\big(-\f{A}{\pi}t\big)+1}.
\]
Then \eqref{eqn: lower bound for f prelim} follows.

Recall that if $f = f(x,t)$ is a strong solution to \eqref{eqn: the main equation} on $\BT$, then for any $\lambda >0$, $f_\lam (x,t): = \lam f(x,\lam t)$ is a solution corresponding to initial condition $\lam f_0(x)$.
Now applying \eqref{eqn: lower bound for f prelim} to $f_\lam$, we deduce that, %for all $\lam>0$,
\beqo
f_{\lam,*}(t)  := \min_{x\in \BT} f_\lam (x,t)\geq
\exp\left[-\f1{\lam^{1/2}A}\cdot \coth\left(\f{\lam^{1/2}A}{2\pi}t\right)\right].
\eeqo
Here we still defined $A = \sqrt{8}\|F_0\|_{L^1(\BT)}^{-1/2}$.
Since $\min_{x\in \BT} f_\lam (x,t) = \lam \min_{x\in \BT} f (x,\lam t)$, for all $\lam,\tau>0$,
\beqo
\min_{x\in \BT}  f(x,\tau)\geq \lam^{-1}
\exp\left[-\f1{\lam^{1/2}A}\cdot \coth\left(\f{\lam^{-1/2}A}{2\pi}\tau\right)\right].
\eeqo
Taking $\lam = A^{-2}$ yields \eqref{eqn: lower bound for f}.
\end{proof}
\begin{rmk}
We chose \eqref{eqn: lower bound for f} instead of \eqref{eqn: lower bound for f prelim} as the main estimate, because \eqref{eqn: lower bound for f} is scaling-invariant while \eqref{eqn: lower bound for f prelim} is not.
Note that $\|F_0\|_{L^1}^{-1} t$ is a dimensionless quantity under the scaling $f_\lam (x,t): = \lam f(x,\lam t)$.
\end{rmk}
\end{lem}

Next we turn to study the upper bound for $f(\cdot,t)$.
In the following lemma, we will prove a ``from $L^1$ to $L^\infty$" type estimate, which states that the $\|f(\cdot,t)\|_{L^\infty}$ should enjoy a decay like $t^{-1/2}$ when $t\ll 1$.
This may be reminiscent of similar boundedness results in, e.g., \cite{caffarelli2010drift,caffarelli2013regularity}, which are proved by considering cut-offs of the solution and applying a De Giorgi-Nash-Moser-type iteration (without the part of proving H\"{o}lder regularity).
Here we provide a different argument.

\begin{lem}
\label{lem: L inf bound}
%Let $f_0\in L^1(\BT)$.
Suppose $f = f(x,t)$ is a strong solution to \eqref{eqn: the main equation} on $\BT\times [0,+\infty)$.
Then for $0<t\leq \|f_0\|_{L^1}^{-1}$,
\[
\|f(\cdot,t)\|_{L^\infty(\BT)}\leq Ct^{-\f12} \|f_0\|_{L^1}^{\f12},
\]
where $C>0$ is a universal constant.
Consequently (cf.\;Lemma \ref{lem: max principle}), for $t>\|f_0\|_{L^1}^{-1}$,
\[
\|f(\cdot,t)\|_{L^\infty(\BT)}\leq C\|f_0\|_{L^1}.
\]
\begin{rmk}
The exponents of $t$ and $\|f_0\|_{L^1}$ in the above estimate are sharp in the view of dimension analysis (cf.\;the energy estimate in Lemma \ref{lem: energy bound and decay of L^p norms}).
\end{rmk}

\begin{proof}
Without loss of generality, let us assume $\|f_0\|_{L^1}\geq 1$ up to suitable scaling.
We proceed in several steps.

\setcounter{step}{0}
\begin{step}[Basic setup]
Let $W: [0,+\infty) \to [0,+\infty)$ be a convex increasing function, which is to be specified later.
For the moment, we assume $W$ is finite and locally $C^{1,1}$ on $[0,+\infty)$, $W(0) = 0$, and $W'(0) = 1$.
Define
\[
I(t) =
\begin{cases}
  \|f_0\|_{L^1}, & \mbox{if } t=0, \\
  \f1{t^{1/2}} \int_\BT W\big(t^{1/2} f(x,t)\big)\,dx, & \mbox{if } t>0.
\end{cases}
\]
Thanks to the assumption on $W$ and that $f$ is a strong solution, $I(t)$ is locally finite near $t = 0$, continuous at $0$, and continuously differentiable on $(0,+\infty)$ whenever it is finite.

Let us explain the idea in the rest of the proof.
We introduce the above functional mainly to study the size of the set
\beq
S_\Lam(t) := \{x\in \BT:\, f(x,t) > \Lambda/t^{1/2}\}
\label{eqn: def of the set S_Lam}
\eeq
for large $\Lam > 0$.
We will show that, one can suitably choose $W$, which will depend on the Lebesgue measure of $S_\Lam(t)$ (denoted by $|S_\Lam(t)|$), such that $I(t)$ stays bounded for all time.
This in turn provides an improved bound for $|S_\Lam(t)|$ for sufficiently large $\Lam$, which gives rise to an improved choice of $W$.
Then we bootstrap to find $|S_\Lam(t)| = 0$ for all $\Lam$'s that are sufficiently large.
\end{step}

\begin{step}[The growth of $I(t)$]
Fix $t \in (0,T]$.
By direct calculation and integration by parts,
\beqo
\begin{split}
\f{d}{dt}I (t)
= &\; \f{1}{2t^{3/2}} \int_\BT  W'\big(t^{1/2} f \big)\cdot t^{1/2} f - W(t^{1/2} f)\,dx\\
&\; - \f1{t} \int_\BT \left[ W'\big(t^{1/2} f \big) \cdot  t^{1/2} f - W\big(t^{1/2} f \big)\right]\cdot (-\D)^{\f12} (t^{1/2} f)  \,dx\\
&\; - \f2{t} \int_\BT W\big(t^{1/2} f \big) \cdot (-\D)^{\f12} (t^{1/2} f)  \,dx.
\end{split}
%\label{eqn: ODE for I(t)}
\eeqo
Denote $g(x) := t^{1/2}f(x,t)$ and $V^2(y) := W'(y)y-W(y)$.
Then this becomes
\beq
\begin{split}
\f{d}{dt}I (t)
= &\; \f{1}{2t^{3/2}} \int_\BT  V^2\big(g(x)\big)\,dx
 - \f1{t} \int_\BT V^2\big(g(x)\big)\cdot (-\D)^{\f12} g(x) \,dx\\
&\; - \f2{t} \int_\BT W\big(g(x) \big) \cdot (-\D)^{\f12} g(x) \,dx.
\end{split}
\label{eqn: ODE for I(t)}
\eeq
With $\Lam > 0$ to be determined,
we define
\[
g_\Lam(x) : = \min\{g(x),\Lambda\},
%g_{\lam,\Lam}(x) : = \max\left\{\min\{g(x),\Lambda\},\lam\right\},
\quad
\tilde{g}_\Lambda(x) : = \max\{g(x),\Lam\}.
\]
In what follows, we shall first take $V$ properly and then determine $W$.

Assume $V$ to be increasing. %$g_{\lam,\Lam} \geq \lam$, and $\tilde{g}_\Lam \geq \Lam$,
Then for the first term on the right-hand side of \eqref{eqn: ODE for I(t)}, we have
\beq
\begin{split}
\int_\BT V^2\big(g(x)\big)\,dx
=&\;\int_\BT \left[ V\big(g_\Lam(x)\big) + \big(V\big(\tilde{g}_\Lam(x)\big)-V(\Lam)\big)\right]^2\,dx\\
\leq &\;
2 \int_{\BT} V^2\big(g_\Lam(x)\big)\,dx
+ 2 \int_{\BT} \big(V\big(\tilde{g}_\Lam(x)\big)-V(\Lambda)\big)^2 \,dx.
\end{split}
\label{eqn: bound for the L^2 norm of V}
\eeq
Assume  %with some universal $C_0>0$ to be determined,
$V^2(y) \leq y^3$
%\label{eqn: assumption on V for small y}
for $y\in [0,\Lam]$.
Then %Thanks to \eqref{eqn: assumption on V for small y},
\beq
\int_{\BT} V^2\big(g_\Lam(x)\big)\,dx
\leq
\int_{\BT} \min\big\{t^{1/2} f(x,t),\Lam\big\}^3 \,dx
\leq
t^{\f32} \int_\BT f(x,t)^3\,dx.
\label{eqn: bound for integral of lower part of V}
\eeq
For the second term in \eqref{eqn: bound for the L^2 norm of V}, we first assume
\beq
\Lam \geq 2\|f_0\|_{L^1}^{\f12},
\label{eqn: largeness of Lambda}
\eeq
and then apply H\"{o}lder's inequality and interpolation (see e.g.\;\cite{kozono2008remarks}) to derive that
\beq
\begin{split}
&\; \int_{\BT} \big(V\big(\tilde{g}_\Lam(x)\big)-V(\Lambda)\big)^2 \,dx\\
\leq &\; \left\|V\big(\tilde{g}_\Lam(x)\big)-V(\Lambda) \right\|_{L^4(\BT)}^2
\big|\{x\in \BT:\, g(x,t) > \Lambda\}\big|^{\f12}\\
\leq &\; C\left\|V\big(\tilde{g}_\Lam(x)\big)-V(\Lambda) \right\|_{L^2(\BT)}
\left\|V\big(\tilde{g}_\Lam(x)\big)-V(\Lambda) \right\|_{\dot{H}^{\f12}(\BT)}
\big|S_\Lam(t)\big|^{\f12}.
\end{split}
\label{eqn: improved Poincare}
\eeq
%Recall that $S_\Lam(t)$ is defined in \eqref{eqn: def of the set S_Lam}.
Here we can use $H^{1/2}$-semi-norm because, for $\Lam \geq 2\|f_0\|_{L^1}^{1/2}$, $V(\tilde{g}_{\Lam}(x))-V(\Lam)$ is zero on a set of positive measure. %\textcolor{red}{
Indeed, this is because, when $t\leq \|f_0\|_{L^1}^{-1}$,
\[
\int_{\BT} g(x,t)\,dx = t^{\f12} \|f(\cdot,t)\|_{L^1} \leq \|f_0\|_{L^1}^{-\f12}\|f_0\|_{L^1}\leq \f12\Lam.
\]
Hence, \eqref{eqn: improved Poincare} gives
\beq
\int_{\BT} \big(V\big(\tilde{g}_\Lam(x)\big)-V(\Lambda)\big)^2 \,dx
\leq  C_1
\big\|V\big(\tilde{g}_\Lam(x)\big) \big\|_{\dot{H}^{\f12}(\BT)}^2
\big|S_\Lam(t)\big|,
\label{eqn: upper bound for large part in the L^2 norm}
\eeq
where $C_1>0$ is a universal constant.
Summarizing \eqref{eqn: bound for the L^2 norm of V}, \eqref{eqn: bound for integral of lower part of V}, and \eqref{eqn: upper bound for large part in the L^2 norm}, we obtain that
\beq
\f{1}{2t^{3/2}}\int_\BT V^2\big(g(x)\big)\,dx
\leq
\|f(x,t)\|_{L^3}^3
+\f{C_1}{t^{3/2}}
\big\|V\big(\tilde{g}_\Lam(x)\big) \big\|_{\dot{H}^{\f12}(\BT)}^2
\big|S_\Lam(t)\big|.
\label{eqn: upper bound for the L^2 norm term}
\eeq

Now we handle the second term on the right-hand side of \eqref{eqn: ODE for I(t)}.
Since $V^2$ is increasing,
%\textcolor{red}{(We need to change the formula for $(-\D)^{\f12}$ because we are now dealing with $\BT$)}
\[
\begin{split}
&\;\int_\BT V^2(g(x))\cdot (-\D)^{\f12} g(x)\,dx \\
%
%= &\; \f{1}{2\pi}\int_\BR \int_\BR \f{\big(V^2(g(x))-V^2(g(y))\big)\big(g(x)-g(y)\big)}{(x-y)^2}\,dx\,dy\\
%%
\geq &\; \f{1}{2\pi}\int_\BT \int_\BT \f{\big(V^2(\tilde{g}_\Lam(x))-V^2(\tilde{g}_\Lam(y))\big)\big(\tilde{g}_\Lam(x)-\tilde{g}_\Lam(y)\big)} {4\sin^2\big(\f{x-y}{2}\big)}\,dx\,dy\\
= &\; \f{1}{2\pi}\int_{\{\tilde{g}_\Lam(x) \neq  \tilde{g}_\Lam(y)\}} \f{\big(V(\tilde{g}_\Lam(x))+V(\tilde{g}_\Lam(y))\big)\big(\tilde{g}_\Lam(x)-\tilde{g}_\Lam(y)\big)} {V(\tilde{g}_\Lam(x))-V(\tilde{g}_\Lam(y))} \cdot \f{\big(V(\tilde{g}_\Lam(x))-V(\tilde{g}_\Lam(y))\big)^2}
{4\sin^2\big(\f{x-y}{2}\big)} \,dx\,dy.
\end{split}
\]
Suppose $V$ is increasing and convex on $[0,+\infty)$, such that for all $z>\Lam$ and all $t\in [0,T]$,
\[
\f{V'(z)}{V(z)} \leq
t^{\f12}\left(C_1
\big|S_\Lam(t)\big|\right)^{-1}.
\]
Then by the mean value theorem,
\[
\f{\big(V(\tilde{g}_\Lam(x))+V(\tilde{g}_\Lam(y))\big)\big(\tilde{g}_\Lam(x)-\tilde{g}_\Lam(y)\big)}{V(\tilde{g}_\Lam(x))-V(\tilde{g}_\Lam(y))}
\geq t^{-\f12}C_1
\big|S_\Lam(t)\big|.
\]
Therefore,
\beq
\f{1}{t}\int_\BT V^2(g(x))\cdot (-\D)^{\f12} g(x)\,dx
\geq  \f{C_1}{t^{3/2}}\big\|V(\tilde{g}_\Lam(x))\big\|_{\dot{H}^{\f12}(\BT)}^2
\big|S_\Lam(t)\big|.
\label{eqn: lower bound for fractional Sobolev norm}
\eeq

The last term on the right-hand side of \eqref{eqn: ODE for I(t)} is clearly non-positive, since $W$ is non-decreasing.
Hence, combining
\eqref{eqn: ODE for I(t)}, \eqref{eqn: upper bound for the L^2 norm term}, and \eqref{eqn: lower bound for fractional Sobolev norm} yields %$I'(t) \leq \|f(x,t)\|_{L^3}^3$.
\beqo
\f{d}{dt}I(t)
%\leq \|f(x,t)\|_{L^3}^3 - \f1{2t} \int_\BT W\big(g(x)\big) \cdot (-\D)^{\f12}g(x) \,dx
\leq \|f(x,t)\|_{L^3}^3.
\eeqo
Thanks to the energy estimate \eqref{eqn: energy estimate} and interpolation, for all $t>0$, %at least for a short period of time (up to $O(\|f_0\|_{L^1}^{-1})$).
%Recall that
%\[
%I(0) = \int_\BT f(x,0)\,dx = \|f_0\|_{L^1}.
%\]
\beq
I(t)\leq I(0)+ \int_0^t \|f(\cdot,\tau)\|_{L^3}^3\,d\tau \leq \|f_0\|_{L^1} + C\|f_0\|_{L^1}^2=:C_*.
\label{eqn: upper bound for I(t)}
\eeq
Note that $C_*$ is a constant only depending on $\|f_0\|_{L^1}$.
Actually, when deriving \eqref{eqn: upper bound for I(t)}, we first obtained
\[
I(t)\leq I(\d)+ \int_\d^t \|f(\cdot,\tau)\|_{L^3}^3\,d\tau,
\]
and then sent $\d \to 0^+$.
\end{step}

\begin{step}[The choice of $W$]
%Now we choose $W$ as follows.
Summarizing our assumptions on $V$, we need
\begin{itemize}
\item $V$ is increasing and convex on $[0,+\infty)$.

\item  For $y\in [0,\Lam]$, $V^2(y) \leq  y^3$.

\item For $y\in (\Lam,+\infty)$,
\beq
\f{V'(y)}{V(y)} \leq
\inf_{t\in [0,T]}
t^{\f12}\left(C_1
\big|S_\Lam(t)\big|\right)^{-1} = : \b.
\label{eqn: constraint on V for large y}
\eeq
%  \item $V(y)\leq C W(y)$ on $[0,\Lam]$ for some universal constant $C>0$.
\end{itemize}
We remark that $\b$ admits an a priori lower bound that only depends on $\Lam$ and $\|f_0\|_{L^1}$.
In fact, by the energy estimate \eqref{eqn: energy estimate} and interpolation,
\[
\|f\|_{L^4_t L^2} \leq C \|f_0\|_{L^1}^{\f34}.
\]
%Here and in what follows, $C(\|f_0\|_{L^1})$ denotes a constant only depending on $\|f_0\|_{L^1}$, which may change from line to line.
Since $\|f(\cdot,t)\|_{L^2}$  is non-increasing in time (see Lemma \ref{lem: energy bound and decay of L^p norms}),
this implies
\[
\|f(\cdot,t)\|_{L^2}\leq C\|f_0\|_{L^1}^\f34 t^{-\f14},
\]
which gives
\[
\big|S_\Lam(t)\big| =
\big|\big\{x\in \BT:\, f(x,t) > \Lambda/t^{1/2}\big\}\big|\leq C \|f_0\|_{L^1}^{\f32} t^{-\f12}\cdot \f{t}{\Lam^2}.
\]
Therefore,
\beq
\b= \inf_{t\in [0,T]}
t^{\f12}\left(C_1
\big|S_\Lam(t)\big|\right)^{-1} \geq C_2 \|f_0\|_{L^1}^{-\f32} \Lam^2,
\label{eqn: lower bound for exponential growth rate}
\eeq
where $C_2$ is a universal constant.

We define $V$ to be the unique continuous function on $[0,+\infty)$ that satisfies
\[
V(y) = y^{\f32} \mbox{ on }[0,\Lam],
\quad \mbox{and}\quad
V'(y) = \b V(y) \mbox{ on } y\in (\Lam, +\infty),
\]
i.e.,
\beq
V(y) =
\begin{cases}
  y^{3/2}, & \mbox{if } y\in [0,\Lam], \\
  \Lam^{3/2} e^{\b(y-\Lam)}, & \mbox{if } y\in [\Lam, +\infty).
\end{cases}
\label{eqn: formula for V}
\eeq
We shall suitably choose $\Lam$ in order to guarantee that $V$ is increasing and convex.
In fact, we only need $V'(\Lam^-)\leq V'(\Lam^+)$, i.e., $\f32 \Lam^{1/2} \leq \b \Lam^{3/2}$.
By virtue of \eqref{eqn: lower bound for exponential growth rate},
it is safe to take
\beq
\Lam \geq C_3\|f_0\|_{L^1}^{\f12},
\label{eqn: lower bound for admissible Lambda}
\eeq
where $C_3$ is a universal constant.
Without loss of generality, we may assume $C_3\geq 2$ so that \eqref{eqn: lower bound for admissible Lambda} is no weaker than \eqref{eqn: largeness of Lambda}.
%We may take the equality without loss of generality.

With this choice of $V$, we thus define $W$ by
\[
\f{d}{dy}\left(\f{W(y)}{y}\right) = \f{V^2(y)}{y^2},\quad \lim_{y\to 0^+}\f{W(y)}{y} = 1.
\]
For future use, we calculate it explicitly
\beq
W(y) =
\begin{cases}
  y+\f12 y^3, & \mbox{if } y\in [0,\Lam], \\
  y\left( 1+\f12\Lam^2 + \Lam^3 \int_\Lam^y  z^{-2} e^{2\b(z-\Lam)}\,dz \right), & \mbox{if } y\in [\Lam,+\infty).
\end{cases}
\label{eqn: formula for W}
\eeq
This resulting $W$ is convex on $[0,+\infty)$.
Indeed, on $(0,\Lam)\cup (\Lam,+\infty)$,
\[
W''(y) = \f1y \big(W'(y)y - W(y)\big)' = \f{2}{y}V(y)V'(y)\geq 0,
\]
and $W'(\Lam^-) = W'(\Lam^+)$ because $V^2(y)$ and $W(y)$ are continuous at $y = \Lam$.

By virtue of \eqref{eqn: upper bound for I(t)}, for any $W$ constructed in the above way, we have
\[
\int_\BT W\big(t^{1/2}f(x,t)\big)\,dx \leq C_* t^{\f12}.
\]
This implies, for any $\lam>0$,
\[
\big|S_\lam(t)\big| = \big|\big\{x\in \BT:\, f(x,t)\geq \lam/t^{1/2}\big\}\big|\leq C_* t^{\f12} W(\lam)^{-1},
\]
and thus,
\beq
\inf_{t\in [0,T]} t^{\f12} \left(C_1 \big|S_\lam (t) \big|\right)^{-1}\geq (C_1C_*)^{-1} W(\lam).
\label{eqn: new lower bound for beta}
\eeq
Here $C_1$ and $C_*$ are defined in \eqref{eqn: upper bound for large part in the L^2 norm} and \eqref{eqn: upper bound for I(t)}, respectively.
We will treat \eqref{eqn: new lower bound for beta} as an improvement of \eqref{eqn: lower bound for exponential growth rate}.
%In particular,
%\[
%\sup_{t\in [0,t_*]}\big|\big\{x\in \BT:\, f(x,t)\geq \Lam/t\big\}\big|\leq C_* t_* W(\Lam)^{-1}.
%\]
\end{step}

\begin{step}[Bootstrap]
Next we introduce a bootstrap argument.
Let $C_1$, $C_2$, and $C_3$ be the universal constants defined in \eqref{eqn: upper bound for large part in the L^2 norm}, \eqref{eqn: lower bound for exponential growth rate}, and \eqref{eqn: lower bound for admissible Lambda}, respectively.
With $C_0 \geq C_3\geq 2$ to be chosen, we let
\[
\Lam_0 = C_0 \|f_0\|_{L^1}^{\f12},
\quad
\b_0 = C_2 \|f_0\|_{L^1}^{-\f32} \Lam_0^2 = C_2 C_0^2 \|f_0\|_{L^1}^{-\f12}.
\]
Let $W_0$ be defined by \eqref{eqn: formula for W}, with $\Lam$ and $\b$ there replaced by $\Lam_0$ and $\b_0$.
For $j \in \BZ_+\cup \{\infty\}$, let
\[
\Lam_j = \left(C_0+\sum_{k = 1}^j k^{-2}\right) \|f_0\|_{L^1}^{\f12}.
\]
In the view of \eqref{eqn: new lower bound for beta} and also \eqref{eqn: constraint on V for large y}, iteratively define
\beq
\b_j = (C_1C_*)^{-1} W_{j-1}(\Lam_j),
\label{eqn: def of beta_j}
\eeq
while $W_j$ is then constructed by \eqref{eqn: formula for W}, with $\Lam$ and $\b$ replaced by $\Lam_j$ and $\b_j$.
%Recall that $C_1$ is the universal constant defined in \eqref{eqn: upper bound for large part in the L^2 norm}.

Observe that $\{\Lam_j\}_j$ is increasing, and
\beq
C_0 \|f_0\|_{L^1}^\f12 \leq \Lam_j\leq C_0'\|f_0\|_{L^1}^\f12,
\label{eqn: upper and lower bound for Lambda_j}
\eeq
where $C_0'/C_0$ is bounded by some universal constant.
With $C_0$ being suitably large, we claim that, for all $j\in \BZ_+$,
\beq
\b_{j-1}\geq j^3 \|f_0\|_{L^1}^{-\f12}.
\label{eqn: lower bound for beta_j}
\eeq
We prove this by induction.
First, choose $C_0\geq \max\{C_2^{-1/2},C_3\}$, so that \eqref{eqn: lower bound for beta_j} holds for $j = 1$.
For $j\geq 2$, by the definitions of $\Lam_j$ and $W_j$,
we derive by \eqref{eqn: formula for W} that for $j\in \BZ_+$,
\[
%\begin{split}
W_{j-1}(\Lam_j) \geq %&\; \Lam_j\left( 1+\f12\Lam_{j-1}^2 + \Lam_{j-1}^3 \int_{\Lam_{j-1}}^{\Lam_j}  z^{-2} e^{2\b_{j-1}(z-\Lam_{j-1})}\,dz \right)\\
%%
%\geq &\;
\Lam_j\left( 1+\f12\Lam_{j-1}^2 + \f{\Lam_{j-1}^3 }{\Lam_j^2} \int_{0}^{\Lam_j-\Lam_{j-1}} e^{2\b_{j-1}z}\,dz \right).
%
%\geq &\; \Lam_j\left( 1+\f12\Lam_{j-1}^2 + \f{\Lam_{j-1}^3 }{\Lam_j^2} \cdot \f{e^{2\b_{j-1}(\Lam_j-\Lam_{j-1})}-1}{2\b_{j-1}} \right)\\
%\end{split}
\]
By \eqref{eqn: def of beta_j} and \eqref{eqn: upper and lower bound for Lambda_j}, and using the definition of $C_*$ in \eqref{eqn: upper bound for I(t)}, we obtain that
\[
\begin{split}
\b_j \geq &\; (C_1C_*)^{-1} \cdot \f{\Lam_{j-1}^3}{\Lam_j} \cdot \f{e^{2\b_{j-1}j^{-2}\|f_0\|_{L^1}^{1/2}}-1}{2\b_{j-1}}\\
\geq &\; \f{C(C_0\|f_0\|_{L^1}^{1/2})^2 }{\|f_0\|_{L^1} + C\|f_0\|_{L^1}^2} \cdot  j^{-2}\|f_0\|_{L^1}^{1/2} \cdot \f{e^{2\b_{j-1}j^{-2}\|f_0\|_{L^1}^{1/2}}-1}{2\b_{j-1}j^{-2}\|f_0\|_{L^1}^{1/2}},
\end{split}
\]
where $C$'s are universal constants.
Using the induction hypothesis \eqref{eqn: lower bound for beta_j} and the assumption $\|f_0\|_{L^1}\geq 1$,
\[
\b_j \geq  CC_0^2 \cdot j^{-2} \cdot \f{e^{2 j}-1}{2 j} \|f_0\|_{L^1}^{-1/2}
\geq  CC_0^2 (j+1)^3\|f_0\|_{L^1}^{-1/2},
\]
where $C$ is a universal constant.
In the second inequality, we used the trivial fact that for all $j\in \BZ_+$,
$\f{e^{2 j}-1}{2 j}\geq Cj^2(j+1)^3$ for some universal $C>0$.
Now we choose $C_0$ to be suitably large but universal, so that $CC_0^2\geq 1$ in the above estimate.
This yields \eqref{eqn: lower bound for beta_j} with $j$ replaced by $(j+1)$.
Therefore, by induction, \eqref{eqn: lower bound for beta_j} holds for all $j\in \BZ_+$.

Using \eqref{eqn: new lower bound for beta} and \eqref{eqn: def of beta_j}, we find for all $j\in \BZ_+$,
\beqo
\sup_{t\in [0,T]} t^{-\f12} \cdot C_1 \big|\big\{x\in \BT:\, f(x,t)\geq \Lam_\infty /t^{1/2}\big\}\big| \leq \b_j^{-1}.
\eeqo
By \eqref{eqn: lower bound for beta_j} and the arbitrariness of $j$, we conclude that for all $t\in [0,T]$,
\[
f(\cdot,t)\leq \Lam_\infty t^{-\f12}\quad \mbox{a.e.\;on }\BT.
\]
Since
\[
\Lam_\infty = \left(C_0+\sum_{k = 1}^\infty k^{-2}\right) \|f_0\|_{L^1}^{\f12},  %\leq C \|f_0\|_{L^1}^{1/2},
\]
where $C_0>0$ is universal, the desired estimate follows.
\end{step}
\end{proof}
\end{lem}

\section{Smoothing and Decay Estimates}
\label{sec: smoothing and decay}

\subsection{Higher regularity with time integrability at $t = 0$}
\label{sec: estimates for h}
With the estimates established so far, we are still not able to make sense of the integral
\[
\int_{\BT\times [0,T)} \pa_x \va \cdot \CH f \cdot F\,dx\,dt
\]
in the weak formulation \eqref{eqn: weak formulation}.
Indeed, with initial data $f_0, F_0\in L^1(\BT)$, we only have
\[
f \in L^\infty_T L^1(\BT)\cap L^2_T H^{\f12}(\BT)\cap L^{2^-}_T L^\infty(\BT),\mbox{ and }F\in L^\infty_T L^1(\BT),
\]
whereas $H^{1/2}(\BT)\not \hookrightarrow L^\infty(\BT)$, and the Hilbert transform is not bounded in $L^1(\BT)$ either.
A natural idea to resolve this integrability issue is to prove that $f$ has some spatial regularity higher than $H^{1/2}$, with the corresponding norm being time-integrable.
This seems not trivial because of the nonlinearity and degeneracy of the equation \eqref{eqn: the main equation}.
In particular, the very bad lower bound for $f$ in Lemma \ref{lem: lower bound} indicates extremely weak smoothing effect in some part of $\BT$, so the higher-order norms of $f$ could be terribly singular near $t = 0$.
Fortunately, %$\ln f$ enjoys neat estimates.
%Based on that,
by studying the special properties of \eqref{eqn: the main equation},
we can prove that certain H\"{o}lder norms of $f$ are time-integrable near $t=0$ (see Corollary \ref{cor: bound for H1 norm}).
This is the theme of this subsection.

For convenience, we assume $f$ to be a positive smooth solution to \eqref{eqn: the main equation} on $\BT\times [0,+\infty)$.
Define $h:= \ln f$.
%In this section, we will prove a few nice a priori estimates regarding $h$.
%
Clearly, $h$ is a smooth solution to
\beq
\pa_t h = \CH f \cdot \pa_x h - (-\D)^{\f12} f.
\label{eqn: equation for log f}
\eeq

We start from the entropy estimate for $F$. %, which is a positive smooth solution to \eqref{eqn: equation for big F}.
\begin{lem}
\label{lem: entropy estimate}
\[
\f{d}{dt}\int_\BT F \ln F\,dx + \|h \|_{\dot{H}^{\f12}}^2\leq 0.
\]
\begin{proof}
We derive with \eqref{eqn: equation for big F} that
\beq
\begin{split}
\f{d}{dt}\int_\BT F \ln F\,dx = &\; \int_\BT (1+\ln F)\pa_t F\,dx
= -\int_\BT \pa_x (1+\ln F)\cdot \CH f\cdot  F\,dx\\
= &\; -\int_\BT \CH f\cdot  \pa_x F\,dx
= \int_\BT F (-\D)^{\f12} f\,dx.
\end{split}
\label{eqn: entropy estimate original}
\eeq
The right-hand side can be written as
\[
\begin{split}
\int_\BT F (-\D)^{\f12} f\,dx
= &\; \int_\BT F(x)\cdot \f{1}{\pi}\pv \int_\BT \f{f(x)-f(y)}{4\sin^2(\f{x-y}{2})}\,dy\,dx\\
= &\; -\f{1}{2\pi} \int_{\BT\times \BT } \f{(F(x)-F(y))(f(y)-f(x))}{4\sin^2(\f{x-y}{2})}\,dy\,dx.
%\leq 0.
\end{split}
\]
Since
\[
\begin{split}
(F(x)-F(y))(f(y)-f(x)) = &\; \f{(F(x)-F(y))^2}{F(x)F(y)}
= \left(\sqrt{\f{F(x)}{F(y)}} - \sqrt{\f{F(y)}{F(x)}}\right)^2\\
\geq &\; \left(2\ln \sqrt{\f{F(x)}{F(y)}}\right)^2 = \big(\ln F(x) - \ln F(y)\big)^2,
\end{split}
\]
we obtain the desired estimate.
\end{proof}
\end{lem}

Then we state an $H^{1/2}$-estimate for $h$.
\begin{lem}
\label{lem: H 1/2 estimate for h}
\[
\f{d}{dt}\|h\|_{\dot{H}^{\f12}}^2 + 4\big\|\sqrt{f}\big\|_{\dot{H}^1}^2
+ \int_\BT  f \big((-\D)^{\f12} h\big)^2\,dx = 0.
\]
\begin{proof}

%For a function $u$ defined on $\BT$, denote its mean by
%\[
%\bar{u}: = \f1{2\pi}\int_\BT u(x)\,dx.
%\]
%Applying the Cotlar's identity to the mean-zero function $u-\bar{u}$ on $\BT$ yields
%\[
%(\CH u)^2 - \big(u-\bar{u}\big)^2 = 2\CH\big((u-\bar{u})\CH u\big).
%\]
%Hence,
%\beq
%(\CH u)^2 - u^2 -\bar{u}^2  = 2\CH (u\CH u ).
%\label{eqn: Cotlar}
%\eeq

Recall that for a mean-zero smooth function $u$ defined on $\BT$, we have the Cotlar's identity
\beq
(\CH u)^2 - u^2 = 2\CH(u\CH u).
\label{eqn: Cotlar}
\eeq
We calculate with \eqref{eqn: equation for log f} and \eqref{eqn: Cotlar} that
\beq
\begin{split}
\f{d}{dt}\|h\|_{\dot{H}^{\f12}}^2
= &\; 2\int_\BT \pa_t h \cdot (-\D)^{\f12} h\,dx \\
= &\; -2\int_\BT  f \cdot \CH\big(\pa_x h (-\D)^{\f12} h\big)\,dx
-2\int_\BT (-\D)^{\f12}f \cdot (-\D)^{\f12} h\,dx \\
= &\; \int_\BT  e^h \Big[ \big(\pa_x h \big)^2 - \big((-\D)^{\f12} h\big)^2\Big]\,dx
- 2\int_\BT \pa_x \big(e^h\big) \cdot \pa_x h\,dx \\
= &\;  -\int_\BT  f \Big[ \big(\pa_x h \big)^2 + \big((-\D)^{\f12} h\big)^2\Big]\,dx.
\end{split}
\label{eqn: H 1/2 estimate for h derivation}
\eeq
Observing that $f(\pa_x h)^2 = 4(\pa_x \sqrt{f})^2$,
%\[
%f(\pa_x h)^2 = \f{(\pa_x f)^2}{f} = 4(\pa_x \sqrt{f})^2,
%\]
we obtain the desired result.
\end{proof}
\end{lem}

We further show $\|\sqrt{f}\|_{\dot{H}^1}$ satisfies another simple estimate, which implies it is non-increasing as well.

\begin{lem}
\label{lem: H 1 estimate for sqrt f}
\[
\f{d}{dt}\big\| \sqrt{f}\big\|_{\dot{H}^1}^2 + \f12 \|f\|_{\dot{H}^{\f32}}^2 = 0.
\]

\begin{proof}
Denoting $\om = f$ and $u = -\CH f$, we recast \eqref{eqn: the main equation} as %in \eqref{eqn: recast as De Gregorio}
\[
\pa_t \om +u \pa_x \om = \om \pa_x u.
\]
Note that in order to keep $\om$ positive, here we take a different change of variable from that in \eqref{eqn: recast as De Gregorio}.
Hence,
\[
\pa_t \pa_x \sqrt{\om} + u \pa_{xx} \sqrt{\om} + \f12 \pa_x u \pa_x \sqrt{\om}
= \f12 \sqrt{\om} \pa_{xx} u.
\]
Then we derive that
\[
\begin{split}
\f12 \f{d}{dt}\big\|\sqrt{w}\big\|_{\dot{H}^1}^2 = &\;
\int_\BT - u \pa_{xx} \sqrt{\om} \pa_x \sqrt{\om}
- \f12 \pa_x u \pa_x \sqrt{\om} \pa_x \sqrt{\om}
+ \f12 \sqrt{\om} \pa_x \sqrt{\om} \pa_{xx} u \,dx \\
= &\;
\f12 \int_\BT - u \pa_{x} \big( \pa_x \sqrt{\om}\big)^2
- \pa_x u \big(\pa_x \sqrt{\om}\big)^2 + \f12 \pa_x \om \pa_{xx} u\,dx \\
= &\;
\f14 \int_\BT \pa_x \om \pa_{xx} u\,dx.
\end{split}
\]
Plugging in $u = -\CH \om$ yields the desired result.
\end{proof}

\begin{rmk}
For the De Gregorio model \eqref{eqn: De Gregorio}, non-negative solutions $\om = \om(x,t)$ that are sufficiently smooth are known to enjoy conservation of $\|\sqrt{\om}\|_{\dot{H}^1}$ \cite{lei2020constantin}.
Interestingly, we can recover that result by following the derivation above and plugging in $u_x = \CH \om$ (cf.\;\eqref{eqn: De Gregorio}) in the last line.
\end{rmk}
\end{lem}

%\begin{rmk}
By Lemma \ref{lem: lower bound} and Lemma \ref{lem: L inf bound}, $f(\cdot,t)$ is positive and bounded for all positive $t$.
Hence, whenever $t>0$, that $f(\cdot,t)\in H^{1/2}(\BT)$ implies $h(\cdot,t)\in H^{1/2}(\BT)$ and vice versa.
This is because
\beq
\|f\|_{\dot{H}^{\f12}}^2 = \f1{2\pi}\int_{\BT\times \BT} \f{(f(x)-f(y))^2}{4\sin^2(\f{x-y}{2})}\,dx\,dy
= \f1{2\pi}\int_{\BT\times \BT} \f{(e^{h(x)}-e^{h(y)})^2}{4\sin^2(\f{x-y}{2})}\,dx\,dy,
\label{eqn: H 1/2 semi norm of f}
\eeq
and
\beq
\min_\BT f^2 \cdot (h(x)-h(y))^2
\leq (e^{h(x)}-e^{h(y)})^2
\leq \max_\BT f^2\cdot (h(x)-h(y))^2.
\label{eqn: equivalence of H 1/2 norm of f and h}
\eeq
Hence, %by Lemma \ref{lem: lower bound},
\beq
\|h(\cdot,t)\|_{\dot{H}^{\f12}(\BT)}
\leq
\|F(\cdot,t)\|_{L^\infty} \|f(\cdot,t)\|_{\dot{H}^{\f12}(\BT)}.
%<+\infty.
\label{eqn: naive bound for H 1/2 norm of h}
\eeq
On the other hand, by Lemma \ref{lem: energy bound and decay of L^p norms}, for any $t>0$, there exists $t_*\in [\f{t}{2}, t]$, such that $\|f(\cdot,t_*)\|_{\dot{H}^{1/2}(\BT)}\leq Ct^{-1/2}\|f_0\|_{L^1}^{1/2}$.
Combining this with \eqref{eqn: naive bound for H 1/2 norm of h}, Lemma \ref{lem: lower bound}, and Lemma \ref{lem: H 1/2 estimate for h}, we find
\beqo
\|h(\cdot,t)\|_{\dot{H}^{\f12}}
\leq
\|h(\cdot,t_*)\|_{\dot{H}^{\f12}}
\leq
Ct^{-\f12} \|f_0\|_{L^1}^{\f12} \|F_0\|_{L^1}
\exp\left[\coth\left(C\|F_0\|_{L^1}^{-1} t\right)\right].
%\label{eqn: naive bound for H 1/2 norm of h precise form}
\eeqo
%As is mentioned before, t
This bound is too singular near $t = 0$ for future analysis.
In what follows, we establish an improved estimate by applying Lemmas \ref{lem: entropy estimate}-\ref{lem: H 1 estimate for sqrt f}.

\begin{prop}
\label{prop: improved H 1/2 estimate for h}
For any $t>0$,
\beq
\|h(\cdot,t)\|_{\dot{H}^{\f12}}
\leq
Ct^{-\f12}\|F_0\|_{L^1}^{\f12}\left[ \coth\left(\f{2}{\pi}\|F_0\|_{L^1}^{-1} t\right) \right]^{\f12},
\label{eqn: improved estimate for H 1/2 norm of h t small}
\eeq
and for any $0<t\leq \|f_0\|_{L^1}^{-1}$,
\beq
\|f(\cdot,t)\|_{\dot{H}^1} \leq
Ct^{-\f74}\|f_0\|_{L^1}^{\f14}\|F_0\|_{L^1},
\label{eqn: bound for H1 norm}
\eeq
where $C$'s are universal constants.
\begin{rmk}
\label{rmk: sharpness of H1/2 bound for h}
These estimates are good enough for future analysis, though we do not know whether they are sharp in terms of the singularity at $t = 0$.
%By adapting the proof below, one can improve the bound for $\|h(\cdot,t)\|_{\dot{H}^{1/2}}$ when $t>\|f_0\|_{L^1}^{-1}$. %which is better than what we stated above. % in a similar manner. %; see \eqref{eqn: main H 1/2 estimate for h}.
However, the large-time decay of \eqref{eqn: improved estimate for H 1/2 norm of h t small} is not optimal.
We will prove exponential decay of $h$ and $f$ in Section \ref{sec: global higher-order estimates} and Section \ref{sec: analyticity and sharp decay estimate}.
%We only mention the following naive bound: $\|h(\cdot,t)\|_{\dot{H}^{1/2}}$ is non-increasing in $t$ according to Lemma \ref{lem: H 1/2 estimate for h}, so
\end{rmk}

\begin{proof}
By Lemma \ref{lem: energy bound and decay of L^p norms} and Lemma \ref{lem: lower bound}, for all $t>0$,
\[
\int_\BT F\ln F\,dx \leq \|F_0\|_{L^1}\left[\ln \left( \f18\|F_0\|_{L^1}\right)
+\coth\left(\f{4}{\pi}\|F_0\|_{L^1}^{-1} t\right) \right].
\]
On the other hand, since $x\mapsto x\ln x$ is convex on $(0,+\infty)$, by Jensen's inequality,
\[
\int_\BT F\ln F\,dx \geq \|F_0\|_{L^1} \ln \left(\f{1}{2\pi}\|F_0\|_{L^1}\right).
\]

Lemma \ref{lem: H 1/2 estimate for h} shows that $\|h\|_{\dot{H}^{1/2}}$ is non-increasing.
Hence, by Lemma \ref{lem: entropy estimate}, for any $t>0$,
\[
\begin{split}
\|h(\cdot,t)\|_{\dot{H}^{\f12}}^2
\leq &\;Ct^{-1}\left[\int_\BT F\ln F\left(x,\f{t}{2}\right)\,dx -\|F_0\|_{L^1} \ln \left(\f{1}{2\pi}\|F_0\|_{L^1}\right)\right] \\
\leq &\;Ct^{-1}\|F_0\|_{L^1} \coth\left(\f{2}{\pi}\|F_0\|_{L^1}^{-1} t\right).
\end{split}
\]
This gives \eqref{eqn: improved estimate for H 1/2 norm of h t small}.
Similarly, since Lemma \ref{lem: H 1 estimate for sqrt f} shows $\|\sqrt{f}\|_{\dot{H}^1}$ is non-increasing, by Lemma \ref{lem: H 1/2 estimate for h}, we find for any $t> 0$,
\beq
\big\|\sqrt{f(\cdot,t)}\big\|_{\dot{H}^1}^2 \leq Ct^{-1}\|h(\cdot,t/2)\|_{\dot{H}^{\f12}}^2
\leq
Ct^{-2}\|F_0\|_{L^1} \coth\left(\f{1}{\pi}\|F_0\|_{L^1}^{-1} t\right).
\label{eqn: bounding H1 norm of root f}
\eeq
Since
\beq
\|f\|_{\dot{H}^1}
\leq C\big\|\sqrt{f}\big\|_{L^\infty} \big\|\sqrt{f}\big\|_{\dot{H}^1}
= C\|f\|_{L^\infty}^{\f12}\big\|\sqrt{f}\big\|_{\dot{H}^1},
\label{eqn: bounding H^1 norm of f using H^1 norm of sqrt f}
\eeq
we apply Lemma \ref{lem: L inf bound} and \eqref{eqn: bounding H1 norm of root f} to derive that, for $0<t\leq \|f_0\|_{L^1}^{-1}$,
\beqo
\|f(\cdot,t)\|_{\dot{H}^1} \leq
%\begin{cases}
Ct^{-\f54}\|f_0\|_{L^1}^{\f14}\|F_0\|_{L^1}^{\f12}\left[\coth\left(\f{1}{\pi}\|F_0\|_{L^1}^{-1} t\right) \right]^{\f12}.
\eeqo
In this case, thanks to the Cauchy-Schwarz inequality, % and the assumption $t\leq \|f_0\|_{L^1}^{-1}$,
\beqo
t^{-1}\|F_0\|_{L^1}\geq \|f_0\|_{L^1}\|F_0\|_{L^1} \geq 4\pi^2,
%\label{eqn: lower bound for L1 norm of F_0 divided by t}
\eeqo
so  $t\leq \f{1}{4\pi^2}\|F_0\|_{L^1}$.
Then we can further simply the above estimate to obtain \eqref{eqn: bound for H1 norm}.
\end{proof}
\end{prop}

\begin{cor}
\label{cor: bound for H1 norm}
For $0<t \leq \|f_0\|_{L^1}^{-1}$
and $\al\in (0,\f15)$, with $f_\infty$ defined in \eqref{eqn: equilibrium},
\beq
\int_0^t \|f(\cdot,\tau)\|_{\dot{C}^{\al}}\,d\tau
\leq C_\al t^{\f{1-5\al}{2}} \big(\|f_0\|_{L^1}-2\pi f_\infty\big)^{\f{1-2\al}{2}}
\|f_0\|_{L^1}^{\f{\al}{2}} \|F_0\|_{L^1}^{2\al},
\label{eqn: time integral estimate of some Holder norm}
\eeq
where $C_\al>0$ is a universal constant depending on $\al$.
To be simpler, we have
\[
\int_0^t \|f(\cdot,\tau)\|_{\dot{C}^{\al}}\,d\tau
\leq C_\al t^{\f{1-5\al}{2}} \|f_0\|_{L^1}^{\f{1-\al}{2}}
\|F_0\|_{L^1}^{2\al}.
\]

\begin{proof}
With $\al\in (0,\f15)$ and $\d>0$ to be determined, by Sobolev embedding, interpolation, and the Young's inequality,
\[
\begin{split}
\|f(\cdot,t)\|_{\dot{C}^{\al}}
\leq &\; C_\al \|f(\cdot,t)\|_{\dot{H}^{\f12+\al}}
\leq C_\al \|f(\cdot,t)\|_{\dot{H}^{\f12}}^{1-2\al}\|f(\cdot,t)\|_{\dot{H}^1}^{2\al}\\
\leq &\; 2\d \|f(\cdot,t)\|_{\dot{H}^{\f12}}^2+ C_\al \d^{-\f{1-2\al}{1+2\al}} \|f(\cdot,t)\|_{\dot{H}^1}^{\f{4\al}{1+2\al}}.
\end{split}
\]
Here $C_\al>0$ is a universal constant depending only on $\al$.
Hence, by Lemma \ref{lem: energy bound and decay of L^p norms} and \eqref{eqn: bound for H1 norm}, for $0<t \leq \|f_0\|_{L^1}^{-1}$,
\[
\begin{split}
\int_0^t \|f(\cdot,\tau)\|_{\dot{C}^{\al}}\,d\tau
\leq &\; 2\d \int_0^t \|f(\cdot,\tau)\|_{\dot{H}^{\f12}}^2\,d\tau
+ C_\al \d^{-\f{1-2\al}{1+2\al}} \int_0^t \|f(\cdot,\tau)\|_{\dot{H}^1}^{\f{4\al}{1+2\al}}\,d\tau\\
\leq &\; \d \big(\|f_0\|_{L^1}-2\pi f_\infty\big)
+ C_\al \d^{-\f{1-2\al}{1+2\al}}
\|f_0\|_{L^1}^{\f{\al}{1+2\al}}
\|F_0\|_{L^1}^{\f{4\al}{1+2\al}}
\int_0^{t}
\tau^{-\f{7\al}{1+2\al}}\,d\tau.
\end{split}
\]
The last integral is finite as long as $\al\in (0,\f15)$.
Calculating the integral and optimizing over $\d>0$ yields
\eqref{eqn: time integral estimate of some Holder norm}.
\end{proof}
\begin{rmk}
By better incorporating the estimate for $\|f\|_{\dot{H}^{3/2}}$ from Lemma \ref{lem: H 1 estimate for sqrt f},
one may slightly improve the power of $t$ in \eqref{eqn: time integral estimate of some Holder norm} to become $t^{\f{1-5\al}{2(1-2\al)}}$, up to corresponding changes in the other factors.
However, we omit the details as that is not essential for the future analysis.
\end{rmk}
\end{cor}

\subsection{Instant smoothing and global bounds}
\label{sec: global higher-order estimates}
Now we derive a priori estimates that show the solution $f$ would instantly become smooth, and that it converges exponentially to the constant $f_\infty$ (see \eqref{eqn: equilibrium}) as $t\to +\infty$.
We start with a higher-order generalization of Lemma \ref{lem: H 1/2 estimate for h} and Lemma \ref{lem: H 1 estimate for sqrt f} for positive smooth solutions.
\begin{lem}
\label{lem: generalized H k+1/2 estimate for h}
Suppose $f = f(x,t)$ is a positive smooth solution on $\BT\times [0,+\infty)$.
Let $f_*(t)$ be defined in Lemma \ref{lem: lower bound}.
For any $k\in \BZ_+$ and any $t>0$,
\beq
\f{d}{dt}\|h\|_{\dot{H}^{k+\f12}}^2
+ f_*(t)\|h\|_{\dot{H}^{k+1}}^2
\leq \| h \|_{\dot{H}^{k+\f12}}^2\cdot
C_k \|f\|_{L^\infty}^2 \|F\|_{L^\infty}^2\big\|\sqrt{f} \big\|_{\dot{H}^{1}}^2
\sum_{m = 0}^{k-1} \| h \|_{\dot{H}^{\f12}}^{2m},
\label{eqn: generalized H k+1/2 estimate for h}
\eeq
and
\beq
\f{d}{dt}\big\|F^{\f12}\pa_x^k f\big\|_{L^2}^2 + \|f\|_{\dot{H}^{k+\f12}}^2
\leq \big\|F^{\f12}\pa_x^k f\big\|_{L^2}^2 \cdot C_k \mathds{1}_{\{k>1\}} \|f\|_{L^\infty}^2\|F\|_{L^\infty}^2 \big\|\sqrt{f}\big\|_{\dot{H}^1}^2,
\label{eqn: generalized weighted H k estimate for f}
\eeq
where $C_k$'s are universal constants only depending on $k$.
\begin{proof}
We first prove \eqref{eqn: generalized H k+1/2 estimate for h}.
With $k\in \BZ_+$, we differentiate \eqref{eqn: equation for log f} in $x$ for $k$-times, integrate it against $ (-\D)^{1/2}\pa_x^k  h$, and derive as in Lemma \ref{lem: H 1/2 estimate for h} that
\beq
\begin{split}
\f12 \f{d}{dt}\|h\|_{\dot{H}^{k+\f12}}^2
%
%= &\; \int_\BT \pa_x^{k} \pa_t h \cdot (-\D)^{\f12} \pa_x^{k} h \,dx\\
= &\; \int_\BT \pa_x^k\big(\CH f \pa_x h \big) \cdot (-\D)^{\f12} \pa_x^{k} h \,dx
-\int_\BT (-\D)^{\f12}\pa_x^k f \cdot (-\D)^{\f12} \pa_x^k h\,dx \\
= &\; -\int_\BT   f \cdot\CH\left[ \pa_x^{k+1} h \cdot (-\D)^{\f12} \pa_x^{k} h \right] dx
-\int_\BT \pa_x^{k+1} \big(e^h\big) \cdot \pa_x^{k+1} h\,dx \\
&\;+ \int_\BT \left[\pa_x^k\big(\CH f \pa_x h \big)- \CH f \pa_x^{k+1} h \right] (-\D)^{\f12} \pa_x^{k} h \,dx\\
= &\; \f12 \int_\BT   f \left[ \big(\pa_x^{k+1} h\big)^2 - \big( (-\D)^{\f12} \pa_x^{k} h \big)^2\right] dx
-\int_\BT e^h \pa_x^{k+1} h\cdot \pa_x^{k+1} h\,dx \\
&\;+ \int_\BT \left[\pa_x^k\big(\CH f \pa_x h \big)- \CH f \pa_x^{k+1} h \right] (-\D)^{\f12} \pa_x^{k} h \,dx\\
&\; -\int_\BT \left[\pa_x^{k} \big(e^h \pa_x h\big)-e^h \pa_x^{k+1} h\right] \pa_x^{k+1} h\,dx\\
=: &\;- \f12 \int_\BT   f \left[ \big(\pa_x^{k+1} h\big)^2 + \big( (-\D)^{\f12} \pa_x^{k} h \big)^2\right] dx  + \CR_k(t),
\end{split}
\label{eqn: time derivative of H k+1/2 norm square of h}
\eeq
where
\[
\CR_k(t) :=\int_\BT \left[\pa_x^k\big(\CH f \pa_x h \big)- \CH f \pa_x^{k+1} h \right] (-\D)^{\f12} \pa_x^{k} h - \left[\pa_x^{k} \big(e^h \pa_x h\big)-e^h \pa_x^{k+1} h\right] \pa_x^{k+1} h\,dx.
\]
In the third equality above, we used the Cotlar's identity \eqref{eqn: Cotlar}.

We bound $\CR_k$ as follows.
\[
\begin{split}
|\CR_k|
\leq &\; C_k \sum_{n = 1}^k
\int_\BT  \big|\pa_x^n \CH f \big| \big|\pa_x^{k+1-n} h \big| \big|(-\D)^{\f12} \pa_x^{k} h\big| + \big|\pa_x^n f\big| \big| \pa_x^{k+1-n} h\big| \big|\pa_x^{k+1} h\big| \,dx\\
\leq &\; C_k  \big\|\pa_x^{k+1} h\big\|_{L^2}
\sum_{n = 1}^k \big\|\pa_x^{k+1-n} h \big\|_{L^{{\f{2(k+1)}{k+1-n}}}}
\big\|\pa_x^n f \big\|_{L^{\f{2(k+1)}{n}}} ,
\end{split}
\]
where $C_k>0$ is a universal constant depending on $k$.
Since
\beq
\pa_x^n f = \pa_x^n \big(e^h\big) = e^h \sum_{m = 1}^n\sum_{a_1+\cdots+ a_m = n\atop a_i\in \BZ_+} C_{a_1,\cdots, a_m}\prod_{i = 1}^m\pa_x^{a_i} h,
\label{eqn: derivative of f}
\eeq
with $C_{a_1,\cdots, a_m}$ depending on $(a_1,\cdots, a_m)$,
we have that
\beq
\begin{split}
|\CR_k|
\leq &\; C_k \big\|\pa_x^{k+1} h\big\|_{L^2} \\
&\;\cdot \sum_{n = 1}^k
\big\|\pa_x^{k+1-n} h \big\|_{L^{{\f{2(k+1)}{k+1-n}}}} \cdot \|f\|_{L^\infty}\sum_{m = 1}^n\sum_{a_1+\cdots+ a_m = n\atop a_i\in \BZ_+} \prod_{i = 1}^m\big\|\pa_x^{a_i} h \big\|_{L^{\f{2(k+1)}{a_i}}}.
\end{split}
\label{eqn: bound for R_k prelim}
\eeq
By Sobolev embedding and interpolation, for $l \in \{1,\cdots, k\}$, % \textcolor{red}{(add ref)},
\[
\big\|\pa_x^{l} h \big\|_{L^{\f{2(k+1)}{l}}}
\leq C_k
\| h \|_{\dot{H}^{l+\f12-\f{l}{2(k+1)}}}
\leq C_k
\left(\| h \|_{\dot{H}^{k+\f12}}
\| h \|_{\dot{H}^{1}}\right)^{\f{l}{k+1}}
\| h \|_{\dot{H}^{\f12}}^{1-\f{2l}{k+1}}.
\]
Hence, \eqref{eqn: bound for R_k prelim} becomes
\beqo
\begin{split}
|\CR_k|
\leq &\; C_k \|f\|_{L^\infty}  \big\| \pa_x^{k+1} h\big\|_{L^2}  \sum_{m = 1}^{k}
\| h \|_{\dot{H}^{k+\f12}} \| h \|_{\dot{H}^{1}}
\| h \|_{\dot{H}^{\f12}}^{m-1}\\
\leq &\; \f12 f_*(t) \big\|\pa_x^{k+1} h\big\|_{L^2}^2
+ \| h \|_{\dot{H}^{k+\f12}}^2\cdot
C_k f_*(t)^{-1}\|f\|_{L^\infty}^2 \| h\|_{\dot{H}^{1}}^2
\sum_{m = 1}^{k} \| h \|_{\dot{H}^{\f12}}^{2(m-1)}.
\end{split}
\eeqo
Combining this with \eqref{eqn: time derivative of H k+1/2 norm square of h} and using the fact $\|h\|_{\dot{H}^1}\leq Cf_*(t)^{-1/2}\|\sqrt{f}\|_{\dot{H}^1}$, we obtain \eqref{eqn: generalized H k+1/2 estimate for h}.
We also remark that, if $k = 0$ in \eqref{eqn: time derivative of H k+1/2 norm square of h}, we obtain $\CR_0(t) = 0$ and thus the estimate in Lemma \ref{lem: H 1/2 estimate for h}.

The proof of \eqref{eqn: generalized weighted H k estimate for f} is similar.
With $k\in \BZ_+$, %we derive as in Lemma \ref{lem: H 1 estimate for sqrt f} that
\beq
\begin{split}
&\;\f{d}{dt}\int_\BT F\big(\pa_x^k f\big)^2\,dx\\
= &\; \int_\BT \pa_x \big( \CH f \cdot F\big) \big(\pa_x^k f\big)^2\,dx
+ 2 \int_\BT F \pa_x^k f\cdot \pa_x^k \big(\CH f \cdot \pa_x f - f(-\D)^{\f12} f\big)\,dx\\
= &\; \int_\BT \pa_x \big( \CH f \cdot F\big) \big(\pa_x^k f\big)^2\,dx
+ 2 \int_\BT F \pa_x^k f \cdot \big(\CH f \cdot \pa_x \pa_x^k f - f(-\D)^{\f12} \pa_x^k f\big)\,dx\\
&\;
+ 2 \int_\BT F \pa_x^k f
\left[\pa_x^k \big(\CH f \cdot \pa_x f - f(-\D)^{\f12} f\big) - \big(\CH f \cdot \pa_x \pa_x^k f - f(-\D)^{\f12} \pa_x^k f\big)\right] dx\\
= &\; \int_\BT \pa_x \big( \CH f \cdot F\big) \big(\pa_x^k f\big)^2 + F \cdot \CH f \cdot \pa_x \big(\pa_x^k f \big)^2
- 2 \pa_x^k f \cdot (-\D)^{\f12} \pa_x^k f \,dx + \tilde{\CR}_k\\
= &\; - 2 \|f\|_{\dot{H}^{k+\f12}}^2 + \tilde{\CR}_k,
\end{split}
\label{eqn: time derivative of weighted H k norm of f}
\eeq
where
\[
\tilde{\CR}_k(t) := 2 \int_\BT F \pa_x^k f
\left[\pa_x^{k-1} \big(\CH f \cdot \pa_x \pa_x f - f(-\D)^{\f12} \pa_x f\big) - \big(\CH f \cdot \pa_x \pa_x^k f - f(-\D)^{\f12} \pa_x^k f\big)\right] dx.
\]
Deriving analogously as before, we can show that
\[
\begin{split}
\big|\tilde{\CR}_k\big|
\leq &\;
C_k \|F\|_{L^\infty} \int_\BT \big|\pa_x^k f\big|
\sum_{n = 1}^{k-1}
\left( \big|\pa_x^n \CH f \cdot \pa_x^{k-1-n}\pa_x \pa_x f\big|
+ \big|\pa_x^{n}f \cdot (-\D)^{\f12} \pa_x^{k-1-n} \pa_x f\big|\right) dx\\
\leq &\;
C_k \|F\|_{L^\infty} \big\|\pa_x^k f\big\|_{L^2}
\sum_{n = 1}^{k-1}
\big\|\pa_x^n f\big\|_{L^{\f{2(k+1)}{n}}} \big\|\pa_x^{k+1-n} f\big\|_{L^{\f{2(k+1)}{k+1-n}}}.
\end{split}
\]
By Sobolev embedding and interpolation, for all $n\in \{1,\cdots, k-1\}$,
\[
\begin{split}
\big\|\pa_x^n f\big\|_{L^{\f{2(k+1)}{n}}} \big\|\pa_x^{k+1-n} f\big\|_{L^{\f{2(k+1)}{k+1-n}}}
\leq &\;C_k \|f\|_{\dot{H}^{n+\f12-\f{n}{2(k+1)}}}
\|f\|_{\dot{H}^{(k+1-n)+\f12-\f{k+1-n}{2(k+1)}}}\\
\leq &\;C_k \|f\|_{\dot{H}^{k+\f12}} \|f\|_{\dot{H}^1}.
\end{split}
\]
Combining these estimates with \eqref{eqn: time derivative of weighted H k norm of f} yields
\beq
\f{d}{dt}\big\|F^{\f12}\pa_x^k f\big\|_{L^2}^2 + 2 \|f\|_{\dot{H}^{k+\f12}}^2
\leq C_k \mathds{1}_{\{k>1\}}\|F\|_{L^\infty} \big\|\pa_x^k f\big\|_{L^2}\|f\|_{\dot{H}^{k+\f12}} \|f\|_{\dot{H}^1},
\label{eqn: estimate for weighted H^k norm of f prelim}
\eeq
and thus, by Young's inequality,
\[
%\begin{split}
\f{d}{dt}\big\|F^{\f12}\pa_x^k f\big\|_{L^2}^2 + \|f\|_{\dot{H}^{k+\f12}}^2
\leq %&\;
\big\|\pa_x^k f\big\|_{L^2}^2 \cdot C_k \mathds{1}_{\{k>1\}} \|F\|_{L^\infty}^2\|f\|_{\dot{H}^1}^2. %\\
\]
Then \eqref{eqn: generalized weighted H k estimate for f} follows.
Note that \eqref{eqn: estimate for weighted H^k norm of f prelim} with $k=1$ exactly corresponds to the estimate in Lemma \ref{lem: H 1 estimate for sqrt f}.
\end{proof}
\end{lem}

The quantities involved in the estimates in Lemma \ref{lem: generalized H k+1/2 estimate for h} can be linked as follows.

\begin{lem}
\label{lem: connecting different quantities}
For any $k\in \BZ_+$ and any $t>0$,
\[
\|h\|_{\dot{H}^{k+\f12}}
\leq \|f\|_{\dot{H}^{k+\f12}}\cdot
C_k \left[ \|f\|_{L^\infty}^{\f12}\|F\|_{L^\infty}^{\f12}
+ \sum_{n = 1}^{k-1}
\|h\|_{\dot{H}^{\f12}}^{2n}
\|f\|_{L^\infty}^n \|F\|_{L^\infty}^n\right] f_*(t)^{-\f32}
\big\|\sqrt{f}\big\|_{\dot{H}^1} ,
\]
and
\[
\big\|F^{\f12}\pa_x^k f\big\|_{L^2}
\leq \|h\|_{\dot{H}^k}\cdot C_k \|f\|_{L^\infty}^{\f12}  \sum_{m = 0}^{k-1}
\|h\|_{\dot{H}^{\f12}}^{m},
\]
where $C_k$'s are universal constants only depending on $k$.

\begin{proof}
For $k\in \BZ_+$,
\[
\begin{split}
\|h\|_{\dot{H}^{k+\f12}}
= &\;\big\|\pa_x^{k-1}\big(F\pa_x f\big)\big\|_{\dot{H}^{\f12}}\\
%\leq &\;C_k \sum_{n = 0}^{k-1} \big\|\pa_x^{n}\big(f^{-1}\big) \pa_x^{k-n} f\big\|_{\dot{H}^{\f12}}\\
\leq &\;
C_k \big\|F\pa_x^{k} f\big\|_{\dot{H}^{\f12}}
+ C_k \sum_{n = 1}^{k-1} \sum_{m=1}^n\sum_{a_1+\cdots+a_m = n \atop a_j\in \BZ_+}\left\|\left(F^{m+1}\prod_{j = 1}^m \pa_x^{a_j} f\right)  \pa_x^{k-n} f\right\|_{\dot{H}^{\f12}}\\
\leq &\;
C_k \|F\|_{\dot{H}^1} \|f\|_{\dot{H}^{k+\f12}}\\
&\;
+ C_k \sum_{n = 1}^{k-1} \sum_{m=1}^n\sum_{a_1+\cdots+a_m = n \atop a_j\in \BZ_+} \|F\|_{\dot{H}^{\f12+\f{1}{4(m+1)}}}^{m+1}
\left(\prod_{j = 1}^m \big\|\pa_x^{a_j} f\big\|_{\dot{H}^{\f12+\f{1}{4m}}}\right)
\big\| \pa_x^{k-n} f\big\|_{\dot{H}^{\f12}}.
\end{split}
\]
Here we used the inequality that, for any $s_1>\f12$ and $s_2>0$ with $s_1\geq s_2$,
\[
\|fg\|_{\dot{H}^{s_2}}\leq C_{s_1,s_2} \|f\|_{\dot{H}^{s_1}}\|g\|_{\dot{H}^{s_2}}.
\]
Since $a_j\in [1,k-1]$, we derive with interpolation and the condition $\sum_j a_j = n$ that
\[
\begin{split}
&\; \|F\|_{\dot{H}^{\f12+\f{1}{4(m+1)}}}^{m+1}
\left(\prod_{j = 1}^m \big\|\pa_x^{a_j} f\big\|_{\dot{H}^{\f12+\f{1}{4m}}}\right)
\big\| \pa_x^{k-n} f\big\|_{\dot{H}^{\f12}}\\
\leq &\;
\|F\|_{\dot{H}^{\f12}}^{m+\f12} \|F\|_{\dot{H}^1}^{\f12}
\left(\prod_{j = 1}^m \|f \|_{\dot{H}^{\f12}}^{1-\f{1}{2m}-\f{a_j}{k}}\|f\|_{\dot{H}^1}^{\f{1}{2m}}
\|f\|_{\dot{H}^{k+\f12}}^{\f{a_j}{k}} \right) \|f\|_{\dot{H}^{\f12}}^{\f{n}{k}}\|f\|_{\dot{H}^{k+\f12}}^{1-\f{n}{k}}\\
\leq &\;
\left(f_*(t)^{-1} \|h\|_{\dot{H}^{\f12}}\right)^{m+\f12}
\left(f_*(t)^{-2} \|f\|_{\dot{H}^1}\right)^{\f12}
\|f \|_{\dot{H}^{\f12}}^{m-\f{1}{2}}\|f\|_{\dot{H}^1}^{\f{1}{2}}
\|f\|_{\dot{H}^{k+\f12}}\\
\leq &\; f_*(t)^{-(m+\f32)} \|h\|_{\dot{H}^{\f12}}^{2m}
\|f \|_{L^\infty}^{m-\f{1}{2}}\|f\|_{\dot{H}^1}
\|f\|_{\dot{H}^{k+\f12}}.
\end{split}
\]
In the second inequality, we used the definition of $H^{1/2}$-semi-norm and the fact $F = e^{-h}$ (cf.\;\eqref{eqn: H 1/2 semi norm of f} and \eqref{eqn: equivalence of H 1/2 norm of f and h}) to obtain $\|F\|_{\dot{H}^{1/2}}\leq f_*(t)^{-1} \|h\|_{\dot{H}^{1/2}}$.
In the last inequality, we used
$\|f\|_{\dot{H}^{1/2}} \leq \|f\|_{L^\infty}\|h\|_{\dot{H}^{1/2}}$ that is also derived from \eqref{eqn: equivalence of H 1/2 norm of f and h}.
Hence,
\[
%\begin{split}
\|h\|_{\dot{H}^{k+\f12}}
\leq %&\;
C_k \left[f_*(t)^{-2} %+ \|f\|_{\dot{H}^1} \|f\|_{\dot{H}^{k+\f12}}\\
%&\;+ C_k
+ \sum_{n = 1}^{k-1} \sum_{m=1}^n
f_*(t)^{-(m+\f32)} \|h\|_{\dot{H}^{\f12}}^{2m}
\|f \|_{L^\infty}^{m-\f{1}{2}} \right]
\|f\|_{\dot{H}^1} \|f\|_{\dot{H}^{k+\f12}}. %\\
\]
This implies the first desired inequality.

On the other hand, thanks to \eqref{eqn: derivative of f}, we derive using Sobolev embedding and interpolation that
\[
\begin{split}
\big\|F^{\f12}\pa_x^k f\big\|_{L^2}
= &\; \left\|e^{\f12 h} \sum_{m = 1}^k \sum_{a_1+\cdots+ a_m = k\atop a_i\in \BZ_+} C_{a_1,\cdots, a_m}\prod_{i = 1}^m\pa_x^{a_i} h \right\|_{L^2}\\
\leq &\; C_k \big\|\sqrt{f}\big\|_{L^\infty} \sum_{m = 1}^k \sum_{a_1+\cdots+ a_m = k\atop a_i\in \BZ_+} \prod_{i = 1}^m \big\|\pa_x^{a_i} h \big\|_{L^{\f{2k}{a_i}}}\\
\leq &\; C_k \|f\|_{L^\infty}^{\f12} \sum_{m = 1}^k \sum_{a_1+\cdots+ a_m = k\atop a_i\in \BZ_+} \prod_{i = 1}^m
\|h\|_{\dot{H}^{a_i+\f12-\f{a_i}{2k}}}\\
\leq &\; C_k \|f\|_{L^\infty}^{\f12} \sum_{m = 1}^k
\|h\|_{\dot{H}^{\f12}}^{m-1} \|h\|_{\dot{H}^k}.
\end{split}
\]
This proves the second desired estimate.
\end{proof}
\end{lem}

Now we may bootstrap with
Lemma \ref{lem: H 1/2 estimate for h}, Lemma \ref{lem: generalized H k+1/2 estimate for h}, and Lemma \ref{lem: connecting different quantities} to prove finiteness of $\|f\|_{\dot{H}^{k}}$ for all $k\in \BZ_+$ at all positive times, as well as their exponential decay as $t\to \infty$.

\begin{prop}
\label{prop: boundedness and decay of H^k norms of h}
Suppose $f = f(x,t)$ is a positive smooth solution on $\BT\times [0,+\infty)$.
For all $k\in \BZ_+$ and $t>0$,
\beq
%\begin{split}
%&
\|h(\cdot,t)\|_{\dot{H}^{k+\f12}},\,\big\|F^{\f12}\pa_x^k f(\cdot,t)\big\|_{L^2} \leq \CC_k(t,\|f_0\|_{L^1},\|F_0\|_{L^1}).
%\end{split}
\label{eqn: claim for smoothness}
\eeq
Here $\CC_k(t,\|f_0\|_{L^1},\|F_0\|_{L^1})$ denotes a generic function of $k$, $t$, $\|f_0\|_{L^1}$, and $\|F_0\|_{L^1}$, which satisfies:
\begin{enumerate}[(i)]
  \item It is finite whenever $t>0$;\label{assumption: finite}
  \item It is non-increasing in $t$, and non-decreasing in $\|f_0\|_{L^1}$ and $\|F_0\|_{L^1}$;\label{assumption: monotonicity}
  \item \label{assumption: exponential decay}
  It decays exponentially to $0$ as $t\to +\infty$ with an explicit rate that only depends on $\|f_0\|_{L^1}$ and $\|F_0\|_{L^1}$, i.e., there exists some $\al(\|f_0\|_{L^1},\|F_0\|_{L^1})>0$ and constant $C_k(\|f_0\|_{L^1},\|F_0\|_{L^1})>0$, such that for all $t\gg 1$,
      \[
      \CC_k(t,\|f_0\|_{L^1},\|F_0\|_{L^1})
      \leq C_k(\|f_0\|_{L^1},\|F_0\|_{L^1}) \exp[-\al(\|f_0\|_{L^1},\|F_0\|_{L^1})t].
      \]

\end{enumerate}
As a result, for any $t_0>0$, $f(x,t)$ is a smooth solution to \eqref{eqn: the main equation} on the space-time domain $\BT\times [t_0,+\infty)$.
As $t\to +\infty$, $f(x,t)$ converges exponentially to $f_\infty$ defined in \eqref{eqn: equilibrium}.
To be more precise, for arbitrary $j,k\in \BN$, and $t>0$,
\beq
%\begin{split}
%&
\big\|\pa_t^j \pa_x^k (f(\cdot,t)-f_\infty)\big\|_{L^2(\BT)}\leq \CC_{j,k}(t, \|f_0\|_{L^1},\|F_0\|_{L^1}).
%\mbox{ for all }t\in [t_0,+\infty),\\
%&\mbox{and it converges to $0$ exponentially as }t\to +\infty.
%\end{split}
\label{eqn: claim for decay}
\eeq
Here $\CC_{j,k}(t, \|f_0\|_{L^1},\|F_0\|_{L^1})$ is a function having similar properties as
$\CC_k(t, \|f_0\|_{L^1},\|F_0\|_{L^1})$ above.
%Here $H^0$-norm is interpreted as the $L^2$-norm.

\begin{rmk}
We do not pursue precise form of estimates for the above claims, as they are complicated but may not be optimal.
Nevertheless, in Section \ref{sec: analyticity and sharp decay estimate} below, we will prove sharp decay estimates for the solution when $t\gg 1$.
\end{rmk}

\begin{proof}
Throughout this proof, %we will use the notation $\CC_{k}(t, \|f_0\|_{L^1},\|F_0\|_{L^1})$ introduced above.
%Besides, we will also use
we will use the notation $C_k(t, \|f_0\|_{L^1},\|F_0\|_{L^1})$ to denote a generic function of $k$, $t$, $\|f_0\|_{L^1}$, and $\|F_0\|_{L^1}$, which satisfies the properties \eqref{assumption: finite} and \eqref{assumption: monotonicity} above, and yet
\begin{enumerate}[(iii')]
  \item $C_k(t, \|f_0\|_{L^1},\|F_0\|_{L^1})$ stays bounded as $t\to +\infty$.
\end{enumerate}
For brevity, we will write it as $C_k(t)$ in what follows.
Like the commonly-used notation $C$ for universal constants, its precise definition may vary from line to line.
If it turns out to not depend on $k$, we simply write it as $C(t)$.

By Lemma \ref{lem: lower bound}, Lemma \ref{lem: L inf bound}, and Propostion \ref{prop: improved H 1/2 estimate for h}, we find
\[
\|f(\cdot,t)\|_{L^\infty},\,
\|F(\cdot,t)\|_{L^\infty},\,
\|h(\cdot,t)\|_{\dot{H}^{\f12}}\leq C(t).
\]
As a result, Lemma \ref{lem: generalized H k+1/2 estimate for h}
and Lemma \ref{lem: connecting different quantities}
imply that, for any $k\in \BZ_+$ and any $t>0$,
\begin{align}
\f{d}{dt}\|h\|_{\dot{H}^{k+\f12}}^2
+ f_*(t)\|h\|_{\dot{H}^{k+1}}^2
\leq &\; \| h \|_{\dot{H}^{k+\f12}}^2\cdot
C_k(t) \big\|\sqrt{f} \big\|_{\dot{H}^{1}}^2,
\label{eqn: inequality for H k+1/2 seminorm of h}\\
\f{d}{dt}\big\|F^{\f12}\pa_x^k f\big\|_{L^2}^2 + \|f\|_{\dot{H}^{k+\f12}}^2
\leq &\;\big\|F^{\f12}\pa_x^k f\big\|_{L^2}^2 \cdot C_k (t) \mathds{1}_{\{k>1\}} \big\|\sqrt{f}\big\|_{\dot{H}^1}^2,
\label{eqn: inequality for weight H k seminorm of f}
\end{align}
and
\begin{align}
\|h\|_{\dot{H}^{k+\f12}}
\leq &\; \|f\|_{\dot{H}^{k+\f12}}\cdot
C_k (t) %f_*(t)^{-\f32}
\big\|\sqrt{f}\big\|_{\dot{H}^1},
\label{eqn: bounding H k+1/2 seminorm of h with the same seminorm of f}\\
\big\|F^{\f12}\pa_x^k f\big\|_{L^2}
\leq &\;\|h\|_{\dot{H}^k}\cdot C_k(t). %\|f\|_{L^\infty}^{\f12}.
\label{eqn: bounding weighted H k seminorm of f with H k seminorm of h}
\end{align}

\setcounter{step}{0}
\begin{step}[Preliminary]
We start from bounding $\|h(\cdot,t)\|_{\dot{H}^{\f12}}$ and $\|\sqrt{f}\|_{\dot{H}^{1}}$.

By \eqref{eqn: H 1/2 estimate for h derivation} in Lemma \ref{lem: H 1/2 estimate for h},
\beqo
\f{d}{dt}\|h\|_{\dot{H}^{\f12}}^2
\leq -f_*(t) \int_\BT \big(\pa_x h \big)^2 + \big((-\D)^{\f12} h\big)^2\,dx = -2f_*(t)\|h\|_{\dot{H}^1}^2. %\leq -2a_*\|h\|_{\dot{H}^{\f12}}^2.
\eeqo
Hence, for arbitrary $t\geq t_*>0$, we have
\[
\|h(\cdot,t)\|_{\dot{H}^{\f12}}^2
\leq
\big\|h(\cdot,t_*)\big\|_{\dot{H}^{\f12}}^2 \exp\left(-2\int_{t_*}^t f_*(\tau)\,d\tau\right),
\]
This together with Lemma \ref{lem: lower bound} and Proposition \ref{prop: improved H 1/2 estimate for h} implies $\|h(\cdot,t)\|_{\dot{H}^{\f12}}$ verifies the properties in \eqref{eqn: claim for smoothness}, although it is not included there.
%is finite for all $t>0$, it is decreasing on $(0,+\infty)$, and it decays exponentially as $t\to +\infty$.

Thanks to \eqref{eqn: bounding H1 norm of root f}, for any $t>0$,
\[
\big\|F^{\f12}\pa_x f(\cdot,t)\big\|_{L^2}^2 = 4\big\|\sqrt{f(\cdot,t)}\big\|_{\dot{H}^1}^2 \leq Ct^{-1}\|h(\cdot,t/2)\|_{\dot{H}^{\f12}}^2.
\]
Hence, $\|\sqrt{f(\cdot,t)}\|_{\dot{H}^1}$ satisfies \eqref{eqn: claim for smoothness} as well.
Consequently, given a $C_k(t)$ satisfying the assumptions at the beginning of the proof, we have
\beq
C_k(t)\big\|\sqrt{f(\cdot,t)} \big\|_{\dot{H}^{1}}^2 \leq \f12 f_*(t)
\label{eqn: smallness of RHS coefficient}
\eeq
for all $t\gg 1$, where the required largeness depends on the form of $C_k(\cdot)$, and thus on $k$, $\|f_0\|_{L^1}$, and $\|F_0\|_{L^1}$ essentially.

In the next two steps, we prove the claim \eqref{eqn: claim for smoothness} by induction.
\end{step}

\begin{step}[Base step]
Let us verify \eqref{eqn: claim for smoothness} with $k = 1$.
We have just proved $\|F^{\f12}\pa_x f(\cdot,t)\|_{L^2}$ satisfies \eqref{eqn: claim for smoothness}.
On the other hand, by Lemma \ref{lem: H 1 estimate for sqrt f}, for arbitrary $t>0$, there exists $t_*\in [\f12t, \f34t]$, such that
\[
\|f(\cdot,t_*)\|_{\dot{H}^{\f32}}^2\leq Ct^{-1}\big\| F^{\f12} \pa_x f(\cdot,t/2)\big\|_{L^2}^2.
\]
Then \eqref{eqn: bounding H k+1/2 seminorm of h with the same seminorm of f} implies
\[
\|h(\cdot,t_*)\|_{\dot{H}^{\f32}}
\leq
Ct^{-\f12}\big\| F^{\f12} \pa_x f(\cdot,t/2)\big\|_{L^2}
\cdot
C(t_*) % f_*(t_*)^{-\f32}
\big\|\sqrt{f(\cdot,t_*)}\big\|_{\dot{H}^1} \leq C(t). % C\big(t,\|f_0\|_{L^1},\|F_0\|_{L^1}\big).
\]
Now applying \eqref{eqn: inequality for H k+1/2 seminorm of h} with $k = 1$, %since $t\geq t_*$,
we find
\beq
\begin{split}
&\;\|h(\cdot,t)\|_{\dot{H}^{\f32}}^2
+ \int_{t_*}^t \exp\left(\int_{t'}^t C(\tau) \big\|\sqrt{f (\cdot,\tau)}\big\|_{\dot{H}^{1}}^2\,d\tau\right) f_*(t')\|h(\cdot,t')\|_{\dot{H}^2}^2\,dt'\\
\leq &\;
\exp\left(\int_{t_*}^t C(\tau) \big\|\sqrt{f (\cdot,\tau)}\big\|_{\dot{H}^{1}}^2\,d\tau\right)
\|h(\cdot,t_*)\|_{\dot{H}^{\f32}}^2.
\end{split}
\label{eqn: bound for H 3/2 seminorm of h}
\eeq
Notice that %, %for any $t_*>0$,
\beq
\int_{t_*}^\infty C(\tau) \big\|\sqrt{f (\cdot,\tau)}\big\|_{\dot{H}^{1}}^2\,d\tau \leq C(t_*) \leq C(t),
%C\big(t,\|f_0\|_{L^1},\|F_0\|_{L^1}\big).
\label{eqn: boundedness of growth factor}
\eeq
so
%Since $t$ is arbitrary,
%we obtain that
$\|h(\cdot,t)\|_{\dot{H}^{\f32}}\leq C(t)$. % is finite for all $t>0$, and it is uniformly bounded on $[t_0,+\infty)$ for arbitrary $t_0>0$.
Also, \eqref{eqn: inequality for H k+1/2 seminorm of h} and \eqref{eqn: smallness of RHS coefficient} imply that, for sufficiently large $t$,
\[
\f{d}{dt}\|h\|_{\dot{H}^{\f32}}^2
+ \f12 f_*(t)\|h\|_{\dot{H}^{\f32}}^2
\leq 0.
\]
Hence, $\|h(\cdot,t)\|_{\dot{H}^{\f32}}$ decays exponentially as $t\to +\infty$ with an explicit rate (cf.\;\eqref{eqn: lower bound for f}).
Therefore, $\|h(\cdot,t)\|_{\dot{H}^{\f32}}$ satisfies \eqref{eqn: claim for smoothness}.

We also derive from \eqref{eqn: bound for H 3/2 seminorm of h} that there exists $t'_*\in [\f34t,\f78t]$, such that
\beq
\|h(\cdot,t_*')\|_{\dot{H}^2}^2
\leq Ct^{-1}f_*(t_*')^{-1}\exp\left(\int_{t_*}^t C(\tau) \big\|\sqrt{f (\cdot,\tau)}\big\|_{\dot{H}^{1}}^2\,d\tau\right)
\|h(\cdot,t_*)\|_{\dot{H}^{\f32}}^2\leq C(t). %C\big(t,\|f_0\|_{L^1},\|F_0\|_{L^1}\big).
\label{eqn: auxiliary estimate}
\eeq
\end{step}

\begin{step}[Induction step]
Suppose that \eqref{eqn: claim for smoothness} holds for the case $k$, and that for any $t>0$, there exists $t_*\in [t-2^{-2k} t, t-2^{-2k-1}t]$, such that $\|h(\cdot,t_*)\|_{\dot{H}^{k+1}}\leq C_k(t)$ (cf.\;\eqref{eqn: auxiliary estimate}). %C(k,t,\|f_0\|_{L^1},\|F_0\|_{L^1})$.
In this step, we verify \eqref{eqn: claim for smoothness} for $\|h(\cdot,t)\|_{\dot{H}^{k+\f32}}$ and $\|F^{\f12}\pa_x^{k+1} f(\cdot,t)\|_{L^2}$ by arguing as above.

Thanks to \eqref{eqn: bounding weighted H k seminorm of f with H k seminorm of h},
\[
\big\|F^{\f12}\pa_x^{k+1} f(\cdot,t_*)\big\|_{L^2}
\leq \|h(\cdot,t_*)\|_{\dot{H}^{k+1}}\cdot C_{k+1}(t_*) \leq C_{k+1}(t).
%\CC_k(t_*) \|f(\cdot,t_*)\|_{L^\infty}^{\f12}\leq C\big(k,t,\|f_0\|_{L^1},\|F_0\|_{L^1}\big).
\]
Then \eqref{eqn: inequality for weight H k seminorm of f} yields
\beq
\begin{split}
&\;\big\|F^{\f12}\pa_x^{k+1} f(\cdot,t)\big\|_{L^2}^2
+ \int_{t_*}^t \exp\left(\int_{t'}^t C_{k+1}(\tau) \big\|\sqrt{f (\cdot,\tau)} \big\|_{\dot{H}^{1}}^2\,d\tau\right) \|f(\cdot,t')\|_{\dot{H}^{k+\f32}}^2\,dt'\\
\leq &\;
\exp\left(\int_{t_*}^t C_{k+1}(\tau) \big\|\sqrt{f (\cdot,\tau)} \big\|_{\dot{H}^{1}}^2\,d\tau\right)
\big\|F^{\f12}\pa_x^{k+1} f(\cdot,t_*)\big\|_{L^2}^2,
\end{split}
\label{eqn: bound for weight L^2 norm of H k+1 seminorm of f}
\eeq
which implies $\|F^{\f12}\pa_x^{k+1} f(\cdot,t)\|_{L^2}\leq C_{k+1}(t)$ (cf.\;\eqref{eqn: boundedness of growth factor}). % on $[t_0,+\infty)$ for any $t_0>0$.
In addition, since
\[
\|f\|_{\dot{H}^{k+\f32}}^2
\geq \big\|f^{\f12} F^{\f12}\pa_x^{k+1} f\big\|_{L^2}^2 \geq
f_*(t)\big\|F^{\f12}\pa_x^{k+1} f\big\|_{L^2}^2,
\]
\eqref{eqn: inequality for weight H k seminorm of f} also implies that, for $t\gg 1$ (cf.\;\eqref{eqn: smallness of RHS coefficient}),
\[
\f{d}{dt}\big\|F^{\f12}\pa_x^{k+1} f\big\|_{L^2}^2 + \f12 f_*(t)\big\|F^{\f12}\pa_x^{k+1} f\big\|_{L^2}^2
\leq 0.
\]
This gives the exponential decay of $\|F^{\f12}\pa_x^{k+1} f(\cdot,t)\|_{L^2}$ as $t\to +\infty$.
Therefore, $\|F^{\f12}\pa_x^{k+1} f(\cdot,t)\|_{L^2}$ satisfies \eqref{eqn: claim for smoothness}.

Besides, \eqref{eqn: bound for weight L^2 norm of H k+1 seminorm of f} also implies that there exists $t_*'\in [t-2^{-2k-1}t, t-2^{-2k-2}t]$, such that
\[
\begin{split}
\|f(\cdot,t_*')\|_{\dot{H}^{k+\f32}}^2
\leq &\;
C 2^{2k}t^{-1}\exp\left(\int_{t_*}^t C_{k+1}(\tau) \big\|\sqrt{f (\cdot,\tau)} \big\|_{\dot{H}^{1}}^2\,d\tau\right)
\big\|F^{\f12}\pa_x^{k+1} f(\cdot,t_*)\big\|_{L^2}^2 %\\
\leq  C_{k+1}(t).
%C\big(k+1, t,\|f_0\|_{L^1},\|F_0\|_{L^1}\big).
\end{split}
\]
Then we can use \eqref{eqn: inequality for H k+1/2 seminorm of h} and \eqref{eqn: bounding H k+1/2 seminorm of h with the same seminorm of f} and argue as in the previous step to show that $\|h(\cdot,t)\|_{\dot{H}^{k+\f32}}$ satisfies \eqref{eqn: claim for smoothness}.
Moreover, there exists $t''_*\in [t-2^{-2k-2}t,t-2^{-2k-3}t]$, such that
$\|h(\cdot,t_*'')\|_{\dot{H}^{k+2}} \leq C_{k+1}(t)$.
%\leq C(k+1, t,\|f_0\|_{L^1},\|F_0\|_{L^1})$.
We omit the details.

By induction, we thus prove \eqref{eqn: claim for smoothness} for all $k\in \BZ_+$.
\end{step}

\begin{step}[Convergence to $f_\infty$]
For all $k\in \BZ_+$,
\[
\|f\|_{\dot{H}^k}^2
= \big\|f^{\f12}F^{\f12}\pa_x^k f(\cdot,t)\big\|_{L^2}^2
\leq \|f\|_{L^\infty}\big\|F^{\f12}\pa_x^k f(\cdot,t)\big\|_{L^2}^2.
\]
By Lemma \ref{lem: L inf bound}, $\|f(\cdot,t)\|_{\dot{H}^k}$ is bounded on $[t_0,+\infty)$ for arbitrary $t_0>0$, and it decays exponentially as $t\to +\infty$.
Hence, $\|f(\cdot,t)\|_{\dot{H}^{\f12}}$ satisfies the same properties.
Combining this with Lemma \ref{lem: energy bound and decay of L^p norms}, we find that $\|f(\cdot,t)\|_{L^1}$ is non-increasing in time and its convergence at $t\to +\infty$ is exponential.
Hence, $f(\cdot,t)$ is smooth in space-time in $\BT\times [t_0,+\infty)$ for any $t_0>0$, and it converges exponentially to some constant as $t\to +\infty$ in $H^k(\BT)$-norms for all $k\in \BN$.
Recall that $\|F\|_{L^1}$ is conserved (cf.\;Lemma \ref{lem: energy bound and decay of L^p norms}), so the constant must be $f_\infty$ defined in \eqref{eqn: equilibrium}.
This proves \eqref{eqn: claim for decay} with $j=0$.

The claim \eqref{eqn: claim for decay} with $j\in \BZ_+$ immediately follows from the case $j = 0$ and equation \eqref{eqn: the main equation}.
\end{step}
\end{proof}
\end{prop}

\subsection{Time-continuity at $t = 0$}
\label{sec: continuity at initial time}
We end this section by showing the time-continuity of the solution at $t = 0$ in a suitable sense.
Note that the time-continuity at all positive times is already covered in Proposition \ref{prop: boundedness and decay of H^k norms of h}.

\begin{lem}
\label{lem: continuity in Wasserstein at time 0}
Suppose $f$ is a positive strong solution of \eqref{eqn: the main equation}.
Then for any $0<t\leq \|f_0\|_{L^1}^{-1}$ and $\al\in (0,\f15)$,
\beqo
W_1(F_0,F(\cdot,t))
\leq  C_\al \|F_0\|_{L^1}
\big(t\|F_0\|_{L^1}^{-1}\big)^{\f{1-5\al}{2}} \big(\|f_0\|_{L^1}\|F_0\|_{L^1}\big)^{\f{1-\al}{2}}.
\eeqo
Here $W_1(\cdot,\cdot)$ denotes the 1-Wasserstein distance \cite[Chapter 6]{villani2009optimal}, and $C_\al>0$ is a universal constant depending on $\al$.

\begin{proof}
Given $t_*\in(0,\|f_0\|_{L^1}^{-1}]$ and $\d<t_*$, let $\eta_\d(t)$ be a smooth non-increasing cutoff function on $[0,+\infty)$, such that $\eta_\d \equiv 1$ on $[0,t_*-\d]$ and $\eta_\d \equiv 0$ on $[t_*,+\infty)$.
Let $\phi\in C^\infty(\BT)$.
We take $\va(x,t) = \phi(x)\eta_\d(t)$ in \eqref{eqn: weak formulation} and derive that
\beqo
\int_{\BT} \phi(x) F_0(x)\,dx + \int_0^{t_*} \eta_\d'(t) \int_\BT \phi(x) F(x,t)\,dx\,dt
= \int_0^{t_*}\eta_\d(t)\int_{\BT} \phi'(x) \cdot \CH f \cdot F\,dx\,dt.
\eeqo
Sending $\d \to 0^+$ and using the smoothness of $f$ and $F$ at positive times, we obtain
\beqo
\int_{\BT} \phi(x) \big(F_0(x)- F(x,t_*)\big)\,dx
= \int_0^{t_*} \int_{\BT} \phi'(x) \cdot \CH f \cdot F\,dx\,dt.
\eeqo
%Hence, %if we additionally assume $t_*\leq \|f_0\|_{L^1}^{-1}$,
By Corollary \ref{cor: bound for H1 norm} and the fact that $\|F(\cdot,t)\|_{L^1}$ is conserved, % with e.g.\;$\al = \f1{10}$,
\beqo
\begin{split}
&\;\left|\int_{\BT} \phi(x) \big(F_0(x)- F(x,t_*)\big)\,dx\right|\\
%\leq &\; \int_0^{t_*} \|\phi'\|_{L^\infty} \|\CH f\|_{L^\infty} \|F\|_{L^1}\,dt\\
%
\leq &\; C_\al \|\phi'\|_{L^\infty}\|F_0\|_{L^1}   \int_0^{t_*} \|f(\cdot,t)\|_{\dot{C}^\al} \,dt\\
\leq &\; C_\al  \|\phi'\|_{L^\infty} \|F_0\|_{L^1}
\big(t_*\|F_0\|_{L^1}^{-1}\big)^{\f{1-5\al}{2}}\big(\|f_0\|_{L^1}\|F_0\|_{L^1}\big)^{\f{1-\al}{2}}.
%\big|\ln \ln \big(t_*^{-1}\|F_0\|_{L^1}\big)\big|^{2\al}.
\end{split}
\eeqo
By the equivalent definition of the $1$-Wasserstein distance (see \cite[Chapter 6]{villani2009optimal}) %,
\beqo
%\begin{split}
W_1(F_0,F(\cdot,t_*)) = %&\;
\sup_{\|\phi\|_{\mathrm{Lip}}\leq 1} \int_{\BT} \phi(x) \big(F_0(x)- F(x,t_*)\big)\,dx, \eeqo
we conclude.
%Here $W_1(\cdot,\cdot)$ denotes the 1-Wasserstein distance \cite[Chapter 6]{villani2009optimal}.
\end{proof}
\end{lem}

\section{Global Well-posedness}
\label{sec: global well-posedness}

\subsection{Well-posedness for band-limited positive initial data}
\label{sec: band-limited initial data}
First we consider positive and band-limited initial data $f_0$; the latter assumption means that $f_0$ is a finite linear combination of Fourier modes.
To prove the well-posedness in this case, it might be feasible to turn the equation \eqref{eqn: the main equation} back to its Lagrangian formulation (cf.\;\eqref{eqn: contour dynamic equation 1D prelim} and \eqref{eqn: contour dynamic equation 1D}), and argue as in \cite{lin2019solvability,mori2019well}.
Indeed, we can recover the string configuration $X_0(s)$ from $f_0(x)$ through \eqref{eqn: def of f} (also see the proof of Corollary \ref{cor: rephrasing in Lagrangian} in Section \ref{sec: existence for general initial data}), and $X_0$ should be smooth and satisfy the well-stretched assumption (cf.\;Remark \ref{rmk: meaning of max and min principle}).
Nonetheless, here we present a more straightforward approach by taking advantage of the special nonlinearity of \eqref{eqn: the main equation}.

Define the Fourier transform of $f\in L^1(\BT)$ as in \eqref{eqn: Fourier transform}.
Denote $\bar{f}:= \f{1}{2\pi}\int_\BT f(x)\,dx = \CF(f)_0$.
We write \eqref{eqn: the main equation} as
\[
\pa_t f + \bar{f}(t) (-\D)^{\f12} f = \CH f\cdot \pa_x f - (f-\bar{f})\cdot (-\D)^{\f12} f.
\]
Taking the Fourier transform on both sides yields, for any $m\in \BZ$,
\[
\begin{split}
\f{d}{dt} \hat{f}_m +|m|\bar{f}(t)\hat{f}_m
= &\; \sum_{n\neq m} -i\mathrm{sgn}(m-n)\hat{f}_{m-n} \cdot in\hat{f}_n -\hat{f}_{m-n} \cdot |n|\hat{f}_n\\
= &\; - \sum_{n\neq 0,m}  \big(1-\mathrm{sgn}(m-n)\mathrm{sgn}(n)\big) \hat{f}_{m-n}\cdot |n|\hat{f}_n.
\end{split}
\]
Given $n\neq 0,m$, we observe that $\mathrm{sgn}(m-n)\mathrm{sgn}(n) \neq 1$ only when $(m-n)$ and $n$ have opposite signs, which implies $\max\{|m-n|,|n|\}> |m|$.
Hence, for $m\geq 0$,
\[
\begin{split}
\f{d}{dt} \hat{f}_m
= &\; - m\bar{f}(t)\hat{f}_m
- \sum_{\max\{|m-n|,|n|\}> m} 2\hat{f}_{m-n}\cdot |n|\hat{f}_n\\
%= &\;- \sum_{\max\{|m-n|,|n|\}\geq m} \f12 \big(|n|+|m-n|\big) \hat{f}_{m-n}\hat{f}_n\\
%= &\;- \sum_{n\leq 0}\f12 \big(|n|+|m-n|\big) \hat{f}_{m-n}\hat{f}_n
%- \sum_{n\geq m}\f12 \big(|n|+|m-n|\big) \hat{f}_{m-n}\hat{f}_n\\
= &\;- m\bar{f}(t)\hat{f}_m - \sum_{n< 0} 2\big(|n|+|m-n|\big) \hat{f}_{m-n}\hat{f}_n,
\end{split}
\]
i.e.,
\beq
\f{d}{dt} \hat{f}_m
= - m \bar{f}(t)\hat{f}_{m}
- \sum_{n>0} 2(m+2n) \hat{f}_{m+n}\overline{\hat{f}_n}.
\label{eqn: equation in Fourier space}
\eeq
The equation for $\hat{f}_m$ with $m<0$ can be obtained by taking the complex conjugate.
Formally, this immediately implies that, if there exists some $K\in \BN$ such that $\CF(f_0)_k = 0$ for all $|k|> K$,
then $\hat{f}_k(t) \equiv 0$ for all $t\geq 0$ and $|k|> K$, and moreover, $\hat{f}_{\pm K}'(t) = -K \bar{f}(t)\hat{f}_{\pm K}(t)$.
Hence, we have the following well-posedness result.

\begin{prop}[Proposition \ref{prop: band-limited initial data intro}]
\label{prop: band-limited initial data}
Suppose $f_0>0$ is band-limited, i.e., there exists some $K\in\BN$, such that $\CF(f_0)_k = 0$ for all $|k|>K$.
Then \eqref{eqn: the main equation} has a global strong solution $f = f(x,t)$ that is also positive and band-limited, such that $\hat{f}_k(t) = 0$ for all $|k|>K$ and $t\geq 0$.
Such a solution is unique, and it can be determined by a finite family of ODEs among \eqref{eqn: equation in Fourier space} with the initial condition $\hat{f}_k(0) = \CF(f_0)_k$.

\begin{proof}
%In this case, finding the solution to \eqref{eqn: the main equation} boils down to
As we pointed out, it is enough to solve the ODEs for $\hat{f}_k(t)$ with $k \in [-K,K]$ (see \eqref{eqn: equation in Fourier space}) and simply put $\hat{f}_k(t) \equiv 0$ if $|k|>K$.
The local well-posedness of the ODEs is trivial.
The local ODE solution can be extended to the time interval $[0,+\infty)$ thanks to the boundedness of $\|f(\cdot,t)\|_{L^2}$ (see Lemma \ref{lem: energy bound and decay of L^p norms}) as well as the Paserval's identity.
Positivity of $f$ follows from Lemma \ref{lem: max principle}.
\end{proof}
\end{prop}

\subsection{Weak solutions for initial data in the energy class}
\label{sec: existence for general initial data}

Now we turn to prove Theorem \ref{thm: main thm}.
Recall that we call $f_0\in L^1(\BT)$ an initial data in the energy class in the view of Remark \ref{rmk: energy class}.

\begin{proof}[Proof of Theorem \ref{thm: main thm} (part)]
In this part of the proof, we will establish existence of a global weak solution to \eqref{eqn: the main equation}, and then show it enjoys the properties \eqref{property: energy inequality in the main thm}-\eqref{property: long-time convergence} in Theorem \ref{thm: main thm}.
The remaining property \eqref{property: large-time analyticity} will be studied in Section \ref{sec: analyticity and sharp decay estimate} (see Corollary \ref{cor: analyticity of the solution}).

To obtain weak solutions corresponding to general initial data, we shall first approximate them by smooth solutions, and then pass to the limit in the weak formulation \eqref{eqn: weak formulation}.
In this process, we will need the Fej\'{e}r kernel on $\BT$ \cite[Chapter 1]{duoandikoetxea2001fourier}
\beq
\mathscr{F}_N(x) :=\f{1}{2\pi N} \f{\sin^2 (\f{N}{2}x)}{\sin^2 \f{x}{2}}
=\f{1}{2\pi}\sum_{|k|<N} \left(1-\f{|k|}{N}\right) e^{ikx}.
\label{eqn: Fejer kernel}
\eeq
For any given $N\in \BZ_+$, $\mathscr{F}_N$ is non-negative and band-limited, having integral $1$ on $\BT$.
$\{\mathscr{F}_N\}_{N = 1}^\infty$ is known to be an approximation of identity on $\BT$.

We proceed in several steps.

\setcounter{step}{0}
\begin{step}[Approximate solutions]
\label{step: approximate solutions}
Given $f_0\in L^1(\BT)$ with $F_0 = \f{1}{f_0} \in L^1(\BT)$, define their cut-offs for $M\in \BZ_+$ that
\[
f_{0,M} := \max\big\{\min\{f_0,M\},M^{-1}\big\},\quad F_{0,M} = \big(f_{0,M}\big)^{-1}.
\]
Then $f_{0,M},F_{0,M}\in [M^{-1},M]$ on $\BT$.
Clearly, as $M\to +\infty$, $f_{0,M}\to f_0$ and $F_{0,M}\to F_0$ in $L^1(\BT)$.
Next, for $N\in \BZ_+$, we let
\[
f_{0,M,N} := \mathscr{F}_N* f_{0,M},\quad F_{0,M,N} : = \big(f_{0,M,N}\big)^{-1}.
\]
Then $f_{0,M,N}$ is positive and band-limited; $f_{0,M,N}\in [M^{-1},M]$ on $\BT$.
As $N\to +\infty$, $f_{0,M,N}\to f_{0,M}$ in $L^1(\BT)$. % thanks to the property of approximation of identity.
In particular, we have $\|f_{0,M,N}\|_{L^1} =\|f_{0,M}\|_{L^1}$ for all $N\in \BZ_+$.
Moreover, $F_{0,M,N}\to F_{0,M}$ in $L^1(\BT)$ as $N\to +\infty$, because
$|F_{0,M,N}- F_{0,M}|\leq M^2 |f_{0,M,N}- f_{0,M}|$.
We also have $\|F_{0,M,N}\|_{L^1} \leq \|F_{0,M}\|_{L^1}$ for all $N\in \BZ_+$.
This is because, by Jensen's inequality,
\[
F_{0,M,N}(x) = \f{1}{\mathscr{F}_N* f_{0,M}(x)} \leq \left[\mathscr{F}_N* \left(\f{1}{f_{0,M}}\right)\right](x) = \mathscr{F}_N*F_{0,M}(x).
\]

By Proposition \ref{prop: band-limited initial data}, there exists a positive and band-limited $f_{M,N}$ that is a strong solution to \eqref{eqn: the main equation} with initial data $f_{0,M,N}$.
Define $F_{M,N} = (f_{M,N})^{-1}$.
Hence, for any $\va\in C_0^\infty(\BT\times [0,+\infty))$ and arbitrary $0<\d\ll 1$,
\beq
\int_{\BT} \va(x,0) F_{0,M,N}(x)\,dx
= \int_{\BT\times [0,\d]} + \int_{\BT\times [\d,+\infty)}   \big(\pa_x \va \cdot \CH f_{M,N} -\pa_t \va \big)  F_{M,N} \,dx\,dt.
\label{eqn: weak formulation of approximate solutions}
\eeq

\end{step}

\begin{step}[Taking the limit in $N$]
\label{step: taking the N limit}
Since $\|f_{0,M,N}\|_{L^1}$ and $\|F_{0,M,N}\|_{L^1}$ are uniformly bounded in $N$, by Lemma \ref{lem: energy bound and decay of L^p norms} and Proposition \ref{prop: boundedness and decay of H^k norms of h}, there exist smooth functions $f_{M}$ and $F_M$ defined on $\BT\times (0,+\infty)$, and subsequences $\{f_{M,N_j}\}_{j = 1}^\infty$ and $\{F_{M,N_j}\}_{j = 1}^\infty$, such that,
as $j\to +\infty$, $f_{M,N_j}\to f_M$ and $F_{M,N_j}\to F_M$ in $C_{loc}^k(\BT\times (0,+\infty))$ ($\forall\, k\in \BN$), and moreover, $f_{M,N_j}\rightharpoonup f_M$ in $L^2([0,+\infty); H^{\f12}(\BT))$.
It is then not difficult to verify the following facts.
\begin{enumerate}[(a)]
\item $F_M = \f1{f_M}$ on $\BT\times (0,+\infty)$; by Lemma \ref{lem: max principle}, $f_M, F_M\in [M^{-1},M]$ on $\BT\times (0,+\infty)$.
\item \label{property: conservation of L^1}
For all $t>0$,
\[
\|F_M(\cdot,t)\|_{L^1}
= \lim_{j\to +\infty} \|F_{M,N_j}(\cdot,t)\|_{L^1}
= \lim_{j\to +\infty} \|F_{0,M,N_j}\|_{L^1}
= \|F_{0,M}\|_{L^1}.
\]
\item \label{property: energy inequality}
For any $0<t_1\leq t_2$,
\beq
\f12 \|f_M(\cdot,t_2)\|_{L^1} + \int_{t_1}^{t_2} \|f_M(\cdot,\tau)\|_{\dot{H}^{1/2} }^2\,d\tau = \f12 \|f_M(\cdot,t_1)\|_{L^1} \leq \f12 \|f_{0,M}\|_{L^1}.
\label{eqn: energy relation prelim}
\eeq
Indeed, one first proves an identical energy relation for $f_{M,N_j}$ (cf.\;\eqref{eqn: energy estimate}), and then send $j\to +\infty$.
This in particular implies the energy inequality, i.e., for all $t>0$,
\beqo
\f12 \|f_M(\cdot,t)\|_{L^1} + \int_0^{t} \|f_M(\cdot,\tau)\|_{\dot{H}^{1/2} }^2\,d\tau \leq \f12 \|f_{0,M}\|_{L^1}.
\eeqo
We will show later that the energy equality is actually achieved (see \eqref{eqn: energy relation} below).
\item \label{property: uniform estimate for Holder norm}
    With $M\in \BZ_+$ fixed, the estimates in Corollary \ref{cor: bound for H1 norm} hold uniformly for all $f_{M,N_j}$ and $f_M$, where $f_0$ and $F_0$ there should be replaced by $f_{0,M}$ and $F_{0,M}$ respectively.
    Let us explain the estimates for $f_M$.
    Indeed, \eqref{eqn: bound for H1 norm} follows from the convergence $f_{M,N_j}\to f_M$ at all positive times, and then the proof of \eqref{eqn: time integral estimate of some Holder norm} for $f_M$ uses the property \eqref{property: energy inequality} above and a similar argument as in Corollary \ref{cor: bound for H1 norm}.
    As a consequence, $\CH f_M \cdot F_M \in L_{loc}^1(\BT\times [0,+\infty))$.

\item With $M\in \BZ_+$ fixed, the estimate \eqref{eqn: claim for decay} in Proposition \ref{prop: boundedness and decay of H^k norms of h} holds uniformly for all $f_{M,N_j}$ and $f_M$, again with $f_0$ and $F_0$ there replaced by $f_{0,M}$ and $F_{0,M}$.
    Observe that $F_{0,M,N}\to F_{0,M}$ in $L^1(\BT)$ implies $f_{\infty,M,N}\to f_{\infty,M}$ (they are defined by \eqref{eqn: equilibrium} in terms of $f_{0,M,N}$ and $f_{0,M}$ respectively).
    Hence, $f_M$ converges to $f_{\infty,M}$ exponentially as $t\to +\infty$.

\end{enumerate}

Now in \eqref{eqn: weak formulation of approximate solutions}, we take $N$ to be $N_j$ and send $j\to +\infty$.
Note that $\d$ is arbitrary in \eqref{eqn: weak formulation of approximate solutions}.
By virtue of \eqref{property: uniform estimate for Holder norm} above, as $\d\to 0^+$,
\[
\sup_{j\in \BZ_+}\int_{\BT\times [0,\d]}\big(\pa_x \va \cdot \CH f_{M,N_j}-\pa_t \va \big) F_{M,N_j} \,dx\,dt\mbox{ and }
\int_{\BT\times [0,\d]}\big(\pa_x \va \cdot \CH f_{M}-\pa_t \va \big) F_{M} \,dx\,dt
\]
both converge to $0$.
Hence,
\beq
\int_{\BT} \va(x,0) F_{0,M}(x)\,dx
= \int_{\BT\times [0,+\infty)} \pa_x \va \cdot \CH f_{M} \cdot F_{M} -\pa_t \va \cdot  F_{M} \,dx\,dt,
\label{eqn: weak formulation of approximate solutions 2nd round}
\eeq
Therefore, $f_M$ is a global weak solution to \eqref{eqn: the main equation} with initial data $f_{0,M}$.

We further confirm that $F_M$ takes the ``initial data" $F_{0,M}$ in a more quantitative sense.
For $t>0$,
\beqo
\begin{split}
&\;W_1\left(\f{F_{0,M}}{\|F_{0,M}\|_{L^1}}, \f{F_M(\cdot,t)}{\|F_{0,M}\|_{L^1}}\right)\\
\leq &\; W_1\left(\f{F_{0,M}}{\|F_{0,M}\|_{L^1}}, \f{F_{0,M,N_j}}{\|F_{0,M,N_j}\|_{L^1}}\right)
+ W_1\left(\f{F_{M,N_j}(\cdot,t)}{\|F_{0,M,N_j}\|_{L^1}}, \f{F_M(\cdot,t)}{\|F_{0,M}\|_{L^1}}\right)\\
&\;
+ W_1\left(\f{F_{0,M,N_j}}{\|F_{0,M,N_j}\|_{L^1}}, \f{F_{M,N_j}(\cdot,t)}{\|F_{0,M,N_j}\|_{L^1}}\right).
\end{split}
\eeqo
The first two terms on the right-hand side converge to zero as $j\to +\infty$ because of strong $L^1$-convergence.
Thanks to Lemma \ref{lem: continuity in Wasserstein at time 0} with e.g.\;$\al= \f{1}{10}$, the last term above satisfies that,
for arbitrary $j\in \BZ_+$ and $t\in (0,\|f_{0,M}\|_{L^1}^{-1}]$ (recall that $\|f_{0,M,N}\|_{L^1} =\|f_{0,M}\|_{L^1}$ for all $N\in \BZ_+$),
\[
W_1\left(\f{F_{0,M,N_j}}{\|F_{0,M,N_j}\|_{L^1}}, \f{F_{M,N_j}(\cdot,t)}{\|F_{0,M,N_j}\|_{L^1}}\right)
\leq C
\big(t\|F_{0,M,N_j}\|_{L^1}^{-1}\big)^{\f1{4}} \big(\|f_{0,M}\|_{L^1}\|F_{0,M,N_j}\|_{L^1}\big)^{\f9{20}}.
\]
Therefore,
\beq
W_1\left(\f{F_{0,M}}{\|F_{0,M}\|_{L^1}}, \f{F_M(\cdot,t)}{\|F_{0,M}\|_{L^1}}\right)
\leq
C
\big(t\|F_{0,M}\|_{L^1}^{-1}\big)^{\f1{4}} \big(\|f_{0,M}\|_{L^1}\|F_{0,M}\|_{L^1}\big)^{\f{9}{20}}
\label{eqn: continuity at initial time for the second approximate solution}
\eeq
for all $t\in (0,\|f_{0,M}\|_{L^1}^{-1}]$.
This further implies \cite{villani2009optimal}, as $t\to 0^+$,
\[
\f{F_M(\cdot,t)}{\|F_{0,M}\|_{L^1}}\mbox{ converges weakly in }\CP(\BT)
\mbox{ to }\f{F_{0,M}}{\|F_{0,M}\|_{L^1}},
\]
where $\CP(\BT)$ denotes the space of probability measure on $\BT$.
Since $|F_M|\leq M$, we further have that, for any $p\in [1,+\infty)$, $F_M(\cdot,t)$ converges weakly in $L^p(\BT)$ to $F_{0,M}$ as $t \to 0^+$.
Since $x\mapsto \f1x$ is convex on $(0,+\infty)$,
$\|f_{0,M}\|_{L^1} \leq \liminf_{t\to 0^+}\|f_{M}(\cdot,t)\|_{L^1}$ (cf.\;\cite[Chapter 2]{evans1990weak}).
This together with \eqref{eqn: energy relation prelim} implies that $\lim_{t\to 0^+} \|f_M(\cdot,t)\|_{L^1} = \|f_{0,M}\|_{L^1}$ and therefore, for any $t>0$,
\beq
\f12 \|f_M(\cdot,t)\|_{L^1} + \int_0^{t} \|f_M(\cdot,\tau)\|_{\dot{H}^{1/2} }^2\,d\tau = \f12 \|f_{0,M}\|_{L^1}.
\label{eqn: energy relation}
\eeq
\end{step}

\begin{step}[Taking the limit in $M$]
%
%Now we send $M\to +\infty$ in a similar manner.
By virtue of the properties of $f_{0,M}$ and $F_{0,M}$, as well as the uniform bounds for $f_M$ and $F_M$ proved above, there exist smooth functions $f$ and $F$ defined on $\BT\times (0,+\infty)$, and subsequences $\{f_{M_i}\}_{i = 1}^\infty$ and $\{F_{M_i}\}_{i = 1}^\infty$, such that,
as $i\to +\infty$, $f_{M_i}\to f$ and $F_{M_i}\to F$ in $C_{loc}^k(\BT\times (0,+\infty))$ ($\forall\, k\in \BN$), and moreover, $f_{M_i}\rightharpoonup f$ in $L^2([0,+\infty); H^{\f12}(\BT))$.
One can similarly verify: % the following facts.
\begin{enumerate}[(a')]
\item $F = \f1{f}$ on $\BT\times (0,+\infty)$.
\item For all $t>0$, $\|F(\cdot,t)\|_{L^1} = \|F_0\|_{L^1}$.
\item For any $t>0$,
\beq
\f12 \|f(\cdot,t)\|_{L^1} + \int_0^t \|f(\cdot,\tau)\|_{\dot{H}^{1/2} }^2\,d\tau \leq \f12 \|f_0\|_{L^1}.
\label{eqn: energy inequality final}
\eeq
\item The estimates in Corollary \ref{cor: bound for H1 norm} hold for $f$.
    Hence, $\CH f \cdot F \in L_{loc}^1(\BT\times [0,+\infty))$.

\item The estimate \eqref{eqn: claim for decay} in Proposition \ref{prop: boundedness and decay of H^k norms of h} holds for $f$, i.e., $f$ converges to $f_{\infty}$ exponentially as $t\to +\infty$, where $f_\infty$ is defined in terms of $f_0$.
\end{enumerate}
Now taking $M$ in \eqref{eqn: weak formulation of approximate solutions 2nd round} to be $M_i$ and sending $i \to +\infty$, we may similarly show that
\beqo
\int_{\BT} \va(x,0) F_{0}(x)\,dx
= \int_{\BT\times [0,+\infty)} \pa_x \va \cdot \CH f \cdot F  -\pa_t \va \cdot  F \,dx\,dt,
\eeqo
i.e., $f$ is a global weak solution to \eqref{eqn: the main equation} with initial data $f_{0}$.
Finally, with the help of \eqref{eqn: continuity at initial time for the second approximate solution}, we can argue analogously as before to show
\beq
\begin{split}
W_1\left(\f{F_{0}}{\|F_{0}\|_{L^1}}, \f{F(\cdot,t)}{\|F_{0}\|_{L^1}}\right)
\leq &\;
C
\big(t\|F_{0}\|_{L^1}^{-1}\big)^{\f14}
\big(\|f_{0}\|_{L^1}\|F_{0}\|_{L^1}\big)^{\f{9}{20}}
\end{split}
\label{eqn: estimate for W_1 distance}
\eeq
for all $t\in (0,\f12\|f_{0}\|_{L^1}^{-1}]$.
Hence, as $t\to 0^+$,
\beqo
\f{F(\cdot,t)}{\|F_{0}\|_{L^1}}\mbox{ converges weakly in }\CP(\BT)
\mbox{ to }\f{F_{0}}{\|F_{0}\|_{L^1}}.
%\label{eqn: weak convergence in P(T) at time 0}
\eeqo
This time we do not necessarily have $F(\cdot,t) \rightharpoonup F_0$ in $L^1(\BT)$ as $t\to 0^+$.
It is then not clear whether one can refine \eqref{eqn: energy inequality final} to become the energy equality.
However, if $F_0\in L^p$ for some $p>1$, one can still prove the weak convergence as $t\to 0^+$ (cf.\;part \eqref{part: decay of L^p norms of f and F} of Lemma \ref{lem: energy bound and decay of L^p norms}) and then the energy equality as before.

Other properties of $f$ and $F$, such as their upper and lower bounds at positive times, can be proved similarly by first applying the a priori estimates (cf.\;Section \ref{sec: positivity and boundedness}) to the approximate solutions and then taking the limits.
We skip the details.
\end{step}
This completes this part of the proof.
\end{proof}

Now let us prove Corollary \ref{cor: rephrasing in Lagrangian}.
The main idea is as follows.
In the view of \eqref{eqn: transporting velocity for particles}, formally, for each $s\in \BR$, $X(s,t)$ should satisfy the ODE
\[
\f{d}{dt}X(s,t) = -\f14 \CH \tilde{f} (X(s,t),t),
\]
where $\tilde{f}$ solves \eqref{eqn: equation for f} on $\BR$ with the periodic initial condition $f_0$ being suitably defined by $X_0$.
Here we used a different notation $\tilde{f}$ so that it is distinguished from $f$ that solves \eqref{eqn: the main equation} on $\BT$.
Hence, our task in the following proof is to define $f_0$ from $X_0$, find $\tilde{f}$ with the help of Theorem \ref{thm: main thm}, and then determine $X(s,t)$ by solving the ODE above.

\begin{proof}[Proof of Corollary \ref{cor: rephrasing in Lagrangian}]

Denote the inverse function of $s\mapsto X_0(s)$ to be $G_0(x)$.
Both $X_0$ and $G_0$ are absolutely continuous on any compact interval of $\BR$ and strictly increasing on $\BR$. %, so is $X_0$ because of its $H^1$-regularity.
Define $F_0(x) = G_0'(x)$ a.e.\;on $\BR$.
Then $F_0$ is non-negative and $2\pi$-periodic, and $\int_{-\pi}^\pi F_0(x)\,dx = 2\pi$.
Moreover, since $G_0\circ X_0(s) = s$ and $X_0\circ G_0(x) = x$ with $G_0$ and $X_0$ being absolutely continuous on any compact interval of $\BR$,
we have $F_0(X_0(s))X_0'(s) = 1$ a.e., and equivalently, $X_0'(G_0(x))G_0'(x) = 1$ a.e..
Hence,
\[
\int_{-\pi}^\pi  |X_0'(s)|^2 \,ds
%= \int_{-\pi}^\pi \f{X_0'(s)}{F_0(X_0(s))}  \,ds
= \int_{G_0(X_0(-\pi))}^{G_0(X_0(\pi))} \f{1}{F_0(X_0(s))^2}  \,ds
= %\int_{X_0(-\pi)}^{X_0(\pi)}\f{1}{F_0(X_0(s))}\,dX_0(s) =
\int_{X_0(-\pi)}^{X_0(\pi)}\f{1}{F_0(x)}\,dx.
\]
Define $f_0 = \f{1}{F_0}$.
Hence, $f_0$ is $2\pi$-periodic, and it is integrable on any interval of length $2\pi$, with the integral being equal to $\int_{-\pi}^{\pi}|X'_0(s)|^2\,ds$.

Now we take $f_0$ and $F_0$ as the initial data and solve \eqref{eqn: equation for f} on $\BT$.
Since \eqref{eqn: equation for f} contains the extra coefficient $\f14$ compared with \eqref{eqn: the main equation} (cf.\;Section \ref{sec: derivation of the model}), we first take $f$ and $F$ to be the global weak solution to \eqref{eqn: the main equation} that was constructed above with the initial data $f_0$ and $F_0$,
and then let
\[
\tilde{f}(x,t) = f\left(x,\f14 t\right),\quad \tilde{F}(x,t) = F\left(x,\f14 t\right).
\]
Clearly, they are global weak solutions to \eqref{eqn: equation for f} on $\BT \times [0,+\infty)$.
We extend them in space $2\pi$-periodically to the entire real line, still denoting them as $\tilde{f}$ and $\tilde{F}$ respectively.

In the view of \eqref{eqn: transporting velocity for particles}, we shall use the transporting velocity field $-\f14 \CH \tilde{f}$ to define a flow map on $\BR$.
Here $\CH \tilde{f}$ denotes the $2\pi$-periodic extension on $\BR$ of the Hilbert transform (on $\BT$) of the original $\tilde{f}$ defined on $\BT$; or equivalently, one may treat it as the Hilbert transform (on $\BR$) of the extended $\tilde{f}$.
To avoid the low regularity issue of $\tilde{f}$ at $t=0$, we specify the ``initial data" of the flow map at $t = 1$.
Given $x\in \BR$, let $\Psi_t(x)$ solve
\[
\f{d}{dt}\Psi_t (x) = -\f14 \CH \tilde{f}( \Psi_t(x),t),\quad \Psi_1(x) = x.
\]
Thanks to the properties of $\tilde{f}$ (see Theorem \ref{thm: main thm} and the proof in Section \ref{sec: existence for general initial data}), it is not difficult to show that
\begin{itemize}
\item $\Psi_t(x)$ can be uniquely defined for all $(x,t)\in \BR\times (0,+\infty)$.
     In $\BR\times (0,+\infty)$, the map $(x,t)\mapsto \Psi_t(x)$ is smooth.
\item For any $t>0$, the map $x\mapsto \Psi_t (x)$ is bijective, strictly increasing, and is a smooth diffeomorphism from $\BR$ to itself.
    $x\mapsto \Psi_t(x)-x$ is $2\pi$-periodic.
\item By \eqref{eqn: time integral estimate of some Holder norm} in Corollary \ref{cor: bound for H1 norm}, for any $x\in \BR$, $t\mapsto \Psi_t(x)$ is H\"{o}lder continuous for sufficiently small $t$. %, with the H\"{o}lder semi-norm being uniformly bounded in $x$.
    The H\"{o}lder exponent can be arbitrarily taken in $(0,\f12)$, and the corresponding H\"{o}lder semi-norm is uniformly bounded in $x$.
    Hence, $\lim_{t\to 0^+}\Psi_t(x)$ exists, denoted by $\Psi_0(x)$.
The convergence $\lim_{t\to 0^+}\Psi_t(x)= \Psi_0(x)$ is uniform in $x\in \BR$.
This further implies $x\mapsto \Psi_0(x)$ is continuous, non-decreasing, and $2\pi$-periodic in $\BR$.
Hence, $x\mapsto \Psi_0(x)$ is surjective from $\BR$ to itself.

\item
By the estimates in Proposition \ref{prop: boundedness and decay of H^k norms of h}, for each $x\in \BR$, as $t\to +\infty$, $\Psi_t(x)$ converges to some $\Psi_\infty(x)$ exponentially.
Hence, the function $t\mapsto \Psi_{t}(x)$ is H\"{o}lder continuous on $[0,+\infty)$, with the H\"{o}lder norm being uniformly bounded in $x$.
In addition, $\osc_{t\geq 0} \Psi_t(x) := \max_{t\geq 0}\Psi_t(x) - \min_{t\geq 0}\Psi_t(x)$ is uniformly bounded for all $x\in \BR$.

\item Thanks to the $\tilde{F}$-equation, for any $t>0$, $\tilde{F}(\cdot,t)$ is the push forward of $\tilde{F}(\cdot,1)$ under the map $x\mapsto\Psi_t(x)$, i.e., for any continuous function $\va(x)$ on $\BR$ with compact support,
\[
\int_\BR \va(x) \tilde{F}(x,t)\,dx = \int_\BR \va(\Psi_t(x))\tilde{F}(x,1)\,dx.
\]
Send $t\to 0^+$.
Because of \eqref{eqn: estimate for W_1 distance} and the uniform convergence from $\Psi_t$ to $\Psi_0$, we find
\[
\int_\BR \va(x) F_0(x)\,dx = \int_\BR \va(\Psi_0(x))\tilde{F}(x,1)\,dx,
\]
i.e, $F_0$ is the pushforward of $\tilde{F}(x,1)$ under the map $x\mapsto \Psi_0(x)$.
\item
The previous fact further implies $x\mapsto \Psi_0(x)$ is injective.
This is because otherwise, the facts that $x\mapsto \Psi_0(x)$ is continuous and non-decreasing, and that $\tilde{F}(x,1)$ has a positive lower bound (see \eqref{eqn: upper bound for f}) imply that $F_0$ must contain Dirac $\delta$-masses, which is a contradiction.
We can further show that $\Psi_0$ is a homeomorphism from $\BR$ to itself.
\end{itemize}

With the above $\{\Psi_t(x)\}_{t\geq 0}$, we define
$X(s,t): = \Psi_t \circ (\Psi_0)^{-1} \circ X_0(s)$.
We claim that this gives the desired function $X(s,t)$ in Corollary \ref{cor: rephrasing in Lagrangian}.

We first verify that, for any $s\in \BR$,
\beq
\int_{X(0,1)}^{X(s,1)} \tilde{F}(y,1)\,dy = s.
\label{eqn: relation between s and F}
\eeq
With $0<\epsilon\ll 1$, let $\chi_\epsilon(y)$ be a continuous function supported on $[X(0,1),X(s,1)]$, such that it is identically $1$ on $[X(0,1)+\epsilon,X(s,1)-\epsilon]$, and linear on $[X(0,1), X(0,1)+\epsilon]$ and $[X(s,1)-\epsilon, X(s,1)]$.
By the monotone convergence theorem,
\[
\begin{split}
\int_{X(0,1)}^{X(s,1)} \tilde{F}(y,1)\,dy
= &\;
\lim_{\epsilon \to 0^+} \int_{\BR} \chi_\epsilon(y)\tilde{F}(y,1)\,dy\\
= &\;
\lim_{\epsilon \to 0^+} \int_{\BR} \chi_\epsilon(\Psi_0^{-1}(y))F_0(y)\,dy
=
\int_{\BR} \chi_{[X(0,1),X(s,1)]}(\Psi_0^{-1}(y))F_0(y)\,dy
\end{split}
\]
By definition, $X(s,1) = (\Psi_0)^{-1}\circ X_0(s)$, and $F_0 = G_0'$ a.e., where $G_0$ is absolutely continuous on any compact interval of $\BR$ and is the inverse function of $X_0$.
So
\[
\int_{X(0,1)}^{X(s,1)} \tilde{F}(y,1)\,dy
=
\int_{\BR}\chi_{[X_0(0),X_0(s)]}(y) F_0(y)\,dy = G_0(X_0(s))-G_0(X_0(0)) = s.
\]
In general, for any $s_1,s_2\in \BR$,
\beq
\int_{X(s_1,1)}^{X(s_2,1)} \tilde{F}(y,1)\,dy
= s_2-s_1.
\label{eqn: relation between s and F more general}
\eeq

The equation \eqref{eqn: relation between s and F} implies that
\beqo
s\mapsto X(s,1)\mbox{ is the inverse function of the map }x\mapsto \int_{X(0,1)}^x \tilde{F}(y,1)\,dy.
\eeqo
By Theorem \ref{thm: main thm}, $\tilde{F}(\cdot,1)$ is smooth, positive, $2\pi$-periodic, and satisfies $\int_{-\pi}^\pi \tilde{F}(x,1)\,dx = 2\pi$.
Hence, $X(\cdot,1)$ is strictly increasing and smooth on $\BR$, with $s\mapsto X(s,1)-s$ being $2\pi$-periodic.
Then using the fact $X(s,t) = \Psi_t(X(s,1))$, \eqref{eqn: relation between s and F more general}, and the above-mentioned properties of $\{\Psi_t(x)\}_{t\geq 0}$, it is not difficult to show that $X(s,t)$ satisfies all the claims in Corollary \ref{cor: rephrasing in Lagrangian}; also see the derivation in Section \ref{sec: derivation of the model} and Remark \ref{rmk: energy class}.
In particular, we may take $c_\infty = \Psi_\infty\circ (\Psi_0)^{-1}\circ X_0(0)$, and $\f{1}{\lam(t)} = \|\tilde{F}(\cdot,t)\|_{L^\infty}$.
We omit the details.
\end{proof}

\subsection{Uniqueness of the dissipative weak solutions}
\label{sec: uniqueness}
As is mentioned before, to prove uniqueness, we need to focus on the dissipative weak solutions that satisfy more assumptions.

\begin{rmk}%[Entropy estimate for $F$]
\label{rmk: entropy estimate}
The inequality \eqref{eqn: entropy type estimate} in the definition of the dissipative weak solutions is motivated by the following calculation.
Suppose $f$ is a positive strong solution to \eqref{thm: main thm}.
Then $F := \f1f$ verifies
\[
\begin{split}
\f{d}{dt}\int_\BT \Phi(F(x,t))\,dx = &\; \int_\BT \Phi'(F)\pa_t F \,dx
= -\int_\BT \pa_x (\Phi'(F))\cdot \CH f\cdot  F\,dx\\
= &\; -\int_\BT F \Phi''(F) \pa_x F \cdot \CH f\,dx
=  -\int_\BT \pa_x \big(F \Phi'(F)-\Phi(F)\big) \cdot \CH f\,dx\\
= &\;- \int_\BT \big(\Phi(F)-F \Phi'(F)\big) (-\D)^{\f12} f\,dx.
\end{split}
\]
For the smooth positive solution $F$, taking $\Phi(y) = \f1y$ yields the energy estimate \eqref{eqn: energy estimate}.
Taking $\Phi(y) = y\ln y$ yields the entropy estimate \eqref{eqn: entropy estimate original}.
In general, when $\Phi\in C^1_{loc}((0,+\infty))$ is convex, $\Phi(y)-y\Phi'(y)$ is non-increasing in $y$.
Hence, using $F = \f{1}{f}$ and a derivation similar to the previous equation, we find
\[
\int_\BT \big(\Phi(F)-F \Phi'(F)\big) (-\D)^{\f12} f\,dx\geq 0.
\]
\end{rmk}

We first show that the dissipative weak solutions are bounded from above and from below if the initial data has such properties.
\begin{lem}
\label{lem: max principle of dissipative weak solution}
Under the assumption of Theorem \ref{thm: uniqueness}, let $f = f(x,t)$ be a dissipative weak solution to \eqref{eqn: the main equation} on $\BT\times [0,T)$.
Then for all $t\in (0,T)$,
\[
\|f(\cdot,t)\|_{L^\infty(\BT)}\leq \|f_0\|_{L^\infty(\BT)},\quad
\|F(\cdot,t)\|_{L^\infty(\BT)}\leq \|F_0\|_{L^\infty(\BT)}.
\]
As a result, $\pa_x f, \pa_x F \in L^2(\BT\times [0,T))$ with a priori estimates depending on $\|f_0\|_{L^\infty}$, $\|F_0\|_{L^\infty}$, and $\|\ln f_0\|_{\dot{H}^{\f12}}$.
\begin{proof}
%We only show the second inequality, as the first one can be proved analogously.
We choose in \eqref{eqn: entropy type estimate}
\[
\Phi(y) =
\begin{cases}
\big(y-\|f_0\|_{L^\infty}^{-1}\big)^2, & \mbox{if } y< \|f_0\|_{L^\infty}^{-1}, \\
0, & \mbox{if } y\in [\|f_0\|_{L^\infty}^{-1},\, \|F_0\|_{L^\infty}], \\
\big(y-\|F_0\|_{L^\infty}\big)^2, & \mbox{if } y> \|F_0\|_{L^\infty}.
\end{cases}
\]
Apparently, $\Phi\in C^1_{loc}((0,+\infty))$ is convex, so \eqref{eqn: entropy type estimate} gives that, for any $t\in (0,T)$,
\beqo
\int_\BT \Phi(F(x,t))\,dx + \int_0^t \int_\BT \big(\Phi(F)-F\Phi'(F)\big)(-\D)^{\f12}f\,dx\,dt
\leq 0.
\eeqo
The second term above is non-negative due to the convexity of $\Phi$ (see Remark \ref{rmk: entropy estimate}).
Hence, %we obtain that
\[
\int_\BT \Phi(F(x,t))\,dx\leq 0,
\]
which further implies $F(\cdot,t)\in [\|f_0\|_{L^\infty}^{-1}, \|F_0\|_{L^\infty}]$ a.e.\;on $\BT$.

The last claim follows from the first claim and the definition of the dissipative weak solution (also see \eqref{eqn: bounding H^1 norm of f using H^1 norm of sqrt f}).
\end{proof}
\end{lem}

Now we are ready to prove Theorem \ref{thm: uniqueness}.

\begin{proof}[Proof of Theorem \ref{thm: uniqueness}]
Under the assumptions on $f_0$ in Theorem \ref{thm: uniqueness}, we can construct a global dissipative weak solution by following the argument in Section \ref{sec: existence for general initial data}.
Indeed, denote $h_0 = \ln f_0$ and we first consider the solution $f_j$ to \eqref{eqn: the main equation} with the approximate initial data $f_{0,j} : = \exp(\mathscr{F}_j * h_0)$ that is positive and smooth.
Such an $f_j$ can be obtained by following the second step in the proof in Section \ref{sec: existence for general initial data}.
$f_j$ is smooth for all positive times;
one can verify that $f_j$ is a dissipative weak solution.
In particular, \eqref{eqn: entropy type estimate} follows from the convexity of $\Phi$ and the weak $L^p$-convergence $(p\in [1,+\infty))$ of $F_{j}(\cdot,t)$ to $F_{0,j}$ as $t\to 0^+$.
Then we take the limit $j\to +\infty$ to obtain the global dissipative weak solution $f$ corresponding to the given initial data $f_0$.

Next, we focus on the uniqueness.
With abuse of notations, suppose $f_1$ and $f_2$ are two dissipative weak solutions to \eqref{eqn: the main equation} on $\BT\times [0,T)$ for some $T>0$, both starting from the initial data $f_0$.
We want to show $f_1 = f_2$ in $\BT\times [0,T)$.
Without loss of generality, we may assume $f_2$ to be the global dissipative weak solution constructed above. % Theorem \ref{thm: main thm}.
Recall that $f_2$ is a positive strong solution to \eqref{eqn: the main equation} at all positive times.
Let $F_i : = \f{1}{f_i}$ $(i = 1,2)$.
%For any $t_0>0$, $F_2$ is a positive strong solution on $\BT\times [t_0,+\infty)$.
%For any $t_0>0$, $F_2$ is positive and smooth in $\BT\times [t_0,+\infty)$.
In what follows, the proof of $f_1 = f_2$ relies on a relative entropy estimate between $F_1$ and $F_2$.

Given $t_*\in (0,T)$ and $0<\d\ll 1$, let $\r_\d(t)$ to be a smooth cutoff function on $[0,+\infty)$, such that $\r_\d(t) \equiv 0$ on $[0,\f{\d}{2}]\cup [t_*-\f{\d}{2},+\infty)$, $\r_\d(t) \equiv 1$ on $[\d,t_*-\d]$, and in addition, $\r_\d$ is increasing on $[\f{\d}{2},\d]$ but decreasing on $[t_*-\d, t_*-\f{\d}2]$.
Since $F_1$ is a weak solution to \eqref{eqn: equation for big F}, we may take $\va = \r_\d(t) \ln F_2$ as the test function in \eqref{eqn: weak formulation} and derive that
\[
\int_{0}^{t_*} \r_\d'(t)\int_\BT F_1\ln F_2 \,dx\,dt
= \int_0^{t_*} \r_\d(t) \int_{\BT} \pa_x \ln F_2 \cdot \CH f_1 \cdot F_1 - F_1 \pa_t \ln F_2 \,dx\,dt.
\]
The right-hand side can be further simplified thanks to the smoothness of $F_2$
\beqo
\begin{split}
\int_{0}^{t_*} \r_\d'(t)\int_\BT F_1\ln F_2 \,dx\,dt
= &\; \int_0^{t_*} \r_\d(t) \int_{\BT} \f{F_1}{F_2} \big(\pa_x F_2 \cdot \CH f_1 -  \pa_x (\CH f_2\cdot F_2)\big) \, dx\,dt\\
= &\; \int_0^{t_*} \r_\d(t) \int_{\BT} \f{F_1}{F_2} \cdot \pa_x F_2 \cdot \CH (f_1-f_2) -  F_1 (-\D)^{\f12} f_2 \, dx\,dt.
\end{split}
\eeqo
Next we will send $\d \to 0^+$.
Recall that Lemma \ref{lem: max principle of dissipative weak solution} gives $f_i, F_i\in L^\infty(\BT\times [0,T))$ and $\pa_x f_i, \pa_x F_i\in L^2(\BT\times [0,T))$. %, all with a priori estimates depending only on $\|f_0\|_{L^\infty}$, $\|F_0\|_{L^\infty}$, and $\|\ln f_0\|_{\dot{H}^{\f12}}$.
This will be enough to pass to the limit on the right-hand side.
To handle the left-hand side, we want to show $\int_\BT F_1 \ln F_2\,dx$ is continuous on $[0,t_*]$.
In fact, by the definition of dissipative weak solution, $\ln F_i\in L^\infty([0,T);H^{\f12}(\BT))$.
The upper and lower bounds for $F_i$ implies $F_i\in L^\infty([0,T);L^\infty\cap H^{\f12}(\BT))$ (cf.\;\eqref{eqn: equivalence of H 1/2 norm of f and h}).
This together with \eqref{eqn: equation for big F} and $f_i\in L^2([0,T);H^1(\BT))$ implies $\pa_t F_i\in L^2([0,T); H^{-\f12}(\BT))$, which in turn gives e.g.\;$F_i\in C([0,T);L^2(\BT))$.
Then we obtain the time-continuity of $\int_\BT F_1 \ln F_2\,dx$.
Therefore, taking $\d \to 0^+$, we obtain
\beq
\begin{split}
&\;\int_\BT F_0\ln F_0 \,dx - \int_\BT F_1(x,t_*)\ln F_2(x,t_*) \,dx\\
= &\; \int_0^{t_*} \int_{\BT} \f{\pa_x F_2}{F_2} \cdot  F_1\cdot \CH (f_1-f_2) -  F_1 (-\D)^{\f12} f_2 \, dx\,dt.
\end{split}
\label{eqn: equation derived from the weak formulation}
\eeq

On the other hand, by the definition of the dissipative weak solution, with $\Phi(y) = y\ln y$ in \eqref{eqn: entropy type estimate},
\[
\int_\BT F_1(x,t_*) \ln F_1(x,t_*)\,dx - \int_\BT F_0 \ln F_0\,dx
\leq \int_0^{t_*} \int_\BT F_1 (-\D)^{\f12} f_1\,dx\,dt.
\]
Adding this and \eqref{eqn: equation derived from the weak formulation} yields an estimate for the relative entropy between $F_1$ and $F_2$
\beq
\begin{split}
&\;\int_\BT F_1(x,t_*) \ln \left(\f{F_1(x,t_*)}{F_2(x,t_*)}\right)dx\\
\leq &\; \int_0^{t_*}\int_\BT (f_1-f_2)(-\D)^{\f12}(F_1-F_2)\, dx\,dt
+\int_0^{t_*}\int_\BT \CH (f_1-f_2)\cdot (F_1-F_2) \f{\pa_x F_2}{F_2}\,dx\,dt.
\end{split}
\label{eqn: time derivative of relative entropy}
\eeq
We then derive that % (cf.\;\cite[\S\,]{caffarelli2013regularity})
\[
\begin{split}
&\;\int_\BT (f_1-f_2) (-\D)^{\f12} (F_1-F_2)\,dx\\%
%= &\; \int_\BT (f_1-f_2)(x)\cdot \f{1}{\pi}\pv \int_\BT \f{(F_1-F_2)(x)-(F_1-F_2)(y)}{4\sin^2(\f{x-y}{2})}\,dy\,dx\\
= &\; \f{1}{2\pi} \int_{\BT\times \BT } \big[(f_1-f_2)(x)-(f_1-f_2)(y)\big] \cdot \f{(F_1-F_2)(x)-(F_1-F_2)(y)} {4\sin^2(\f{x-y}{2})}\,dy\,dx\\
%
%= &\; -\f{1}{2\pi} \int_{\BT\times \BT } \big[f_1f_2(x)(F_1-F_2)(x)-f_1f_2(y)(F_1-F_2)(y)\big] \cdot \f{(F_1-F_2)(x)-(F_1-F_2)(y)} {4\sin^2(\f{x-y}{2})}\,dy\,dx\\
= &\; -\f{1}{2\pi} \int_{\BT\times \BT } f_1f_2(x)\cdot \f{[(F_1-F_2)(x)-(F_1-F_2)(y)]^2} {4\sin^2(\f{x-y}{2})}\,dy\,dx\\
&\; -\f{1}{2\pi} \int_{\BT\times \BT }  (F_1-F_2)(y)\big[f_1f_2(x)-f_1f_2(y)\big] \cdot \f{(F_1-F_2)(x)-(F_1-F_2)(y)} {4\sin^2(\f{x-y}{2})}\,dy\,dx\\
\leq &\; -\f{1}{2\pi} \|F_0\|_{L^\infty}^{-2} \int_{\BT\times \BT } \f{[(F_1-F_2)(x)-(F_1-F_2)(y)]^2} {4\sin^2(\f{x-y}{2})}\,dy\,dx\\
&\; -\f{1}{4\pi} \int_{\BT\times \BT }  \big[f_1f_2(x)-f_1f_2(y)\big] \cdot \f{(F_1-F_2)^2(x)-(F_1-F_2)^2(y)} {4\sin^2(\f{x-y}{2})}\,dy\,dx.
\end{split}
\]
To obtain the last term above, we exchanged the $x$- and $y$-variables in the integral.
Then \eqref{eqn: time derivative of relative entropy} becomes
\beqo
\begin{split}
&\;\int_\BT F_1(x,t_*) \ln \left(\f{F_1(x,t_*)}{F_2(x,t_*)}\right)dx
+\|F_0\|_{L^\infty}^{-2}
\|(F_1-F_2)(\cdot,t)\|_{L^2_{t_*}\dot{H}^{\f12}}^2 \\
\leq
&\; -\f12 \int_0^{t_*}\int_{\BT\times \BT }  (F_1-F_2)^2 (-\D)^{\f12}(f_1f_2)\,dx\,dt
-2\int_0^{t_*}\int_\BT \CH (f_1-f_2)\cdot (F_1-F_2) F_2^{\f12}\pa_x \sqrt{f_2}\,dx\,dt\\
\leq &\; \f12  \|F_1-F_2\|_{L^4_{t_*} L^4}^2 \|f_1f_2\|_{L^2_{t_*}\dot{H}^1}
+C\|f_1-f_2\|_{L^4_{t_*} L^4} \|F_1-F_2\|_{L^4_{t_*} L^4} \|F_2\|_{L^\infty_{t_*}  L^\infty}^{\f12}\big\|\pa_x\sqrt{f_2}\big\|_{L^2_{t_*} L^2} \\
%
%\leq &\;C\|F_1-F_2\|_{L^4}^2 \|f_0\|_{L^\infty}^2 \|F_0\|_{L^\infty}^{\f12} \left(\big\|\pa_x\sqrt{f_1}\big\|_{L^2} +\big\|\pa_x\sqrt{f_2}\big\|_{L^2}\right)\\
\leq &\;C \|f_0\|_{L^\infty}^2 \|F_0\|_{L^\infty}^{\f12} \left(\big\|\pa_x\sqrt{f_1}\big\|_{L^2_{t_*} L^2} +\big\|\pa_x\sqrt{f_2}\big\|_{L^2_{t_*} L^2}\right) \|F_1-F_2\|_{L^\infty_{t_*} L^2}\|F_1-F_2\|_{L_{t_*}^2\dot{H}^{\f12}}.
%+C\|F_1-F_2\|_{L^4}^2 \|f_0\|_{L^\infty}^2 \|F_0\|_{L^\infty}^{\f12}\big\|\pa_x\sqrt{f_2}\big\|_{L^2} \\
\end{split}
\eeqo
Here $L_{t_*}^p$ denotes the $L^p$-norm on the time interval $[0, t_*]$.
Applying Young's inequality, we find that %for any $t_*\in (0,T)$,
\beq
\begin{split}
&\;\left\|\int_\BT F_1 \ln \left(\f{F_1}{F_2}\right)dx\right\|_{L^\infty_{t_*}}
+\|F_0\|_{L^\infty}^{-2} \|F_1-F_2\|_{L_{t_*}^2\dot{H}^{\f12}}^2\\
\leq &\;C\big(\|f_0\|_{L^\infty},\|F_0\|_{L^\infty}\big) \|F_1-F_2\|_{L_{t_*}^\infty L^2}^2 \int_0^{t_*}\big\|\pa_x\sqrt{f_1}\big\|_{L^2}^2 +\big\|\pa_x\sqrt{f_2}\big\|_{L^2}^2 \,dt.
\end{split}
\label{eqn: estimate for relative entropy}
\eeq
%Here $L_T^p: = L^p([0,T])$ for $p\in [1,+\infty]$.

Now we need a variation of Pinsker's inequality for the relative entropy, under the condition that $F_1$ and $F_2$ with identical $L^1$-norms are both positive, bounded, and bounded away from zero.
The following argument is adapted from that in \cite{kemperman1969optimum} (also see \cite{pollard2022}).
Let $R:= \f{F_1}{F_2}-1$.
Observe that $(1+x)\ln(1+x)-x\geq Cx^2$ on any bounded interval in $(-1,+\infty)$, where $C$ depends on the upper bound of the interval.
Since $\|f_0\|_{L^\infty}^{-1}\leq F_i\leq \|F_0\|_{L^\infty}$ $(i = 1,2)$, we derive that
\[
\begin{split}
\int_\BT F_1 \ln \left(\f{F_1}{F_2}\right)dx
= &\; \int_\BT F_2(1+R) \ln(1+R) - F_2 R\, dx \geq C\int_\BT F_2 R^2\, dx\\
= &\;C\int_\BT F_2^{-1} (F_1-F_2)^2\, dx
\geq C\|F_1-F_2\|_{L^2}^2,
\end{split}
\]
where $C$ essentially depends
%on the upper and lower bounds of $F_i$, and thus
on $\|f_0\|_{L^\infty}$ and $\|F_0\|_{L^\infty}$.

Applying this to \eqref{eqn: estimate for relative entropy} yields
\beqo
\begin{split}
&\;\left\|\int_\BT F_1 \ln \left(\f{F_1}{F_2}\right)dx\right\|_{L^\infty_{t_*}}
+\|F_0\|_{L^\infty}^{-2} \|F_1-F_2\|_{L_{t_*}^2\dot{H}^{\f12}}^2\\
\leq &\;C\big(\|f_0\|_{L^\infty},\|F_0\|_{L^\infty}\big) \left\|\int_\BT F_1 \ln \left(\f{F_1}{F_2}\right)dx\right\|_{L^\infty_{t_*}}  \int_0^{t_*}\big\|\pa_x\sqrt{f_1}\big\|_{L^2}^2 +\big\|\pa_x\sqrt{f_2}\big\|_{L^2}^2 \,dt.
\end{split}
\eeqo
Since $\|\pa_x \sqrt{f_i}\|_{L^2([0,T);L^2(\BT))}<+\infty$, we take $t_*$ to be sufficiently small to conclude
\[
\int_\BT F_1 \ln \left(\f{F_1}{F_2}\right)dx = 0\mbox{ on }[0,t_*],
\]
and therefore, $F_1 = F_2$ on $[0,t_*]$.
Repeating this argument for finitely many times yields the uniqueness on the entire time interval $[0,T)$.

Therefore, we conclude the uniqueness of the dissipative weak solutions on $\BT\times [0,+\infty)$.
\end{proof}

\section{Analyticity and Sharp Decay Estimates for Large Times}
\label{sec: analyticity and sharp decay estimate}

Let us introduce some notations.
Define the Fourier transform of $f\in L^1(\BT)$ as in \eqref{eqn: Fourier transform}.
For $m\in \BN$, let
\[
%\|f\|_{\dot{\CF}^{0,1}}: = \sum_{k\in \BZ} |\hat{f}_k|, \quad
\|f\|_{\dot{\CF}^{m,1}}: = \sum_{k\neq 0} |k|^m |\hat{f}_k|.
\]
With $\nu\geq 0$, we additionally define
\[
%\|f\|_{\CF_\nu^{0,1}}: = \sum_{k\in \BZ} e^{\nu |k|} |\hat{f}_k|,\quad
%\|f\|_{\dot{\CF}_\nu^{0,1}}: = \sum_{k\neq 0} e^{\nu |k|} |\hat{f}_k|. ,\quad
\|f\|_{\dot{\CF}_\nu^{m,1}}: = \sum_{k\neq 0} e^{\nu |k|} |k|^m |\hat{f}_k|.
\]
Clearly, $\|f\|_{\dot{\CF}_0^{m,1}} = \|f\|_{\dot{\CF}^{m,1}}$.

If $\|f_0\|_{\dot{\CF}^{0,1}}$ is suitably small compared to $f_{\infty}$, we can prove that the solution will become analytic in space-time instantly after the initial time.

\begin{prop}\label{prop: instant analyticity for small F01 data}
There exists a universal number $\epsilon>0$ such that the following holds.
Suppose $f_0$ in \eqref{eqn: the main equation} is positive and smooth on $\BT$,  such that $\|f_0\|_{\dot{\CF}^{0,1}}\leq \epsilon f_\infty$.
Let $f$ be the unique strong solution to \eqref{eqn: the main equation}.
Then for all $t_0>0$, $f$ is analytic in the space-time domain $\BT\times (t_0,+\infty)$.
In particular, for any $t\geq 0$,
\beq
\|f(\cdot,t)\|_{\dot{\CF}_{\nu(t)}^{0,1}} \leq 2 \|f_0\|_{\dot{\CF}^{0,1}}.
\label{eqn: analyticity norm}
\eeq
Here
\beqo
\nu(t): = \f12 \ln \big[\th + (1-\th) \exp(2f_\infty t)\big],\quad
\th := C_* f_\infty^{-1}\|f_0\|_{\dot{\CF}^{0,1}},
%\label{eqn: analyticity rate}
\eeqo
with $C_*$ being a universal constant.
In fact, it would suffice to have required $\epsilon < C_*^{-1}$ above.

\begin{rmk}
%The uniqueness of the strong solution follows from Theorem \ref{thm: uniqueness}.
Note that $\nu(t)\sim f_\infty t$ for large $t$ as long as $\epsilon < C_*^{-1}$.
In the view of the linearization of \eqref{eqn: the main equation} around the equilibrium $f_\infty$, i.e.,
$\pa_t f \approx -f_\infty (-\D)^{\f12}f$,
the estimate \eqref{eqn: analyticity norm} provides the sharp rate of analyticity and decay.
\end{rmk}

\begin{proof}
First suppose $f_0$ is band-limited (i.e., only finitely many Fourier coefficients of $f_0$ are non-zero), and $f_0(x)\not \equiv f_\infty$.
It has been shown in Proposition \ref{prop: band-limited initial data} that there is a unique band-limited global solution $f(\cdot,t)$, which is automatically analytic.
In the view of Theorem \ref{thm: uniqueness}, it is also the unique strong solution starting from $f_0$.

To be more quantitative, we recall \eqref{eqn: equation in Fourier space}:
for all $k \geq 0$,
\beqo
\f{d}{dt} \hat{f}_k
= - k \bar{f}(t)\hat{f}_{k}
- \sum_{j\geq 1} 2(k+2j) \hat{f}_{k+j}\overline{\hat{f}_j}
=: - k \bar{f}(t)\hat{f}_{k}
- \hat{N}_k.
\eeqo
The case $k = 0$ corresponds to the energy estimate \eqref{eqn: energy estimate} in Lemma \ref{lem: energy bound and decay of L^p norms},
while for all $k>0$,
\[
\f{d}{dt}\big|\hat{f}_k\big|
= \f1{2 |\hat{f}_k|} \left(\hat{f}_k \f{d}{dt}\overline{\hat{f}_k}+ \overline{\hat{f}_k} \f{d}{dt}\hat{f}_k\right)
\leq -|k|\bar{f}(t)\big|\hat{f}_k\big| + \big|\hat{N}_k\big|.
\]
This derivation is valid for $|\hat{f}_k|>0$, but it is not difficult to see the resulting inequality also holds when $|\hat{f}_k| = 0$.
A similar inequality holds for negative $k$'s.
Summing over $k\neq 0$, with $\nu(t)\geq 0$ to be determined, we derive that
\beq
\begin{split}
&\;\f{d}{dt}\|f(\cdot,t)\|_{\dot{\CF}_{\nu(t)}^{0,1}}
+ \big(\bar{f}(t) - \nu'(t)\big)\|f(\cdot,t)\|_{\dot{\CF}_{\nu(t)}^{1,1}}
\\
\leq &\; C\sum_{k\geq 1} e^{k \nu(t)} \sum_{j\geq 1} (k+2j) \big|\hat{f}_{k+j}(t)\big|\big|\hat{f}_j(t)\big|\\
\leq &\; C \sum_{j\geq 1} e^{-j \nu(t)}\big|\hat{f}_j(t)\big| \sum_{k\geq 1} (k+j) e^{(k+j) \nu(t)}
\big|\hat{f}_{k+j}(t)\big| \\
%\leq &\; C_* \|f(\cdot,t)\|_{\dot{\CF}_{\nu(t)}^{1,1}} \sum_{j\geq 1} e^{-j \nu(t)}\big|\hat{f}_j(t)\big|\\
\leq &\; \f12 C_* e^{- 2\nu(t)}\|f(\cdot,t)\|_{\dot{\CF}_{\nu(t)}^{0,1}} \|f(\cdot,t)\|_{\dot{\CF}_{\nu(t)}^{1,1}},
\end{split}
\label{eqn: inequality for time varying analyticity norm}
\eeq
where $C_*>0$ is a universal constant.
%In the second inequality, we use the fact $j\geq 1$.
Note that, because $f$ is a band-limited strong solution, $\|f(\cdot,t)\|_{\dot{\CF}_{\nu(t)}^{0,1}}$ is finite at all time and it is smooth in $t$ as long as $\nu(t)$ is finite and smooth.

Now let $\nu(t)$ solve
\beq
f_\infty - \nu'(t)
=
C_* e^{- 2\nu(t)}\|f_0\|_{\dot{\CF}^{0,1}},\quad \nu(0) = 0.
\label{eqn: condition on nu(t)}
\eeq
%for all time.
Denote $\th := C_* f_\infty^{-1}\|f_0\|_{\dot{\CF}^{0,1}}$. Then $\mu(t) := e^{2\nu(t)}$ solves
%Then if we replace the inequality in \eqref{eqn: condition on nu(t)} by equality, we would obtain
\[
\mu'(t) = 2f_\infty  \big(\mu(t) - \th\big),\quad \mu(0)  = 1.
\]
As long as $\th<1$, or equivalently $\|f_0\|_{\dot{\CF}^{0,1}} < C_*^{-1} f_\infty$, this equation has a unique positive solution $\mu(t) = \th + (1-\th) \exp(2f_\infty t)$ on $[0,+\infty)$.
Then $\nu(t): = \f12 \ln \mu(t)$ satisfies \eqref{eqn: condition on nu(t)}.

With this $\nu(t)$, we claim that %for all $t\geq 0$,
\beq
\|f(\cdot,t)\|_{\dot{\CF}_{\nu(t)}^{0,1}} \leq 2 \|f_0\|_{\dot{\CF}^{0,1}}\quad \forall\, t\geq 0,
\label{eqn: bound for time varying analyticity norm}
\eeq
which further implies spatial analyticity of $f(\cdot,t)$ for all $t>0$.
Indeed, if not, we let $t_*$ denote the infimum of all times where \eqref{eqn: bound for time varying analyticity norm} does not hold.
By the time-continuity of $\|f(\cdot,t)\|_{\dot{\CF}_{\nu(t)}^{0,1}}$, $t_*$ is positive, and the equality in \eqref{eqn: bound for time varying analyticity norm} is achieved at time $t_*$.
This together with \eqref{eqn: condition on nu(t)} and the fact $\bar{f}(t)\geq f_\infty$ implies
\beqo
\bar{f}(t) - \nu'(t)
\geq \f12 C_* e^{- 2\nu(t)}\|f(\cdot,t)\|_{\dot{\CF}_{\nu(t)}^{0,1}}\quad \forall\, t\in [0,t_*].
\eeqo
%for all $t\in [0,t_*]$.
Combining this with \eqref{eqn: inequality for time varying analyticity norm}, we find
%$\f{d}{dt}\|f(\cdot,t)-f_\infty\|_{\CF_{\nu(t)}^{0,1}} \leq 0$.
%This further implies
\[
\|f(\cdot,t_*)\|_{\dot{\CF}_{\nu(t_*)}^{0,1}}
\leq
\|f_0\|_{\dot{\CF}^{0,1}}
< 2\|f_0\|_{\dot{\CF}^{0,1}},
\]
which contradicts with the definition of $t_*$.

For smooth $f_0$ that is not band-limited, we let $f$ be the unique global solution corresponding to the initial data $f_0$, which is constructed in Section \ref{sec: existence for general initial data}.
Recall the Fej\'{e}r kernel $\mathscr{F}_n$ was defined in \eqref{eqn: Fejer kernel}.
Let $f_n$ be the unique band-limited global solution corresponding to the band-limited initial data $\mathscr{F}_n*f_0$.
Then with $\nu(t)$ defined above, for arbitrary $m\in \BZ_+$ and all $t\geq 0$,
\beqo
\sum_{j\neq 0} e^{\nu(t) \min\{|j|,m\}} |\CF(f_n)_j(t)|
\leq
\|f_n(\cdot,t)\|_{\dot{\CF}_{\nu(t)}^{0,1}} < 2 \|\mathscr{F}_n * f_0\|_{\dot{\CF}^{0,1}}.
\eeqo
Similar to Section \ref{sec: existence for general initial data}, we can prove that, as $n\to +\infty$, $f_n\to f$ in $C_{loc}^k(\BT\times (0,+\infty))$ for any $k\in \BN$.
In fact, we first obtain convergence along a subsequence, and then show the convergence should hold for the whole sequence due to the uniqueness of the solution $f$.
Hence, sending $n\to +\infty$
and then letting $m\to +\infty$ yields \eqref{eqn: bound for time varying analyticity norm} again in this general case.

The space-time analyticity follows from the spatial analyticity of $f$ and the equation \eqref{eqn: the main equation}.
\end{proof}
\end{prop}

\begin{cor}[Property \eqref{property: large-time analyticity} of the weak solution constructed in Theorem \ref{thm: main thm}]
\label{cor: analyticity of the solution}
Under the assumptions of Theorem \ref{thm: main thm}, there exists $T_* > 0$ depending only on $\|f_0\|_{L^1}$ and $\|F_0\|_{L^1}$, such that the constructed solution $f$ of \eqref{eqn: the main equation} is analytic in the space-time domain $\BT\times (T_*,+\infty)$.
It satisfies that, for any $t>0$,
\beqo
\|f(\cdot,t+T_*)\|_{\dot{\CF}_{\nu(t)}^{0,1}} \leq C f_\infty,
\eeqo
with
\[
\nu(t)\geq \f12 \ln \left[1 + \exp(2f_\infty t)\right] - \f12 \ln 2,
\]
and $C>0$ being a universal constant.
\begin{proof}
Using the notations in Proposition \ref{prop: instant analyticity for small F01 data}, we take $\epsilon$ there to satisfy $C_*\epsilon \leq \f12$.
By the property \eqref{property: long-time convergence} of the solution (also see Proposition \ref{prop: boundedness and decay of H^k norms of h}), there is $T_*>0$ depending only on $\|f_0\|_{L^1}$ and $\|F_0\|_{L^1}$, such that %we have
\[
\|f(\cdot,T_*)\|_{\dot{\CF}^{0,1}(\BT)} \leq C\|f(\cdot,T_*)\|_{\dot{H}^2} \leq \epsilon f_\infty.
\]
Here we used the embedding $H^2(\BT)\hookrightarrow \CF^{0,1}(\BT)$. Then we apply Proposition \ref{prop: instant analyticity for small F01 data} to conclude.
In particular, since $\th \leq \f12$ in this case, the desired estimate for $\nu(t)$ follows.
\end{proof}
\end{cor}

\appendix

\section{An $H^{-\f12}$-estimate for $F$}
\label{sec: H-1/2 estimate for F}
It is worth mentioning another interesting estimate for $F$, although it is not used in the other parts of the paper.
For convenience, we assume $f$ to be a positive strong solution to \eqref{eqn: the main equation} on $\BT\times [0,T)$.
Define
\[
\bar{F} : = \f{1}{2\pi}\int_\BT F(x,t)\,dx.
\]
By \eqref{eqn: equilibrium} and Lemma \ref{lem: energy bound and decay of L^p norms}, $\bar{F} = \f{1}{f_\infty}$, which is a time-invariant constant.
Then we have the following estimate.

%\textcolor{red}{We need to be careful about the following derivation, since negative powers of $(-\D)$ can only be applied to mean zero quantities.}

\begin{lem}
\label{lem: H-1/2 estimate of F}
\beqo
\f{d}{dt}\big\|(-\D)^{-\f14} (F-\bar{F}) \big\|_{L^2}^2
+\int_\BT f \cdot (\CH F)^2 \,dx  = 2\pi \bar{F} - \bar{F}^2\int_\BT f\,dx\leq 0.
\eeqo

\begin{proof}

By the Cotlar's identity \eqref{eqn: Cotlar} for mean-zero functions
\[
(\CH F)^2 - \big(F-\bar{F}\big)^2 = 2\CH\big((F-\bar{F})\CH F\big).
\]
Hence,
\beqo
(\CH F)^2 - F^2 -\bar{F}^2 + 2F\bar{F} = 2\CH (F\CH F) - 2\bar{F}\CH\CH F
= 2\CH (F\CH F) + 2\bar{F} (F-\bar{F}),
%\label{eqn: Cotlar for general functions}
\eeqo
which simplifies to 
\beq
(\CH F)^2 - F^2 +\bar{F}^2
= 2\CH (F\CH F).
\label{eqn: Cotlar for general functions}
\eeq
Using this identity and the $F$-equation, we find that
\beqo
\begin{split}
&\; \f{d}{dt}\big\|(-\D)^{-\f14} (F-\bar{F}) \big\|_{L^2}^2\\
= &\; 2\int_\BT \pa_t F\cdot (-\D)^{-\f12} (F-\bar{F})\,dx \\
= &\;- 2\int_\BT  \CH f \cdot F \cdot \pa_x (-\D)^{-\f12}(F-\bar{F})\,dx
= - 2\int_\BT f \cdot \CH\big( F \cdot \CH F\big) \,dx \\
= &\; - \int_\BT f \cdot \big( (\CH F)^2 - F^2+\bar{F}^2\big) \,dx \\
= &\; - \int_\BT f \cdot (\CH F)^2 \,dx  + \int_\BT F\,dx - \bar{F}^2 \int_\BT f\,dx.
\end{split}
%\label{eqn: derivation for H_minus norm of F}
\eeqo
Then the desired estimate follows by rearranging the equation and applying the Cauchy-Schwarz inequality.
\end{proof}
\end{lem}

Under the additional assumption $F_0-\bar{F}\in H^{-\f12}(\BT)$, Lemma \ref{lem: H-1/2 estimate of F} implies $F-\bar{F}\in L^\infty_{T} H^{-\f12}$ and $f^{\f12}\CH F\in L^2_{T} L^2$.
Interestingly, this provides another way of making sense of the nonlinear term
\[
\int_{\BT\times [0,T)} \pa_x \va \cdot \CH f \cdot F\,dx\,dt
\]
in the weak formulation \eqref{eqn: weak formulation} (cf.\;Section \ref{sec: estimates for h}).
Note that the additional assumption $F_0-\bar{F}\in H^{-\f12}$ has the same scaling as the assumption $F_0\in L^1$ that we have been using.

Indeed, by Cotlar's identity \eqref{eqn: Cotlar for general functions} and the Cauchy-Schwarz inequality, % we can derive that %(also see Remark \ref{rmk: an identity})
\[
\begin{split}
\int_\BT F\cdot (\CH f)^2\,dx
= &\;\int_\BT F\cdot \big( (\CH f)^2-f^2 + \bar{f}^2\big)\,dx + \int_\BT f\,dx - \bar{f}^2 \int_\BT F\,dx\\
\leq &\;\int_\BT  F\cdot 2\CH ( f \CH f)\,dx 
= - 2 \int_\BT  \CH F\cdot  f \CH f \,dx\\
\leq &\; 2 \big\|f^{\f12}\CH F\big\|_{L^2} \big\|f^{\f12}\big\|_{L^6}\big\|\CH f\big\|_{L^3}  \\
\leq &\; C \big\|f^{\f12}\CH F\big\|_{L^2} \|f\|_{L^3}^{\f32}.
\end{split}
\]
By the interpolation inequality \cite{kozono2008remarks},
$\|f\|_{L^3}^{3/2} \leq C \|f\|_{L^1}^{1/2} \|f\|_{\dot{H}^\f12}$.
Hence, by Lemma \ref{lem: energy bound and decay of L^p norms}, Lemma \ref{lem: H-1/2 estimate of F}, and the assumption $F_0-\bar{F}\in H^{-\f12}$, we obtain that $F^{\f12}\CH f \in L_{T}^2L^2$.

Then the facts $F^{\f12}\in L^\infty_T L^2$ and $F^{\f12}\CH f\in L_{T}^2 L^2$ allow us to well define the above nonlinear term in the weak formulation.
Alternatively, the facts $F-\bar{F}\in L_{T}^\infty H^{-\f12}$ and $\CH f\in L_T^2 H^{\f12}$ (see Lemma \ref{lem: energy bound and decay of L^p norms}) also suffice.
In either case, one can show that
\[
\int_0^t \int_{\BT} \pa_x \va \cdot \CH f \cdot F\,dx\,dt = O(t^{1/2}),
\]
which may lead to a simpler proof of existence of the weak solutions.

%\bibliographystyle{unsrt}
%\bibliography{tangential_model}

\end{document}